\documentclass[11pt]{amsart}
\usepackage{amsfonts, pdfsync}
\usepackage{graphicx, comment}
\usepackage{times}
\usepackage{amsmath}
\usepackage{amssymb}
\usepackage{latexsym}
\usepackage{amsthm}
\usepackage{enumitem}
\usepackage[usenames,dvipsnames]{xcolor}
 \definecolor{blue(ryb)}{rgb}{0.01, 0.28, 1.0}
 \definecolor{bleudefrance}{rgb}{0.19, 0.55, 0.91}
  \definecolor{azure(colorwheel)}{rgb}{0.0, 0.5, 1.0}

\usepackage[bookmarks, breaklinks, colorlinks,  linktocpage=true, linkcolor=red,anchorcolor=Periwinkle,citecolor=bleudefrance,urlcolor=Emerald, pdftitle={Strongly Compact Strong Trajectory attractors for Evolutionary systems and Their Applications},
  pdfsubject={Mathematics},
  pdfauthor={Songsong Lu},
  pdfkeywords={trajectory attractor, global attractor, evolutionary system, Navier-Stokes equations, reaction-diffusion system},
 ]{hyperref}

\allowdisplaybreaks[4]
\newtheorem{theorem}{Theorem}[section]
\newtheorem{lemma}[theorem]{Lemma}

\newtheorem{definition}[theorem]{Definition}
\newtheorem{corollary}[theorem]{Corollary}
\newtheorem{remark}[theorem]{Remark}
\newtheorem{open}[theorem]{Open Problem}
\newtheorem{mtheorem}{Main Theorem}
\newcommand{\As}{\mathcal{A}_{\mathrm{s}}}
\newcommand{\Aw}{\mathcal{A}_{\mathrm{w}}}

\newcommand{\ds}{\mathrm{d}_{\mathrm{s}}}
\newcommand{\dw}{\mathrm{d}_{\mathrm{w}}}

\newcommand{\Dc}{\mathcal{E}}
\newcommand{\s}{\mathrm{s}}
\newcommand{\w}{\mathrm{w}}
\newcommand{\Xw}{X_{\mathrm{w}}}
\newcommand{\Xs}{X_{\mathrm{s}}}
\newcommand{\Hw}{H_{\mathrm{w}}}
\newcommand{\Ab}{\mathcal{A}_{\bullet}}
\newcommand{\db}{\mathrm{d}_{\bullet}}
\newcommand{\dd}{\mathrm{d}}

\newcommand{\wb}{{\omega}_{\bullet}}
\newcommand{\ww}{{\omega}_{\mathrm{w}}}
\newcommand{\ws}{{\omega}_{\mathrm{s}}}

\newcommand{\ddt}{\frac{d}{dt}}

\newcommand{\Rc}{\widetilde{R}}

\def\setmarsing{
\oddsidemargin-0in
\evensidemargin-0in
\textwidth5.5in
\textheight8in
}
\setmarsing
\begin{document}

\title
[Strongly Compact Strong Trajectory Attractors]
{Strongly Compact Strong Trajectory Attractors for Evolutionary Systems and Their Applications}
\author{Songsong Lu}
\thanks{S.L. was partially supported
by Specialized Research Fund for the Doctoral Program of Higher Education (200805581025), NSFC 11001279, 11571383, 11971497, the Fundamental Research Funds for the Central Universities (SYSU 11lgpy27), NSF of Guangdong Province (2015A030313097), the Science and Technology Program of Guangzhou (201607010144).}
\address[]
{Department of Mathematics, Sun Yat-sen University
\\Guangzhou, 510275, P.R. China}
\email{luss@mail.sysu.edu.cn}

\begin{abstract} We show that for any fixed accuracy and time length $T$, a {\it finite} number of $T$-time length pieces of the complete trajectories on the global attractor are capable of uniformly approximating all trajectories within the accuracy in the natural strong metric after sufficiently large time when the observed dissipative system is asymptotically compact. Moreover, we obtain the strong equicontinuity of all the complete trajectories on the global attractor.  These results follow by proving the existence of a strongly compact strong trajectory attractor. The notion of a trajectory attractor was previously constructed for a family of auxiliary systems including the originally considered one without uniqueness.  Recently, Cheskidov and the author developed a new framework called evolutionary system, with which a (weak) trajectory attractor can be actually defined for the original system. In this paper, the theory of trajectory attractors is further developed in the natural strong metric for our purpose. We then apply it to both the 2D and the 3D Navier-Stokes equations and a general nonautonomous reaction-diffusion system.


\end{abstract}

\maketitle
\tableofcontents

\section{Introduction}
The global attractor is a natural mathematical object describing the long-time behavior of solutions of many dissipative partial differential equations (PDEs). All the solutions  converge to the attractor as time goes to infinity. Its studies goes back to the seminal work of Foias and Prodi \cite{FP67}, who proved that the long-time behavior of certain weak solutions of the 2D Navier-Stokes equations  (NSE) is determined by the long-time behavior of a finite number of numerical parameters. For further development, see e.g. \cite{La72, FT77, CFT85, Hal88, T88, Ha91, La91, BV92, Ro01, CV02, SY02, CLR13}.
\subsection{Main results and preliminary comments}
In this paper we will show that the global attractor possesses a {\it finite} strong uniform tracking property.   More precisely, we prove that,

\begin{mtheorem}\label{t:MTIntrod}(Conclusion 3 of Theorem \ref{t:sA}) Let $\Dc$ be an asymptotically compact evolutionary system\footnote{See Definitions \ref{Dc} and \ref{d:AC} in  preliminary
 Section 2 below, where we will briefly recall the basic definitions of the theory of evolutionary systems developed in \cite{CF06, C09, CL09, CL14}.}
satisfying the fundamental assumption\footnote{We will discuss the assumption a little later in this subsection.} A1 and  let $\bar{\Dc}$  be the closure of $\Dc$. Then the global attractor
$\As$ for $\Dc$ possesses the finite strong uniform tracking property, i.e., for any fixed accuracy $\epsilon>0$ and time length $T>0$, there exist $t_0$ and a {\it finite} set $P_T^f$ consisting of $T$-time length pieces  on $[0,T]$ of the complete trajectories of $\bar{\Dc}$ on  $\As$, such that for any $t^*>t_0$, every trajectory $u(t)$ of $\Dc$ satisfies
$$\ds(u(t), v(t-t^*)) < \epsilon,\quad\forall t\in [t^*,t^*+T],$$
for some $T$-time length piece  $v\in P^f_T$. Here $\ds$ is the natural strong metric.
\end{mtheorem}

\begin{remark}\label{r:remarkof MT1}We first give some explanation on terminologies in the theorem.
\begin{itemize}
\item[1.]One feature of our framework called evolutionary system is that the  phase space $X$ (typically being a bounded absorbing set of the dissipative system under consideration) is endowed with both a weak metric and a strong metric. In applications, the strong metric induces the natural strong topology we are concerned about. Notations with subscript $_\mathrm s$ or $_\mathrm w$ are related to the strong metric $\ds$ or the weak metric $\dw$, respectively.
\item[2.]All Leray-Hopf weak solutions of the 3D Navier-Stokes equations with a fixed time-dependent force staying in $X$, for instance, form an evolutionary system  satisfying the fundamental assumption A1. In general, evolutionary systems defined by PDEs of mathematical physics satisfy A1 (cf. e.g. [T88, CV02]).
\item[3.]Recently, Cheskidov and the author developed the framework of an evolutionary system  in \cite{CL14} by introducing a ``closure of the evolutionary system $\Dc$''. The structure of the  global
attractor $\As$ for $\Dc$ is obtained via  all the complete trajectories $\bar{\Dc}((-\infty, \infty))$ of its closure $\bar{\Dc}$. In applications, for an autonomous system, the closure $\bar{\Dc}$ of the associated evolutionary system $\Dc$ is identical to $\Dc$ itself.

\end{itemize}

\end{remark}

Intuitively, the {\it finite} strong uniform tracking property means that for any fixed accuracy $\epsilon$ and time length $T$, a
{\it finite} number of $T$-time length pieces of the complete trajectories on the global attractor $\As$ are capable of uniformly approximating all trajectories within the accuracy $\epsilon$ in the strong
metric after sufficiently large time.
The uniform tracking property indicates how the dynamics on  $\As$ describe the asymptotic behavior of all solutions of PDEs by picking up approximating pieces on  $\As$ one by one with smaller and smaller accuracies and longer and longer time lengths (see Corollary \ref{t:fsutpdetermine2}). It was studied in \cite{V92, LR99} for some special cases and in \cite{C09,CL14} for general cases.  The main novelty of this theorem is the {\it finiteness} of the number of candidate approximating pieces for every fixed accuracy and time length. It follows by proving the existence of a strongly compact strong trajectory attractor, which we will discuss in  Main Theorem \ref{t:sAIntrod} below.

The notions of a weak global attractor and a (weak) trajectory attractor were introduced for the autonomous 3D NSE by Foias and Temam \cite{FT85} and by Sell \cite{Se96}, respectively. As the fundamental model for the flow of fluid, the NSE are of great physical importance. However, the problem of uniqueness is still a highlighted difficulty in the theory of PDEs. Their methods attempt to bypass this obstacle\footnote{\label{fn:nonuniq}Very recently, the nonuniqueness of weak solutions to the 3D NSE has greatly progressed. See \cite{BuVi18, ABC22, CLuo22} and the references therein. Especially, the nonuniqueness of Leray-Hopf weak solutions have been proved \cite{ABC22} for the 3D NSE in the whole space with a non-zero force.}. The weak global attractor captures the long-time behavior of all Leray-Hopf weak solutions with respect to (w.r.t.) the weak topology of the natural phase space. The trajectory attractor is a global attractor in the space of trajectories, which is endowed with a suitable topology usually related to the weak topology. The set of points on all the trajectories in the trajectory attractor coincides with the weak global attractor \cite{FT85, Se96, C09, CL14}.  Only few papers are concerned with the strong topology (see \cite{VZ96, C09, CL09, VZC10, CVZ11, CL14}).

The trajectory attractor was further studied in \cite{Se96, CV97,CV02, SY02} for the nonautonomous system by investigating  a family of auxiliary systems containing the originally considered one, but not just the original one. Its trajectory attractor does not always have to be the one for the original system as  suggested by open problems in \cite{CL09, CL14}. Precisely, it might not satisfy the minimality property\footnote{See Definition \ref{d:wta}.} for the original system. Recently,  in our paper \cite{CL14},  Cheskidov and the author developed  the framework called evolutionary system and made it natural to construct a (weak) trajectory attractor for the {\it original} system rather than  for a family of systems. Indeed, the aforementioned minimality property does hold\footnote{We will have more detailed comments on our framework  and the previous ones in next subsetion.}. The notion of an evolutionary system $\mathcal E$ was initiated in \cite{CF06} to study a weak global attractor and a trajectory attractor for the autonomous 3D NSE, and continued in \cite{C09} where the strong convergence of trajectories to the trajectory attractor was studied. Our new results in the current paper are systematic investigations on trajectory attractors involved in the natural strong topology.

In fact, the previous paper \cite{CL14} presented primarily an approach that deals directly with the notion of a uniform global attractor for {\it original} nonautonomous systems. This notion, introduced by Haraux \cite{Ha91}, naturally generalizes that of a global attractor to nonautonomous ones. It was proved \cite{CL14} that the uniform global attractor possesses the uniform tracking property under the assumption A1.

In this paper, we further develop the theory of trajectory attractors in the natural strong metric for our purpose. We will show the existence of a strongly compact strong trajectory attractor when an evolutionary system is asymptotically compact.  As a consequence, we obtain that a {\it finite} number of pieces of  the complete trajectories on the global attractor are enough to ensure  the uniform tracking property in the strong metric, which is exactly Main Theorem \ref{t:MTIntrod}. The proof relies on two main ingredients. One is a new point of view that identifies a weak global attractor possessing the weak uniform tracking property with  the existence of a (weak) trajectory attractor.  The other one is a simultaneous use of the weak and strong metrics.

With the new perspective, part of results in \cite{CL14} can be reformulated into an inspirational form Theorem \ref{t:wA}, then we conveniently generalize the related notions to strong metric versions and  prove the following utilizing the weak and strong metrics at the same time.
\begin{mtheorem} \label{t:sAIntrod}(Conclusions 4-5 of Theorem \ref{t:wA} and 1-2 of Theorem \ref{t:sA}) Let $\Dc$ be an asymptotically compact evolutionary system
satisfying A1 and  let $\bar{\Dc}$  be the closure of $\Dc$. Then
\begin{itemize}\item[1.]The strongly compact strong trajectory attractor\footnote{See Definition \ref{d:sta}.} $\mathfrak A_{\s}$ exists, and it is  the restriction of all the complete trajectories of $\bar{\Dc}$ on $[0,\infty)$:
\[\mathfrak A_{\s}=\Pi_+\bar{\Dc}((-\infty, \infty)):=\left\{u(\cdot)|_{[0,\infty)}:u\in \bar{\Dc}((-\infty, \infty))\right\}. \]
\item[2.] $\As$ is a section of $\mathfrak A_{\s}$:
\[\As=\mathfrak A_{\s}(t):=\left\{u(t): u\in \mathfrak A_{\s}\right\},\quad \forall\, t\geq 0.\]

\end{itemize}
\end{mtheorem}
\begin{remark}\label{r:remarkof MT2}We supplement some notes.
\begin{itemize}
\item[1.]We can see that  all the complete trajectories $\bar{\Dc}((-\infty, \infty))$ of $\bar{\Dc}$ lie in $\As$ and their restriction on $[0,\infty)$ is $\mathfrak A_{\s}$. Then, the approximating pieces in Main Theorem \ref{t:MTIntrod} lie in  $\mathfrak A_{\s}$. Hence, it is also convenient to call that $\mathfrak A_{\s}$ or $\Dc$ possesses a finite strong uniform tracking property.
\item[2.]In applications, the evolutionary system $\Dc$ is only constructed from the solutions of the originally considered PDEs rather than those from a family of auxiliary PDEs, and thereby the attractors $\As$ and $\mathfrak A_{\s}$ are indeed the ones for the original system. They faithfully satisfy the minimality property for the original system. See next subsection for more detailed comments.
    \end{itemize}

\end{remark}
By the Arzal\`{a}-Ascoli Theorem, Main Theorem \ref{t:sAIntrod} has Main Theorem \ref{t:MTIntrod} and the following as corollaries.

\begin{mtheorem}\label{t:equiconIntrod}(Conclusion 4 of Theorem \ref{t:sA}) Let $\Dc$ be an asymptotically compact evolutionary system
satisfying A1 and  let $\bar{\Dc}$  be the closure of $\Dc$. Then, the strongly compact strong trajectory attractor $\mathfrak A_{\s}=\Pi_+\bar{\Dc}((-\infty, \infty))$ is strongly equicontinuous on $[0,\infty)$, i.e.,
\[
\ds\left(v(t_1),v(t_2)\right)\leq \theta\left(|t_1-t_2|\right), \quad\forall\, t_1,t_2\ge 0, \  \forall v\in \mathfrak A_{\s},
\]
where $\theta(l) $ is a positive function tending
to $0$ as $l\rightarrow 0^+$.
\end{mtheorem}
It is easy to see that, equivalently, all the complete trajectories $\bar{\Dc}((-\infty, \infty))$ of $\bar{\Dc}$ is strongly equicontinuous on $(-\infty,\infty)$.   Roughly speaking, Main Theorem \ref{t:equiconIntrod} excludes the existence of a complete trajectory on $\As$ that oscillates  more and more drastically.

Global attractors  are ever anticipated to be very complicated objects (fractals), which obstruct their applications. We expect that our finite strong uniform tracking property and strong equicontinuity, which are now described by the existence of a strongly compact strong trajectory attractor,  will do some good for their practical utilization, for instance for numerical simulations.

Now we make some comments on the assumption A1 that appears in our above results. It is the following:
 \begin{itemize}
\item[A1] The set consisting of all trajectories of $\Dc$ on $[0,\infty)$ is a precompact set in the space\footnote{It is the space of all weakly continuous $X$-valued  functions on $[0,\infty)$. See Remark \ref{r:remarkof MT1} or more details in Subsection \ref{ss:WTAS}.}
$ C([0,\infty); \Xw)$.
\end{itemize}
It provides the existence of a (weak) trajectory attractor for an evolutionary system. It also implies that the weak global attractor consists of points on the trajectories in the trajectory attractor (see Theorem \ref{t:wA}). A1 is satisfied by the associated evolutionary systems of the 2D and the 3D NSE and a general dissipative reaction-diffusion system (RDS) that we will investigate in this paper. In general, the associated evolutionary systems of PDEs that arise in mathematical physics satisfy A1 (cf. e.g. \cite{T88, CV02}). We will study more details on its relationships with the canonical closedness condition and continuity condition for evolutionary systems with uniqueness, which contain the classical  frameworks of a semigroup for autonomous systems and a family of processes for nonautonomous systems (see Lemmas \ref{t:closedtocontinuous}-\ref{t:continuoustoclosed}).

Our abstract theory  can be directly applied to both the 2D and the 3D nonautonomous NSE on a bounded domain $\Omega$ with space periodic  or with non-slip boundary conditions. We have the following for the 3D NSE.

 \begin{mtheorem}\label{t:As03DNSEIntrod}(Theorem \ref{t:As03DNSE}) Assume that the external force $g_0(t)$ is normal\footnote{See Definition \ref{d:normal}.} in $L^2_{\mathrm{loc}}(\mathbb{R}; V')$ and
every complete trajectory of the closure $\bar{\Dc}$ of the corresponding evolutionary system $\Dc$ of the 3D NSE with the fixed force $g_0(t)$ is strongly
continuous. Then
\begin{itemize}
\item[1.] The global attractor $\As$ satisfies the finite strong
uniform tracking property, i.e., for any fixed accuracy $\epsilon>0$ and time length $T>0$, there exist $t_0$ and a {\it finite} set $P_T^f$ consisting of $T$-time length pieces  on $[0,T]$ of the complete trajectories of $\bar{\Dc}$ on  $\As$, such that for any $t^*>t_0$, every Leray-Hopf weak solution $u(t)$  satisfies
$$|u(t)-v(t-t^*)| < \epsilon, \quad \forall t\in [t^*,t^*+T],$$
for some  $T$-time length piece  $v\in P_T^f$. Here $|\cdot|$ is the $\left(L^2(\Omega)\right)^3$-norm.
\item[2.]The strongly compact strong trajectory attractor $\mathfrak A_{\s}=\Pi_+\bar{\Dc}((-\infty, \infty))$, which is  the restriction of all the complete trajectories of $\bar{\Dc}$ on $[0,\infty)$, is strongly equicontinuous on $[0,\infty)$, i.e.,
\[
\left|v(t_1)-v(t_2)\right|\leq \theta\left(|t_1-t_2|\right), \quad\forall\, t_1,t_2\ge 0, \  \forall v\in \mathfrak A_{\s},
\]
where $\theta(l) $ is a positive function tending
to $0$ as $l\rightarrow 0^+$.
\end{itemize}
\end{mtheorem}

\begin{remark}We have the following remarks on the theorem.
\begin{itemize}
\item[1.]In our application, the phase space $X$ is a bounded absorbing subset in  the space $H$, the whole of which the previous frameworks are  used to taking as the natural phase space.  The strong metric $\ds$ of $X$ is defined by the $\left(L^2(\Omega)\right)^3$-norm $|\cdot|$  and  induces the natural strong topology the same as before. The weak metric $\dw$ of $X$ induces the weak topology that is the restriction on $X$ of  the weak topology of $\left(L^2(\Omega)\right)^3$ as mentioned earlier.
\item[2.]The normality condition on the  force was introduced in \cite{LWZ05} and the author \cite{Lu06}. Note that  the class of normal functions  in $L^2_{\mathrm{loc}}(\mathbb{R}; V')$ is bigger than the class of so-called translation compact functions in $L^2_{\mathrm{loc}}(\mathbb{R}; V')$ (see Definition \ref{d:trc}), which is necessary by applying previous frameworks \cite{Se96, CV94, CV97, CV02, SY02}.
\item[3.] Every complete trajectory in $\bar{\Dc}((-\infty, \infty))$ is a weak solution of the 3D NSE with the force $g_0(t)$ or  a force related to $g_0(t)$. See the proof of Lemma 5.4 in \cite{CL14}. For the autonomous 3D NSE, that is, the force is  time-independent,  $\bar{\Dc}((-\infty, \infty))=\Dc((-\infty, \infty))$.
\item[4.]The attractors $\As$ and $\mathfrak A_{\s}$  are exactly for the original 3D NSE with a fixed force $g_0(t)$. With a more restricted force $g_0(t)$ that is a translation compact function in $L^2_{\mathrm{loc}}(\mathbb{R}; V')$, the previous frameworks (see e.g. \cite{CV02}) are able to construct attractors for a family of auxiliary 3D NSE with forces related to $g_0(t)$. However, it is not know whether they are actual attractors for the 3D NSE  with the fixed force $g_0(t)$ we originally consider.  See Open Problem \ref{op:attractor}, or  next subsection for more details.
\item[5.]For the 2D NSE, concerned on both weak and strong solutions,  and a general RDS \eqref{RDSIntrod} below, we obtain similar results (see Theorems \ref{t:As02DNSEw}, \ref{t:As02DNSEs} and \ref{t:sturctureofattractorsRDS}, respectively). Note that, in these cases, the solutions have  already been proved to be strongly continuous (see Theorems \ref{t:Leray2Dws}, \ref{t:Leray2Dss} and \ref{t:existenceRDS}, respectively).
\item[6.]For the 2D cases, and for the RDS with more regular interaction terms, $\bar{\Dc}((-\infty, \infty))$ can be known in more details due to the uniqueness and better regularity of their solutions (see Theorems \ref{t:Aw02DNSEw} and \ref{t:As02DNSEw}, \ref{t:Aw02DNSEs} and \ref{t:As02DNSEs},  \ref{t:AwRDSStrcuniq} and \ref{t:AsRDSStrcuniq}, respectively).
\end{itemize}
\end{remark}

 It is worth to mention that we obtain the finite strong uniform attracting property and the strong equicontinuity of all the complete trajectories for these systems without additional conditions (cf. \cite{ CL14, LWZ05, Lu06, Lu07, CL09}).

The dissipative RDS:
\begin{equation}\label{RDSIntrod}
\partial_t u- a\Delta u+f(u,t)=g(x,t),
\end{equation}
 on a bounded domain $\Omega$ with Dirichlet  or with Neumann or with periodic boundary conditions
is another fundamental  model in the theory of  infinite dimensional dynamical systems.  It is quite general that covers many examples arising  in physics, chemistry and biology etc. We just list a few: the RDS with polynomial nonlinearity, Ginzburg-Landau equation, Chafee-Infante equation, FitzHugh-Nagumo equations and Lotka-Volterra competition system. See e.g. \cite{M87, T88, Ro01, CV02, SY02} for more. The main difficulty of RDS different from NSE in previous literature is the time-dependence of the reaction terms  $f(\cdot,t)$. Nevertheless, we obtain the much deeper results that are similar to Main Theorem \ref{t:As03DNSEIntrod} under even less restrictions on  $f(\cdot,t)$: We make no additional assumption other than the necessary conditions of continuity, dissipativeness  and growth (see \eqref{c:continuityoffRDS}-\eqref{c:boundednessoffRDS} in Section \ref{s:RDS}).  Contrastively, some conditions are imposed to construct a so-called symbol space, which is necessary for previous works \cite{CV94, CV97, CV02, Lu07, CL09}. Then, as a by product, we  give an  answer to an open problem in \cite{Lu07, CL09}, which concerns how to describe the structure of global attractors for the RDS with general interaction terms $f(\cdot,t)$.

It is not yet known  whether previous frameworks can also deal with such above cases as we indicate in Open Problems \ref{op:attractor} and \ref{op:attractorRDS} below (see also \cite{CL09, CL14}).

\subsection{More detailed comments on our new results}In this subsection, we discuss in more detail on the novelty of the paper, from both the theoretical and the applied points of view.

\begin{itemize}
[labelsep=0.5em, itemindent=0em, leftmargin=1.0em, rightmargin=1.5em,topsep=10pt, partopsep=0pt]
\item {\bf Theoretical point of view.}
{\it The main object of the paper is a notion of strongly compact strong trajectory attractor. In our previous work \cite{CL14},  a weak trajectory attractor is indeed defined for the original nonautonomous system under consideration rather than for a family of  auxiliary systems as done with previous frameworks. Now, it is possible to establish one in strong metric also for the original system. The proofs of the main theoretical results of Main Theorems \ref{t:MTIntrod}-\ref{t:equiconIntrod} rely on a new point of view and a simultaneous use of the weak and strong metrics. We also discuss the special evolutionary systems that include the classical frameworks of a semigroup and a family of processes.}
\end{itemize}

Our theory is based on the previous works of the  framework of an evolutionary system. It is originally designed in \cite{CF06, C09} for autonomous systems and developed in \cite{CL09, CL14} especially for nonautonomous systems.  Note that the   phase space $X$ (typically taking a bounded absorbing set of the observed dissipative system) of an evolutionary system is endowed with both a weak metric and a strong metric. The work \cite{CL14} mainly focused on the notion of a uniform global attractor for the {\it originally} considered nonautonomous system rather than for a family of auxiliary systems. The open problems in \cite{CL09, CL14} (see Open Problem \ref{op:attractor}) indicate that the uniform global and (weak) trajectory attractors constructed by the previous frameworks (see  \cite{Se96, CV97, CV02, SY02}) might not satisfy the minimality property with respect to uniformly attracting for the {\it original} nonautonomous system. Cheskidov and the author \cite{CL14} overcame the difficulty by taking a closure of the associated evolutionary system $\Dc$. The properties of the uniform global attractor for $\Dc$  are able to be investigated via that for its closure $\bar{\Dc}$.

Note that in applications the evolutionary system $\Dc$ only consists of the solutions of the original nonautonomous PDEs, thereby its attractors do faithfully satisfy above mentioned minimality property. In contrast, the previous frameworks investigate all the solutions of, for instance, a family of 3D NSE with forces in a suitable closure of the translation family of the original force. It results in that a strong translation compactness condition on the force  is necessary and such constructed attractors are only for the family of 3D NSE, which do not always have to coincide with those obtained in \cite{CL14} (see Open Problem \ref{op:attractor}). In theoretical settings,  the previous mostly used framework developed by Chepyzhov and Vishik \cite{CV97, CV02} was based on the use of the so-called time symbol (e.g. the external force in 3D NSE) and the related symbol  space, which on one hand require from the very beginning some necessary restrictions on the time symbols, on the  other hand, only provide auxiliary attractors. Moreover, as we will study in Section \ref{s:RDS},
in some interesting cases it is not even clear how to choose a symbol space. Our method \cite{CL14} is able to deal directly with the actual attractors for the original nonautonomous systems with less restrictions on the time symbols, since we avoid constructing such a symbol space. See \cite{CL14} as well as Subsection \ref{sss:ESuni} below for more discussions on our framework and the previous ones.

In the previous paper \cite{CL14}, it is also possible and natural to define a trajectory attractor for the {\it original} nonautonomous system. The notion of a trajectory attractor was originally introduced in \cite{Se96}  to bypass the difficulty of uniqueness issue of the 3D NSE (cf. footnote \ref{fn:nonuniq}). Its theory (see \cite{Se96, CV97, CV02, SY02}) is usually related to the weak topologies of the functional settings of the PDEs. In this paper, we further develop the theory of strongly compact strong trajectory attractors, which is   principally Main Theorem \ref{t:sAIntrod}. As a consequence, we derive Main Theorem \ref{t:MTIntrod} as well as Main Theorem \ref{t:equiconIntrod}. Our theory and its applications are new for both autonomous and nonautonomous dissipative systems of or lack of uniqueness.

It is worthwhile to remind again that, such a strongly compact strong trajectory attractor is certainly for the original system under consideration. Especially,  the minimality property with respect to uniformly strongly attracting (cf. Definitions \ref{d:stas} and \ref{d:sta}) the  original nonautonomous system is valid. This is possible on account of the theory established in our previous paper \cite{CL14} and the reformulation with a new perspective in the ongoing  paper. We will realize it as Main Theorem \ref{t:sAIntrod}. Moreover, we will  see that the existence a trajectory attractor describes the global attractor possessing the uniform tracking property (see Lemma \ref{l:wgawutp=wta}) and its strong compactness provides the {\it finiteness} of the number of candidate strongly approximating pieces (see Main Theorem \ref{t:MTIntrod}) and the strong equicontinuity of all the complete trajectories on the global attractor  (see Main Theorem \ref{t:equiconIntrod}).
 It is also worth to mention that we obtain simultaneously the strongly compact global attractor and the strongly compact strong trajectory attractor.

The proofs of Main Theorems \ref{t:MTIntrod}-\ref{t:equiconIntrod} are proceed as follows. First, we give a convenient definition of the weak uniform tracking property slightly different from that in \cite{C09, CL09, CL14}. Second, we show that, a weak global attractor for an evolutionary system $\Dc$  satisfying the weak uniform tracking property is equivalent to the existence of a (weak) trajectory attractor for $\Dc$  under the assumption A1. This equivalence makes it possible to reformulate the related results in the previous work \cite{CL14} in a more concise and inspirational form. Third, following this new form, we set out to prove a version in the strong metric. A strong trajectory attractor is naturally defined as a (weak) trajectory attractor which is also a strong trajectory attracting set (see Definition \ref{d:stas}).  The strong trajectory attracting property was investigated  in \cite{VZ96, C09, CL09, VZC10, CVZ11, CL14}, which are few papers involving the natural strong topology we are really interested in. Fourth, by taking advantage of a simultaneous use of the weak and strong metrics, we derive the strong compactness of the strong trajectory attractor, namely Main Theorem \ref{t:sAIntrod}, when the evolutionary system is asymptotically compact. Finally, thanks to the Arzal\`{a}-Ascoli Theorem, we  obtain the {\it finite} strong uniform tracking property and the strong equicontinuity of all the complete trajectories on the global attractor, that is, Main Theorems \ref{t:MTIntrod} and  \ref{t:equiconIntrod}, respectively. In particular, the systems that possess strongly compact global attractors are asymptotically compact. Hence, in these cases, the strongly compact strong trajectory attractors follow immediately once property A1 holds (see Corollary \ref{t:SAtoSTA} and Remark \ref{r:SAtoSTA}).

We investigate more details on the structure of the attractors, especially for closed evolutionary systems and evolutionary systems with uniqueness (see Definitions \ref{d:ESclosed} and \ref{d:ESuni}). These evolutionary systems include the classical frameworks of a semigroup and a family of processes as shown in \cite{CL14}. Thereout, we generalize the  results in \cite{CV02, LWZ05, Lu06, Lu07, CL14} in that, without additional assumptions, we obtain the existence of strongly compact strong trajectory attractors, and the sequent consequences of the {\it finite} strong uniform tracking property, the  strong equicontinuity of all the complete trajectories on global attractors and other properties following from Corollaries \ref{t:fsutpdetermine1}, \ref{t:fsutpdetermine2} and \ref{t:fsutpdetermine3}. In the case of  these evolutionary systems,  we also compare A1 with the familiar related conditions of closedness  and continuity imposed on the previous framework of a family of processes (see Lemmas \ref{t:closedtocontinuous}-\ref{t:continuoustoclosed}).

\begin{itemize}
[labelsep=0.5em, itemindent=0em, leftmargin=1.0em, rightmargin=1.5em,topsep=12pt, partopsep=0pt]
\item {\bf Applied point of view.} {\it The applications of our new theory are first focused on both the 3D and the 2D Navier-Stokes equations (NSE). The NSE are probably the most fundamental example in the theory of infinite dimensional dynamics systems, most part of which has been established by investigating this example. We also study a general reaction-diffusion system (RDS), which is another fundamental model.}
\end{itemize}

We apply our abstract theory to the 3D NSE, concerned on Leray-Hopf weak solutions, and the 2D NSE, concerned on both weak and strong solutions, with a fixed force as continuations of \cite{CL14}, \cite{Lu06} and \cite{LWZ05}, respectively. No extra condition is assumed. The property A1 is verified by the well-known compactness lemmas (see e.g. Lemma \ref{l:precompactofLH23}) on the solutions. The new results are Main Theorem \ref{t:As03DNSEIntrod} for the 3D case and analogues for the 2D cases. Besides the {\it finite} strong uniform tracking property, one of the new observation is that,  for instance,  for the 2D NSE with a fixed normal force in $L^2_{\mathrm{loc}}(\mathbb{R}; V')$,  it denies the existence of the more and more wildly oscillating complete weak  solutions. Note that, actually, for the autonomous 2D NSE, the strong compactness of the strong trajectory attractors could have been deduced by the classical estimates (see \eqref{i:uniq2DNSEautot}) that imply the uniqueness of the solutions and the continuity of the associated semigroups. Nevertheless, this is not valid for the nonautonomous case (see Remark \ref{r:uniq2DNSE}). Now we are able to cope with this case.

Another fundamental model we study  is a RDS (see \eqref{RDSIntrod}) with a fixed pair of a time-dependent nonlinearity and a driving force. It is treated  along the same line as the NSE. Although the nonlinearity depending on time is a main difficulty in previous studies (see e.g. \cite{CV97, CV02, Lu07, CL09}), we assume less that the nonlinearity only satisfies the basic conditions of continuity, dissipativeness  and growth (see \eqref{c:continuityoffRDS}-\eqref{c:boundednessoffRDS}). This is owing to no need to construct a so-called symbol space with our method. The three conditions neither  guarantee the unique solvability of \eqref{RDSIntrod} nor  provide a suitable  symbol space. The assumption on the force is a translation boundedness condition, which is the weakest condition that ensures the existence of a bounded  uniformly w.r.t. the initial time absorbing ball. We take this ball as a phase space $X$. The  weak and strong metrics are metrics that induce the usual weak and strong topologies of the space\footnote{$N$ is the number of  components of the unknown vector function $u$ of \eqref{RDSIntrod}.} $\left(L^2(\Omega)\right)^N$ restricted on $X$, respectively.   We verify that all the weak solutions of the RDS \eqref{RDSIntrod} form an evolutionary system satisfying A1. Therefore, we obtain the existence and structure of the weak attractors $\Aw$ and $\mathfrak A_{\w}$. In addition, if the force is normal in $L_{\mathrm{loc}}^2\left(
\Bbb{R};V'\right)$, where $V'$ is the dual of $\left(H_0^1(\Omega)\right)^N$, the evolutionary system is asymptotically compact.  Then, the two weak attractors  $\Aw$ and $\mathfrak A_{\w}$ are strongly compact strong attractors $\As$ and $\mathfrak A_{\s}$. Consequently, the {\it finite} strong uniform tracking property for the RDS \eqref{RDSIntrod} holds, i.e., for  any fixed accuracy $\epsilon$ and time length $T$, a
{\it finite} number of $T$-time length pieces of the complete trajectories on $\As$ are enough to uniformly approximate all weak solutions within the accuracy $\epsilon$ in the $\left(L^2(\Omega)\right)^N$-norm
after sufficiently large time. Moreover, all the complete trajectories on $\As$ is equicontinuous on $(-\infty, \infty)$ in the $\left(L^2(\Omega)\right)^N$-norm. Note again that there is no additional assumption on the time-dependent nonlinearity. Hence, part of these results answer an open problem in \cite{Lu07, CL09}, where an appropriate symbol space is absent (see Remark \ref{r:ansOP} for more details).

We also consider the RDS \eqref{RDSIntrod} with more regular nonlinearities $f(\cdot,t)$.  One more assumption imposed on $f(\cdot,t)$  is a translation compact condition (see Section \ref{ss:MoreRegN}), which is necessary for the previous works to obtain the structure of the global attractors (see \cite{CV94, CV95, CV97, CV02, Lu07, CL09}). Analogously for the 3D NSE, we boost the results in \cite{CL09}. Another more assumption (see \eqref{c:uniquenessoffRDS}) on $f(\cdot,t)$ guarantees the uniqueness of the weak solutions of \eqref{RDSIntrod}. Hence, we can know more properties about the structure of the attractors and  obtain similar results to those of 2D NSE, which generalize the results in \cite{CV02, Lu07}.

 At the end of the paper, we construct several  interesting examples of nonlinearities $f(\cdot,t)$  that do not satisfy the two assumptions just mentioned any more, but our theory is still applicable for the RDS \eqref{RDSIntrod} with these nonlinearities. The  pointwise limit function $f_\infty(\cdot)$  of one example as $t\rightarrow +\infty$ is discontinuous and the others even have no pointwise limit functions.  These facts hint that there might not exist a suitable symbol space which is very required in previous studies \cite{CV97, CV02, Lu07, CL09}. It is not clear whether  the results of the RDS \eqref{RDSIntrod} with such nonlinearities can also be obtained with previous frameworks (see Open Problem \ref{op:attractorRDS}).

\subsection{Paper outline}
The rest of this paper is organized as follows.
\begin{itemize}
\item In Sect. \ref{s:ES}, we briefly recall the basic definitions of  the theory of evolutionary systems and a criterion of their asymptotical compactness, developed in \cite{CF06, C09, CL09, CL14}.  For our convenience, we formulate as a definition the strong trajectory attracting property obtained in \cite{C09, CL09, CL14}.
\item In Sect.  \ref{s:AttractES}, we devote to prove main theoretical theorems of the paper.
In Subsect. \ref{ss:WAttractES}, we concern about the weak trajectory attractors and in Subsect.  \ref{ss:SAttractES}, the strongly compact strong trajectory attractors and the sequent corollaries, such as the finite strong uniform tracking property. Comparisons with  known results are also mentioned.
\item In Sect.  \ref{s:NSE}, we apply our new theory to  both the 3D and  the 2D NSE with a fixed force. Similarly, we also present comparisons with existing literature.
\item In Sect. \ref{s:RDS}, we apply the theory to  the RDS with a fixed pair of a time-dependent nonlinearity and a driving force along the same line as the NSE. In Subsect. \ref{ss:MoreRegN}, the  RDS with more regular nonlinearities are considered.
In Subsect.  \ref{ss:nonln}, we collect and study some properties on the nonlinearities and give several examples with which our theory is applicable for the RDS while previous frameworks do not work.
\end{itemize}

\section{Evolutionary System}\label{s:ES}
Now we  briefly recall the  basic definitions on  evolutionary
systems. See \cite{C09, CL09, CL14} for details. An important property obtained in these references is formulated as a definition (see Definition \ref{d:stas}) for the purpose of this paper. We also  recall a criterion of the asymptotical compactness for evolutionary systems developed in \cite{CF06, C09, CL09, CL14}.

\subsection{Phase space endowed with two metrics}\label{ss:PhaseS}
Assume that a set $X$ is  endowed with two metrics $\ds(\cdot,\cdot)$ and $\dw(\cdot,\cdot)$ respectively, satisfying the following conditions:
\begin{enumerate}
\item $X$ is $\dw$-compact.
\item If $\ds(u_n, v_n) \to 0$ as $n \to \infty$ for some
$u_n, v_n \in X$, then $\dw(u_n, v_n) \to 0$ as $n \to \infty$.
\end{enumerate}
Hence, we will  refer
to $\ds$ as a strong metric and $\dw$ as a weak metric. Let $\overline{A}^{\bullet}$ be the closure of a set $A\subset X$ in the topology generated by $\db$. Here (the same below) $\bullet=\s$ or $\w$.
Note that any strongly compact ($\ds$-compact) set is weakly compact
($\dw$-compact), and any weakly closed set is strongly closed.

\subsection{(Autonomous) evolutionary system}\label{ss:AutoES}
Let
\[
\mathcal{T} := \left\{ I: \ I=[\tau,\infty) \subset \mathbb{R}, \mbox{ or }
I=(-\infty, \infty) \right\},
\]
and for each $s\in \mathbb{R}$, let $I+s$ be $[\tau,\infty)$ if $I=[\tau+s,\infty)$ or $(-\infty, \infty) $ if
$I=(-\infty, \infty) $. For any$I \in\mathcal{T}$, denote by $\mathcal{F}(I)$
the set of all $X$-valued functions on $I$.
Now we define an evolutionary system $\Dc$ as follows.
\begin{definition} \label{Dc}
A map $\Dc$ that associates to each $I\in \mathcal{T}$ a subset
$\Dc(I) \subset \mathcal{F}(I)$ will be called an evolutionary system if
the following conditions are satisfied:
\begin{enumerate}
\item $\Dc([0,\infty)) \ne \emptyset$.
\item
$\Dc(I+s)=\{u(\cdot): \ u(\cdot +s) \in \Dc(I) \}$ for
all $s \in \mathbb{R}$.
\item $\{u(\cdot)|_{I_2} : u(\cdot) \in \Dc(I_1)\}
\subset \Dc(I_2)$ for all
pairs $I_1,I_2 \in \mathcal{T}$, such that $I_2 \subset I_1$.
\item
$\Dc((-\infty , \infty)) = \{u(\cdot) : \ u(\cdot)|_{[\tau,\infty)}
\in \Dc([\tau, \infty)), \forall \tau \in \mathbb{R} \}.$
\end{enumerate}
\end{definition}

We will refer to $\Dc(I)$ as the set of all trajectories
on the time interval $I$. The set $\Dc((-\infty,\infty))$ is called the kernel of $\Dc$ and the trajectories in it are called complete.
Let $P(X)$ be the set of all subsets of $X$.
For every $t \geq 0$, define a set-valued map
\begin{eqnarray*}
&R(t):P(X) \to P(X),&\\
&R(t)A := \{u(t): u(0)\in A, u \in \Dc([0,\infty))\}, \quad
A \subset X.&
\end{eqnarray*}
Note that the assumptions on $\Dc$ imply that $R(t)$ enjoys
the following property:
\begin{equation*} \label{eq:propR(T)}
R(t+s)A \subset R(t)R(s)A, \quad \, \forall A \subset X,\quad t,s \geq 0.
\end{equation*}
\begin{definition} \label{d:attractor}
A set $\Ab\subset X$ is a $\db$-global attractor if $\Ab$ is a minimal set
that is
\begin{enumerate} \item $\db$-closed. \item $\db$-attracting: for
any $B\subset X$ and $\epsilon>0$, there exists $t_0$, such that
\[R(t)B\subset B_\bullet(\Ab,\epsilon):=\{y\in X:
\inf_{x\in \Ab}\db(x,y)<\epsilon\},\quad \forall\, t\geq t_0.
\]
\end{enumerate}
\end{definition}
\begin{definition}The $\wb$-limit of a set $A\subset X$ is
\[
\wb(A):=\bigcap_{\tau\ge0}\overline{\bigcup_{t\ge \tau}R(t)A}^{\bullet}.
\]
\end{definition}

\begin{definition}\label{d:AC}An evolutionary system $\Dc$ is asymptotically
compact if for any $t_n\rightarrow+\infty $ and any $x_n\in R(t_n)X$,
the sequence $\{x_n\}$ is relatively strongly compact.
\end{definition}
\begin{definition}
Let $\Dc$ be an evolutionary system. If a map $\Dc^1$ that associates to each
$I\in \mathcal T$ a subset $\mathcal  E^1(I)\subset \mathcal E(I)$ is also an evolutionary system, we will call it an evolutionary subsystem of $\Dc$, and denote by $\mathcal  E^1\subset \mathcal E$.
\end{definition}
Note that it is in fact sufficient that for each $I\in \mathcal T\backslash\{(-\infty,\infty)\}$, $\mathcal  E^1(I)\subset \mathcal E(I)$, since $\mathcal  E^1((-\infty,\infty))\subset \mathcal E((-\infty,\infty))$ follows immediately from the definition of an evolutionary system.

In order to extend the notion of invariance from a semiflow to an
evolutionary system, we will need the following mapping:
\[
\Rc(t) A := \{u(t): u(0) \in A, u \in \Dc((-\infty,\infty))\}, \quad A \subset X,
\ t \in \mathbb{R}.
\]
\begin{definition} A set $A\subset X$ is positively invariant if
\[
\Rc(t) A\subset A, \quad \forall t\geq 0.
\]
$A$ is invariant if
\[
\Rc(t) A= A, \quad \forall t\geq 0.
\]
$A$ is quasi-invariant if for every $a\in A$ there exists a complete trajectory
$u \in \Dc((-\infty,\infty))$ with $u(0)=a$ and $u(t) \in A$ for all $t\in \mathbb{R}$.
\end{definition}

\subsection{Nonautonomous evolutionary system and reducing to autonomous system}\label{ss:NonautoES}
Let $\Sigma $ be a parameter set and $\left\{T(s): s\ge
0\right\}$ be a family of operators acting on $\Sigma$ satisfying
 $T(s)\Sigma=\Sigma$, $\forall s\ge0$.  Any element $\sigma\in \Sigma $ will be  called (time) symbol and $ \Sigma $ will be called (time) symbol
 space. For instance, in many applications $\{T(s)\}$ is the translation
 semigroup and $ \Sigma$ is  the translation family
 of the time-dependent items of the considered system or its closure in some appropriate topological space (for more examples see \cite{CV94, CV02, CL14}, the appendix in \cite{CLR13}).
\begin{definition} \label{d:NDc}A family of maps $\mathcal E_\sigma $, $\sigma \in \Sigma $ that for every $\sigma\in \Sigma$
associates to each $I\in \mathcal T$ a subset $\mathcal
E_\sigma(I)\subset \mathcal F(I)$ will be called a nonautonomous
evolutionary system  if the following conditions are satisfied:
\begin{enumerate}
\item $\mathcal E_\sigma ([\tau, \infty ))\neq \emptyset, \forall\, \tau \in \mathbb R$.
\item $\mathcal E_ \sigma (I+s)=\{u(\cdot): u(\cdot+s)\in  \mathcal E_{T(s)\sigma}(I)\},
\forall s\ge 0$. \item $\{u(\cdot)|_{I_2}:u(\cdot)\in \mathcal
E_\sigma(I_1)\}\subset \mathcal E_\sigma(I_2)$, $\forall\, I_1$,
$I_2\in \mathcal T$, $I_2\subset I_1$.
\item $\mathcal E_\sigma((-\infty,\infty))=\{u(\cdot):u(\cdot)|_{[\tau,\infty)}\in
\mathcal E_\sigma([\tau,\infty)),\forall\,\tau \in \mathbb R\}$.
\end{enumerate}
\end{definition}

It is shown in \cite{CL09,CL14} that any nonautonomous evolutionary system can be viewed as an (autonomous)  evolutionary system.  Let\footnote{Here we fix a minor flaw in Section 2.3 in \cite{CL14}. See Lemma \ref{t:kerofESSig} in  Section \ref{ss:Ker} below.}
\begin{align*}\mathcal E_\Sigma (I):=\bigcup_{\sigma\in \Sigma}\mathcal E_\sigma (I),\quad \forall\, I\in \mathcal
T\backslash\{(-\infty,\infty)\},\end{align*}
and
\begin{align*}\mathcal E_\Sigma ((-\infty,\infty)):=\{u(\cdot):u(\cdot)|_{[\tau,\infty)}\in
\mathcal E_\Sigma([\tau,\infty)),\forall\,\tau \in \mathbb R\}.\end{align*}
Now we define an (autonomous) evolutionary system $\mathcal E$ in the following way:
$$\mathcal E (I):=\mathcal E_\Sigma (I),\quad
\forall\, I\in \mathcal T.$$
It can be checked that  all the conditions in Definition~\ref{Dc} are satisfied. Consequently, the above notions of invariance, quasi-invariance, and a global attractor for $\Dc$  can be extended to  the nonautonomous evolutionary system $\{\Dc_\sigma\}_{\sigma\in \Sigma}$.  The global attractor in the nonautonomous case will be conventionally called a uniform global attractor (or simply a global attractor). However,  for some evolutionary systems constructed from nonautonomous dynamical systems the associated symbol spaces are not known. See \cite{CL14} and the following sections for more details. Thus, we will not distinguish between autonomous and nonautonomous evolutionary systems. If it is necessary,  we denote an evolutionary system with a symbol space $\Sigma$ by $\Dc_{\Sigma}$ and its global attractor by $\mathcal A^{\Sigma}$, and so on.

\begin{definition}\label{d:ESuni}An evolutionary system $\Dc_\Sigma$ is a system with uniqueness if for every $u_0\in X$
and $\sigma \in \Sigma$, there is a unique trajectory $u \in \Dc_\sigma([0, \infty)) $ such that $u(0) =u_0$.
\end{definition}

\subsection{Weak trajectory attracting set and weak trajectory attractor}\label{ss:WTAS}
Denote by  $C([a, b];X_\bullet)$ the space of $\db$-continuous $X$-valued
functions on $[a, b]$  endowed with the metric
\[
\dd_{C([a, b];X_\bullet)}(u,v) := \sup_{t\in[a,b]}\db(u(t),v(t)).
\]
Let also $C([a, \infty);X_\bullet)$ be the space of $\db$-continuous
$X$-valued functions on $[a, \infty)$
endowed with the metric
\[
\dd_{C([a, \infty);X_\bullet)}(u,v) := \sum_{l\in \mathbb{N}} \frac{1}{2^l} \frac{\dd_{C([a, a+l];X_\bullet)}(u,v)}
{1+\dd_{C([a, a+l];X_\bullet)}(u,v)}.
\]
Note that the convergence in $C([a, \infty);X_\bullet)$ is equivalent to uniform convergence on
compact sets.

Now we suppose that evolutionary systems $\Dc$ satisfy the following assumption:
\[
\Dc([0,\infty))\subset C([0,\infty); \Xw).
\]
 Define the family of translation operators $\{T(s)\}_{s\geq0}$,
\begin{equation}\label{e:dTO}
(T(s)u)(\cdot):=u(\cdot+s)|_{[0,\infty)},\quad u\in C([0,\infty);\Xw).
\end{equation}
Due to the property (3) of the evolutionary system (see Definitions
\ref{Dc} and \ref{d:NDc}), we have that,
$$T(s)\Dc([0,\infty))\subset \Dc([0,\infty)),\quad \forall\, s\geq 0.$$
Note that $\Dc([0,\infty))$ may not be closed in $ C([0,\infty);\Xw)$. We consider the dynamics of the translation semigroup $\{T(s)\}_{s\geq 0} $ acting on the phase space  $C([0,\infty);\Xw)$.
A set $P\subset C([0,\infty);\Xw)$ weakly uniformly attracts a set
$Q\subset \Dc([0,\infty))$ if for any $\epsilon>0$, there exists $t_0$,
such that
\begin{equation*}\label{e:wtaing}
T(t)Q\subset \left\{v\in C([0,\infty);\Xw):\inf_{u \in P}\dd_{C([0,\infty);\Xw)}(u,v)<\epsilon\right\}, \quad\forall\, t\geq t_0.
\end{equation*}
\begin{definition}\label{d:wtas}
A set $P\subset C([0,\infty);\Xw)$ is a weak trajectory attracting set for an evolutionary system $\Dc$ if it
weakly uniformly attracts $\Dc([0,\infty))$.
\end{definition}

\begin{definition}\label{d:wta}
A set $\mathfrak A_{\w}\subset  C([0,\infty);\Xw)$ is a weak trajectory attractor for an evolutionary system $\Dc$ if $\mathfrak A_{\w}$ is a minimal  weak trajectory attracting set that is
\begin{enumerate}
\item Closed in  $ C([0,\infty);\Xw)$.
\item Invariant: $T(t)\mathfrak A_{\w}=\mathfrak A_{\w}$,  $\forall\, t\geq 0$.
\end{enumerate}
\end{definition}
It is easy to see that if a weak trajectory attractor exists, it is unique. In previous literature (see e.g. \cite{Se96, CV97, CV02,  SY02, C09, CL14}),  this kind of attractor is just called a trajectory attractor. We now use the current name for a distinction, since in the paper we develops the theory of trajectory attractors related to the strong topology.

\subsection{Strong trajectory attracting set}\label{ss:STAS}
In  \cite{C09, CL14}, a property of uniformly attracting $\Dc([0,\infty))$ in $L_{\mathrm{loc}}^{\infty}((0,\infty);\Xs)$ was obtained. Here, analogy to Definition \ref{d:wtas}, we retell it as following Definition \ref{d:stas}.
A set $P\subset C([0,\infty);\Xw)$  strongly uniformly attracts a set
$Q\subset \Dc([0,\infty))$ if for any $\epsilon>0$  and $T>0$, there exists $t_0$,
such that
\begin{equation*}\label{e:wtaing}
T(t)Q\subset \left\{v\in C([0,\infty);\Xw):\inf_{u\in P}\sup_{s\in [0,T]}\ds(u(s),v(s))<\epsilon\right\}, \quad\forall\, t\geq t_0.
\end{equation*}
\begin{definition}\label{d:stas}
A set $P\subset C([0,\infty);\Xw)$ is a strong trajectory attracting set for an evolutionary system $\Dc$ if it
 strongly uniformly attracts $\Dc([0,\infty))$.
\end{definition}
The advantage of such a definition will be seen in the next section, where we will naturally define a strong trajectory attractor analogy to  Definition \ref{d:wta}, and obtain its strong compactness.

\subsection{Fundamental assumption A1}\label{ss:FAA1}
We will investigate evolutionary systems $
\Dc$ satisfying the following property:
\begin{itemize}
\item[A1] $\Dc([0,\infty))$ is a precompact set in
$ C([0,\infty); \Xw)$.
\end{itemize}
In general, the evolutionary systems defined by PDEs of mathematical physics satisfy A1 (cf. e.g. \cite{T88, CV02}).

  The evolutionary systems satisfying A1 are closely related to the concept of the uniform w.r.t. the initial time global attractor for a nonautonomous dynamical system, initiated by Haraux \cite{Ha91}. For instance, as shown in \cite{CL14}, the uniform global attractor for an evolutionary system $\Dc_{\Sigma}$ defined by a process $\{U_{\sigma_0}(t,\tau)\}$  is the uniform w.r.t. the initial time global attractor for $\{U_{\sigma_0}(t,\tau)\}$ due to Haraux. A stronger version of A1 used in \cite{CF06, C09, CL09} is the following:
  \begin{itemize}
\item[\={A1}] $\Dc([0,\infty))$ is a compact set in
$ C([0,\infty); \Xw)$.
\end{itemize}
However, instead of the property \={A1}, the  evolutionary system $\Dc_{\Sigma}$ usually satisfies only A1. For more details see \cite{CL09, CL14}.

\subsection{Closure of an evolutionary system}\label{ss:ClosureES}
Let
\[
\bar{\Dc}([\tau,\infty)):=\overline{\Dc([\tau,\infty))}^{ C([\tau,
\infty);\Xw)},\quad\forall\,\tau\in\mathbb R,
\]
and
\begin{align*}\bar{\Dc} ((-\infty,\infty)):=\{u(\cdot):u(\cdot)|_{[\tau,\infty)}\in
\bar{\Dc}([\tau,\infty)),\forall\,\tau \in \mathbb R\}.\end{align*}
It can be checked that $\bar{\Dc}$ is also an evolutionary system. We call
$\bar{\Dc}$ the closure of the evolutionary system $\Dc$, and add for  $\bar{\Dc}$ the top-script $\bar{\ }$ to the corresponding notations for $\Dc$. For example, we denote by $\bar{\mathcal
A}_{\bullet}$ the uniform $\db$-global attractor for $\bar{\Dc}$.

Obviously,  if $\Dc$ satisfies A1, then $\bar{\Dc}$ satisfies \={A1}.
Note that, a nonautonomous evolutionary system $\Dc$ usually satisfies only A1 rather than \={A1}. Moreover, for some nonautonomous  evolutionary systems, as we will see in Sections \ref{s:NSE} and \ref{s:RDS},  there may not exist suitable symbol spaces associated to their closures. However, Cheskidov and the author showed in \cite{CL14} that it is possible to investigate the global attractor  as well as the weak trajectory attractor for any  $\Dc$ via those for its closure $\bar{\Dc}$, no matter whether it is lack of uniqueness or  lack of a symbol space, since our approach avoids the necessity  of constructing a symbol space.

\subsection{A criterion of asymptotical compactness}\label{ss:ACriterion}

We recall a method  to verify the asymptotical compactness of  evolutionary systems satisfying these
additional properties (see \cite{CF06, C09, CL09, CL14}):
\begin{itemize}
\item[A2] (Energy inequality)
Assume that $X$ is a set in some
Banach space $H$ satisfying the Radon-Riesz property (see below)
with the norm denoted by $|\cdot|$, such
that $\ds(x,y)=|x-y|$ for $x,y \in X$ and $\dw$ induces the weak
topology on $X$.
Assume also that for any $\epsilon
>0$, there exists $\delta>0$, such that for every $u \in
\Dc([0,\infty))$ and $t>0$,
\[
|u(t)| \leq |u(t_0)| + \epsilon,
\]
for $t_0$ a.e. in $ (t-\delta, t)$.
\item[A3] (Strong convergence
a.e.) Let $u_n \in \Dc([0,\infty))$ be such that,
$u_n$ is $\mathrm d_{ C([0, T];\Xw)}$-Cauchy sequence in
$ C([0, T];\Xw)$ for some $T>0$. Then $u_n(t)$ is
$\ds$-Cauchy sequence a.e. in $[0,T]$.
\end{itemize}

A Banach space $\mathcal B$ is said to satisfy the Radon-Riesz property if for any sequence $\{x_n\}\subset \mathcal B$,
\[
x_n\rightarrow x \  \mbox{ strongly in  } \mathcal B\Leftrightarrow
\left\{
\begin{split}
& x_n\rightarrow x \  \mbox{ weakly in  } \mathcal B\\
\|&x_n\|_{\mathcal B}\rightarrow \|x\|_{\mathcal B}
\end{split}
\right., \quad \mbox{as } n\rightarrow \infty.
\]
In many applications, $H$ in A2 is a uniformly convex separable
Banach space and  $X$ is a bounded closed set in $H$. Then the weak topology of $H$ is metrizable on $X$, and $X$ is
compact w.r.t. such a metric $\dw$. Moreover, the Radon-Riesz property is automatically satisfied in this case.

We have the following criterion of asymptotical compactness that is sufficient for the applications in this paper. For more, see Remark \ref{r:SAtoSTA}.

\begin{theorem} \cite{CL14}\label{t:A123Comp}
Let $\Dc$ be an evolutionary system satisfying A1, A2, and
A3, and assume that its closure $\bar{\Dc}$ satisfies
$\bar{\Dc}((-\infty,\infty))\subset C((-\infty, \infty);\Xs)$. Then
$\Dc$ is asymptotically compact.
\end{theorem}

\section{Attractors for Evolutionary System}\label{s:AttractES}

An important property of a global attractor called uniform tracking property has been studied in \cite{C09, CL09, CL14}. This property indicates how the dynamics on the global attractor describes a long-time behavior of every trajectory of an evolutionary system (see e.g. \cite{Ro01}). We now show that a  weak global attractor possessing the weak uniform tracking property is equivalent to the existence of a weak trajectory attractor.
Inspired by this new point of view, we further develop in this section a notion of a strongly compact strong trajectory attractor for an evolutionary system $\Dc$, which connotes deep results on the strong uniform tracking property. It is remarkable that we obtain its existence at the same time we get the strongly compact strong global attractor.

\subsection{Weak trajectory attractor:  Revisit with a new point of view}\label{ss:WAttractES}
In this subsection, we first investigate some properties of a weak trajectory attractor.  Then, we reformulate related results obtained in \cite{CL14}. We start with introducing the following definition.

\begin{definition}\label{d:wutp}
A set $P\subset C([0,\infty);\Xw)$  satisfies the weak uniform tracking property for an evolutionary system $\Dc$ if for any $\epsilon >0$, there exists $t_0$, such that for any $t^*>t_0$, every trajectory $u \in
\Dc([0,\infty))$ satisfies
\begin{equation*} \label{e:wutp}
\dd_{C([t^*,\infty);\Xw)} (u(\cdot), v(\cdot-t^*)) < \epsilon,
\end{equation*} for some trajectory $v \in P$.
\end{definition}
\begin{lemma}\label{l:=wutp}
Let  $P\subset C([0,\infty);\Xw)$ satisfy $T(s)P=P$,  $\forall\, s\geq 0$.  Then $P$ satisfies the weak uniform tracking property for an evolutionary system $\Dc$ if and only if for any $\epsilon >0$, there exists $t_0$, such that for any $t^*>t_0$, every trajectory $u \in
\Dc([0,\infty))$ satisfies
\begin{equation*} \label{e:=wutp}
\dd_{C([t^*,\infty);\Xw)} (u, v) < \epsilon,
\end{equation*} for some trajectory $v \in P$.
\end{lemma}
\begin{proof}
By the assumption, for any $t^*\geq 0$, $v\in P$ if and only if there exists $v^*\in P$ that $T(t^*)v^*=v$. Hence, for any $\epsilon >0$  and $u \in
\Dc([0,\infty))$,
\begin{equation*}
\dd_{C([t^*,\infty);\Xw)} (u(\cdot), v(\cdot-t^*)) < \epsilon,
\end{equation*}for some trajectory $v \in P$ is equivalent to
\begin{equation*}
\dd_{C([t^*,\infty);\Xw)} (u(\cdot), v^*(\cdot)) < \epsilon,
\end{equation*}for some trajectory $v^* \in P$.
\end{proof}
The latter property in  Lemma \ref{l:=wutp} is in fact called weak uniform tracking property in \cite{C09, CL09, CL14}. Instead, with this lemma in hand, we now use Definition \ref{d:wutp} for our later convenience.

Now, we have the following relationship between the weak uniform tracking property and the weak trajectory attracting property.
\begin{lemma}\label{l:wta=wutp}
A set $P\subset C([0,\infty);\Xw)$ satisfies  the weak uniform tracking property for an evolutionary system $\Dc$ if and only if it is  a weak trajectory attracting set for $\Dc$.
\end{lemma}
\begin{proof}Suppose that $P$ is  a weak trajectory attracting set for an evolutionary system $\Dc$. For any $\epsilon>0$, there exists $t_0$,
such that
\begin{equation*}\label{e:wtaing1}
T(t)\Dc([0,\infty))\subset \left\{v\in C([0,\infty);\Xw):\dd_{C([0,\infty);\Xw)}(P,v)<\epsilon/2\right\}, \ \ \forall\, t\geq t_0.
\end{equation*}
Hence, for any $t^*\geq t_0$ and every trajectory $u\in \Dc([0,\infty))$, we know that
\begin{equation}\label{e:wtaing2}
\dd_{C([0,\infty);\Xw)}(v (\cdot), (T(t^*)u)(\cdot))<\epsilon,
\end{equation}
for some $v\in P$. By (\ref{e:dTO}) of the definition of the family of translation operators $\{T(s)\}_{s\geq 0}$, (\ref{e:wtaing2}) is
\begin{equation}\label{e:wtaing3}
\dd_{C([0,\infty);\Xw)}(v (\cdot), u(\cdot+t^*))<\epsilon.
\end{equation}
By a change of the variable,  (\ref{e:wtaing3}) is equivalent to
\begin{equation}\label{e:wtaing4}
\dd_{C([t^*,\infty);\Xw)}(v (\cdot-t^*), u(\cdot))<\epsilon.
\end{equation}
Therefore, $P$ satisfies the weak uniform tracking property.

Conversely, assume that for any $\epsilon >0$, there exists $t_0$, such that for any $t^*>t_0$, every trajectory $u \in
\Dc([0,\infty))$ satisfies the inequality \eqref{e:wtaing4} for some $v\in P$. Equivalently, \eqref{e:wtaing2} is valid, which implies that
\begin{equation*}
T(t)\Dc([0,\infty))\subset \left\{v\in C([0,\infty);\Xw):\dd_{C([0,\infty);\Xw)}(P,v)<\epsilon\right\}, \ \ \forall\, t\geq t_0+1.
\end{equation*}
Thus,  $P$ is  a weak trajectory attracting set.
\end{proof}

Let $\Dc$ be an evolutionary system satisfying A1. Due to Theorems 3.5 and 4.3 in \cite{CL14}, the weak global attractor $\Aw$ and the  weak trajectory attractor $\mathfrak A_{\w}$ for $\Dc$ exist, and satisfy
\begin{equation}\label{e:strucofwta}\mathfrak A_{\w}=\Pi_+\bar{\Dc}((-\infty, \infty)):=\{u(\cdot)|_{[0,\infty)}:u\in \bar{\Dc}((-\infty, \infty))\},
\end{equation}
 and
\begin{equation}\label{e:strucofwga}\Aw=\mathfrak A_{\w}(t):=\{u(t): u\in \mathfrak A_{\w}\},\quad \forall\, t\geq 0.
\end{equation}
Here, $\bar{\Dc}$ is the closure of $\Dc$. It follows from Lemma \ref{l:wta=wutp} that $\mathfrak A_{\w}$ satisfies the weak uniform tracking property for $\Dc$, that is, for any $\epsilon >0$, there exists $t_0$, such that for any $t^*>t_0$, every trajectory $u \in
\Dc([0,\infty))$ satisfies
\begin{equation*}
\dd_{C([t^*,\infty);\Xw)} (u(\cdot), v(\cdot-t^*)) < \epsilon,
\end{equation*} for some trajectory $v \in \mathfrak A_{\w}$. In view of \eqref{e:strucofwta}, there is $\tilde{v}(\cdot)\in \bar{\Dc}((-\infty, \infty))$ such that,
\[
\tilde{v}(\cdot)|_{[0,\infty)}=v.
\]
Especially, we have
\begin{equation*}
\dd_{C([t^*,\infty);\Xw)} (u(\cdot), \tilde{v}(\cdot-t^*)) < \epsilon.
\end{equation*}
Hence, we may conveniently call that $ \bar{\Dc}((-\infty, \infty))$ or $\Dc$ itself satisfies the weak uniform tracking property. By \eqref{e:strucofwta} and \eqref{e:strucofwga}, we know that, for any $\tilde{v}(\cdot)\in \bar{\Dc}((-\infty, \infty))$,
$$\{\tilde{v}(t):t\in \mathbb R\}\subset \Aw.$$
Then, we may also call that the weak global attractor $\Aw$ satisfies the weak uniform tracking property.

Contrarily, suppose that, for any $\epsilon >0$, there exists $t_0$, such that for any $t^*>t_0$, every trajectory $u \in
\Dc([0,\infty))$ satisfies
\begin{equation*}
\dd_{C([t^*,\infty);\Xw)} (u(\cdot), \hat{v}(\cdot-t^*)) < \epsilon,
\end{equation*} for some complete trajectory $\hat{v}$ on $\Aw$. Then, by Definition \ref{d:wutp}, $\Pi_+\bar{\Dc}((-\infty, \infty))$ satisfies the weak uniform tracking property. It is deduced from Lemma \ref{l:wta=wutp} that it is a weak trajectory attracting set for $\Dc$. Obviously, $\Pi_+\bar{\Dc}((-\infty, \infty))$ is invariant and closed in  $C([0,\infty);\Xw)$. As shown in the second part of the proof of Theorem 4.3 in \cite{CL14}, it is also  a minimal weak trajectory attracting set for $\Dc$. That is,  $\Pi_+\bar{\Dc}((-\infty, \infty))$ is a weak trajectory attractor for $\Dc$. Therefore, we have the following equivalence.

\begin{lemma}\label{l:wgawutp=wta}Let $\Dc$ be an evolutionary system satisfying A1. The weak global attractor $\Aw$ for $\Dc$ satisfies  the weak uniform tracking property if and only if  $\Dc$ possesses a weak trajectory attractor $\mathfrak A_{\w}$.
\end{lemma}

Since the evolutionary system $\Dc$ satisfies A1,  its weak trajectory attractor is always weakly compact, that is, it is compact in $C([0,\infty);\Xw) $. Then, we introduce the following definition.
\begin{definition}\label{d:fwutp}
A set $P\subset C([0,\infty);\Xw)$  satisfies the finite weak uniform tracking property for an evolutionary system $\Dc$ if for any $\epsilon >0$, there exist $t_0$ and a finite subset $P^f\subset P$, such that for any $t^*>t_0$, every trajectory $u \in
\Dc([0,\infty))$ satisfies
\begin{equation*} \label{e:fwutp}
\dd_{C([t^*,\infty);\Xw)} (u(\cdot), v(\cdot-t^*)) < \epsilon,
\end{equation*} for some trajectory $v \in P^f$.
\end{definition}
With the new perspective of above Lemma \ref{l:wgawutp=wta},  Theorems 3.5 and 4.3 in \cite{CL14} can be restated in the following  more concise form.
\begin{theorem} \label{t:wA}
Let $\Dc$ be an evolutionary system. Then \begin{itemize}\item
[1.] The weak global attractor $\Aw$ exists, and $ \Aw =\ww(X)$.
\end{itemize}
Furthermore, assume that  $\Dc$ satisfies A1. Let $\bar{\Dc}$ be
the closure of $\Dc$. Then
\begin{itemize}
\item [2.]$ \Aw =\ww(X)=\bar{\omega}_{\mathrm w}(X)=\bar{\omega}_{\mathrm s}(X)=\bar{\mathcal A}_{\mathrm w}$.
\item [3.]$\Aw$ is the maximal
invariant and maximal quasi-invariant set w.r.t. $\bar{\Dc}$:
\[\Aw=\left\{ u_0 \in X:  u_0=u(0) \mbox{ for some } u \in
\bar{\Dc}((-\infty, \infty))\right\}.\]
\item [4.] The weak trajectory attractor $\mathfrak A_{\w}$ exists, it is weakly compact, and
\[\mathfrak A_{\w}=\Pi_+\bar{\Dc}((-\infty, \infty)):=\left\{u(\cdot)|_{[0,\infty)}:u\in \bar{\Dc}((-\infty, \infty))\right\}. \]
Hence,  $\mathfrak A_{\w}$ satisfies the finite weak uniform tracking property for $\Dc$ and is weakly equicontinuous on $[0,\infty)$.
\item[5.] $\Aw$ is a section of $\mathfrak A_{\w}$:
\[\Aw=\mathfrak A_{\w}(t):=\left\{u(t): u\in \mathfrak A_{\w}\right\},\quad \forall\, t\geq 0.\]
\end{itemize}
\end{theorem}
\begin{proof}
The conclusions 1-3 are the corresponding results of  Theorem 3.5 in \cite{CL14}. The first part of the conclusion 4 and the conclusion 5 are just Theorem 4.3 in  \cite{CL14}.  The conclusion 4 in Theorem 3.5 in \cite{CL14} is incorporated in the existence of $\mathfrak A_{\w}$ due to Lemmas \ref{l:=wutp} and \ref{l:wgawutp=wta}. In other words, $\mathfrak A_{\w}$ satisfies the weak uniform tracking property for $\Dc$. Thus, for any $\epsilon >0$, there exists $t_0$, such that for any $t^*>t_0$, every trajectory $u \in
\Dc([0,\infty))$ satisfies
\begin{equation} \label{ie:wutpwta}
\dd_{C([t^*,\infty);\Xw)} (u(\cdot), v^*(\cdot-t^*)) < \epsilon/2,
\end{equation} for some trajectory $v ^*\in \mathfrak A_{\w}$. Thanks  to the assumption A1,  $\mathfrak A_{\w}$ is weakly compact.
Hence, we can take a set $P^f$  consisting of a finite number of trajectories, $$P^f:=\{u_1,u_2,\cdots,u_K\}\subset \mathfrak A_{\mathrm{w}},$$
such  that, for any $v \in \mathfrak A_{\mathrm{w}}$,
$$\dd_{C([0,\infty);\Xw)} (v,u_i) < \epsilon/2,$$
for some $u_i\in P^f$.
Then,
 there exists some $u_j(\cdot)\in P^f$ satisfying
$$\dd_{C([0,\infty);\Xw)} (v^*,u_j) < \epsilon/2,$$
which deduces that
$$\dd_{C([t^*,\infty);\Xw)} (v^*(\cdot-t^*),u_j(\cdot-t^*)) < \epsilon/2.$$
Combining with (\ref{ie:wutpwta}), it implies that
$$\dd_{C([t^*,\infty);\Xw)} (u(\cdot),u_j(\cdot-t^*)) < \epsilon,$$
for some  trajectory $u_j \in P^f$. This means that $\mathfrak A_{\mathrm{s}}$ satisfies the finite weak  uniform tracking property for $\Dc$.

Denote by
$$\mathfrak A_{\w}|_{[\alpha,\beta]}:=\left\{u(\cdot)|_{[\alpha,\beta]}:u\in \mathfrak A_{\w}\right\},\quad\forall \beta> \alpha\ge 0.$$
 Note that $\mathfrak A_{\w}|_{[0,1]}$
 is compact in  $C ([0,1];\Xw)$. It is deduced from the Arzel\`{a}-Ascoli compactness criterion that $\mathfrak A_{\w}|_{[0,1]}$ is weakly equicontinuous on $[0,1]$. By the invariance of $\mathfrak A_{\w}$, we have
  $$\left\{v(\cdot+\alpha):v(\cdot)\in \mathfrak A_{\w}|_{[\alpha,\alpha+1]}\right\}=\mathfrak A_{\w}|_{[0,1]}, \  \forall \alpha\ge 0.$$
  Thus, $\mathfrak A_{\w}$ is weakly equicontinuous on $[0,\infty)$.
\end{proof}
Accordingly, we can call that $\Aw$, or $\bar{\Dc}((-\infty, \infty))$, or $\Dc$ satisfies the finite weak uniform tracking property.

\subsection{Strongly compact strong trajectory attractor} \label{ss:SAttractES}Such a form of Theorem \ref{t:wA} is indeed inspirational. We will establish a version of it in the strong metric in this subsection. We begin with the following.

\begin{lemma} \label{l:stas-is-wtas}
A strong trajectory attracting set for an evolutionary system $\Dc$ is a weak trajectory attracting set for $\Dc$.
\end{lemma}
\begin{proof}
Let $P$  is a strong trajectory attracting set for an evolutionary system $\Dc$. Suppose that it is not a weak trajectory attracting set for $\Dc$. Then, there exist $\epsilon_0>0$, and  sequences $t_n\rightarrow \infty$ as $n\rightarrow \infty$, $u_n\in \Dc([0,\infty))$, such that
\begin{equation}\label{i:staiswta1}
\dd_{ C([0,\infty);\Xw)}(P,T(t_n)u_n)>3\epsilon_0.
\end{equation}
Let $l_0\in \mathbb N$ satisfy
\[ \sum_{l>l_0} \frac{1}{2^l}\leq \epsilon_0. \]
By the definition of the metric $\dd_{ C([0,\infty);\Xw)}$, we obtain from (\ref{i:staiswta1}) that
 \begin{equation*}
  \dd_{C([0, l_0];X_\w)}(P,T(t_n)u_n)\sum_{l\leq l_0} \frac{1}{2^l} +\sum_{l>l_0} \frac{1}{2^l}>3\epsilon_0, \ \ \forall\, n\in \mathbb N,
\end{equation*}
which yields that
 \begin{equation}\label{i:staiswta2}
  \dd_{C([0, l_0];X_\w)}(P,T(t_n)u_n)>2\epsilon_0, \ \ \forall\, n\in \mathbb N.
 \end{equation}

 On the other hand,  since $P$  is a strong trajectory attracting set, we have that
 \begin{equation*}
 \lim_{n\rightarrow \infty}   \inf_{v\in P}\sup_{s\in [0,l_0]}\ds(v(s),(T(t_n)u_n)(s))=0.
 \end{equation*}
Passing to  a subsequence and dropping a subindex, we can assume that there exists a sequence
$\{v_n\} \subset P$, such that
\begin{equation}\label{e:staiswta3}
 \lim_{n\rightarrow \infty}  \sup_{s\in [0,l_0]}\ds(v_n(s),(T(t_n)u_n)(s))=0.
 \end{equation}
 Thanks to (\ref{i:staiswta2}), there exists  a sequence  $\{s_n\}\subset [0,l_0]$, such that
 \begin{equation}\label{i:staiswta3}
\dd_\w(v_n(s_n),(T(t_n)u_n)(s_n))>\epsilon_0, \ \ \forall\, n\in \mathbb N.
  \end{equation}
  However, it follows from (\ref{e:staiswta3}) that
 \begin{equation*}
 \lim_{n\rightarrow \infty} \dd_\s(v_n(s_n),(T(t_n)u_n)(s_n))=0,
 \end{equation*}
 which implies that
 \begin{equation*}
 \lim_{n\rightarrow \infty} \dd_\w(v_n(s_n),(T(t_n)u_n)(s_n))=0.
 \end{equation*}
 This contradicts to (\ref{i:staiswta3}). We complete the proof.
\end{proof}

According to Definitions \ref{d:wta}, \ref{d:stas} and Lemma \ref{l:stas-is-wtas}, we naturally define a strong trajectory attractor as well as a strongly compact one.
\begin{definition}\label{d:sta}
A set $\mathfrak A_{\s}\subset C([0,\infty);\Xw) $ is a strong trajectory attractor for an evolutionary system $\Dc$ if $\mathfrak A_{\s}$ is a minimal  strong trajectory attracting set that is
\begin{enumerate}
\item Closed in  $ C([0,\infty);\Xw)$.
\item Invariant: $T(t)\mathfrak A_{\s}=\mathfrak A_{\s}$,  $\forall\, t\geq 0$.
\end{enumerate}
It is said that $\mathfrak A_{\s}$ is strongly compact if it is compact in $C([0,\infty);\Xs) $.
\end{definition}

 Hence, if a strong trajectory attractor exists,  it is unique. Moreover, such a definition means that, a strong trajectory attractor is a weak trajectory attractor whenever it is  also a strong trajectory attracting set.

We establish the following definition and lemmas that can be viewed as versions of Definition \ref{d:wutp} and Lemmas \ref{l:=wutp} and \ref{l:wta=wutp} in the strong metric, respectively.
\begin{definition}\label{d:sutp}
A set $P\subset C([0,\infty);\Xw)$  satisfies the strong uniform tracking property  for an evolutionary system $\Dc$  if for any $\epsilon >0$ and $T>0$, there exists $t_0$, such
that for any $t^*>t_0$, every trajectory $u \in \Dc([0,\infty))$
satisfies
\begin{equation*}\label{e:sutp}
\ds(u(t), v(t-t^*)) < \epsilon, \quad\forall t\in [t^*,t^*+T],
\end{equation*}
for some  $T$-time length piece $v \in P_T$. Here $P_T:=\{v(\cdot)|_{[0,T]}: v\in P\}$.
\end{definition}
\begin{lemma}\label{l:=sutp}
Let  $P\subset C([0,\infty);\Xw)$ satisfy $T(s)P=P$,  $\forall\, s\geq 0$.  Then $P$ satisfies the strong uniform tracking property for an evolutionary system $\Dc$ if and only if for any $\epsilon >0$ and $T>0$, there exists $t_0$, such
that for any $t^*>t_0$, every trajectory $u \in \Dc([0,\infty))$
satisfies
\begin{equation*}\label{e:sutp}
\ds(u(t), v(t)) < \epsilon, \quad\forall t\in [t^*,t^*+T],
\end{equation*}
for some  trajectory $v \in P$.
\end{lemma}
\begin{proof}The proof is analogous to that of Lemma \ref{l:=wutp}. Note that, by definition, $v\in P_T$ if and only if there is $\tilde{v}\in P$ such that $\tilde{v}|_{[0,T]}=v$.
Hence, due to the assumption, for any $t^*\geq 0$, $v\in P_T$ if and only if there exists $v^*\in P$ that $T(t^*)v^*=\tilde{v}$. Then, for any $\epsilon >0$, $T>0$  and $u \in
\Dc([0,\infty))$,
\begin{equation*}
\ds(u(t), v(t-t^*)) < \epsilon, \quad\forall t\in [t^*,t^*+T],
\end{equation*}for some $T$-time length piece $v \in P_T$ is equivalent to
\begin{equation*}
\ds(u(t), v^*(t)) < \epsilon, \quad\forall t\in [t^*,t^*+T],
\end{equation*}for some trajectory $v^* \in P$.
\end{proof}
Similarly, the latter property in  Lemma \ref{l:=sutp} is called strong uniform tracking property in \cite{C09, CL09, CL14}. However, we will soon see  that, it is convenient to  substitute Definition \ref{d:sutp}.

\begin{lemma}\label{l:sta=sutp}
A set $P\subset C([0,\infty);\Xw)$ satisfies  the strong uniform tracking property for an evolutionary system $\Dc$ if and only if it is  a strong trajectory attracting set for $\Dc$.
\end{lemma}
\begin{proof}Assume that $P$ is  a strong trajectory attracting set for an evolutionary system $\Dc$. Then, for any $\epsilon>0$  and $T>0$, there exists $t_0$,
such that
\begin{equation*}\label{e:staing}
T(t)\Dc([0,\infty))\subset \left\{v\in C([0,\infty);\Xw):\inf_{u\in P}\sup_{s\in [0,T]}\ds(u(s),v(s))<\frac{\epsilon}{2}\right\}, \  \forall\, t\geq t_0.
\end{equation*}
Hence, for any $t^*\geq t_0$ and every trajectory $u\in \Dc([0,\infty))$, we have
\begin{equation}\label{e:staing2}
\sup_{s\in [0,T]}\ds(\tilde{v}(s) ,(T(t^*)u)(s) )<\epsilon,
\end{equation}
for some $\tilde{v}\in P$. Thanks to (\ref{e:dTO}) of the definition of the family of translation operators $\{T(s)\}_{s\geq 0}$,  (\ref{e:staing2}) is equivalent to
\begin{equation}\label{e:staing3}
\sup_{s\in [0,T]}\ds(v(s) ,u(s+t^*) )<\epsilon,
\end{equation}
with $v=\tilde{v}|_{[0,T]}$.
 By a change of the variable, (\ref{e:staing3}) is
\begin{equation}\label{e:staing4}
\sup_{s\in [t^*,t^*+T]}\ds(v(s-t^*) ,u(s) )<\epsilon.
\end{equation}
Therefore, $P$ satisfies the strong uniform tracking property.

On the contrary, suppose that for any $\epsilon >0$ and $T>0$, there exists $t_0$, such
that for any $t^*>t_0$, every trajectory $u \in \Dc([0,\infty))$ satisfies the inequality
 \eqref{e:staing4} for some $v\in P_T$. Equivalently, \eqref{e:staing2} holds for some $\tilde{v}\in P$ that $\tilde{v}|_{[0,T]}=v$. Hence, we have,
 \begin{equation*}\label{e:staing}
T(t)\Dc([0,\infty))\subset \left\{v\in C([0,\infty);\Xw):\inf_{u\in P}\sup_{s\in [0,T]}\ds(u(s),v(s))<\epsilon \right\}, \forall\, t> t_0,
\end{equation*}
which means that  $P$ is  a strong trajectory attracting set.
\end{proof}

Due to Lemma \ref{l:sta=sutp}, a strong trajectory attractor is the minimal set that is  invariant, closed in $C([0,\infty);\Xw) $  and satisfying the strong uniform tracking property.

Now we arrive at one of the main theoretical results of the paper, which improves  Theorems 3.6 and 4.4 in \cite{CL14} by obtaining the strong compactness of strong trajectory attractors and its corollaries without additional condition.
\begin{theorem} \label{t:sA}
Let $\Dc$ be an asymptotically compact evolutionary system.
Then
\begin{enumerate}
\item[1.]
The strong global attractor $\As$ exists, it is strongly compact, and
$\As=\Aw$.
\end{enumerate}
Furthermore, assume that $\Dc$ satisfies A1. Let $\bar{\Dc}$ be
the closure of $\Dc$. Then
\begin{enumerate}
\item[2.]
The strong trajectory attractor $\mathfrak A_{\mathrm{s}}$ exists and $\mathfrak A_{\mathrm{s}}=\mathfrak A_{\mathrm{w}}=\Pi_+\bar{\Dc}((-\infty, \infty))$, it is  strongly compact.
\item[3.]$\mathfrak A_{\mathrm{s}}$ satisfies the finite strong uniform tracking property for $\Dc$, i.e., for any $\epsilon >0$ and $T>0$, there exist $t_0$ and a finite subset $P^f_T\subset \mathfrak A_{\s}|_{[0,T]}$, such
that for any $t^*>t_0$, every trajectory $u \in \Dc([0,\infty))$
satisfies
$$\ds(u(t), v(t-t^*)) < \epsilon,\quad\forall t\in [t^*,t^*+T],$$
for some $T$-time length piece  $v\in P^f_T$.
\item[4.]$\mathfrak A_{\s}=\Pi_+\bar{\Dc}((-\infty, \infty))$  is strongly equicontinuous on $[0,\infty)$, i.e.,
\[
\ds\left(v(t_1),v(t_2)\right)\leq \theta\left(|t_1-t_2|\right), \quad\forall\, t_1,t_2\ge 0, \  \forall v\in \mathfrak A_{\s},
\]
where $\theta(l) $ is a positive function tending
to $0$ as $l\rightarrow 0^+$.
\end{enumerate}
\end{theorem}
\begin{proof}The conclusion 1 is that of Theorem 3.6 in \cite{CL14}.

Due to Theorem 4.4 in \cite{CL14} and Definition \ref{d:stas}, the weak trajectory attractor  $\mathfrak A_{\mathrm{w}}$ is a  strong trajectory attracting set that is  invariant and compact in $C([0,\infty);\Xw) $.  Suppose that $P\subset C([0,\infty);\Xw) $ is any other strong trajectory attracting set being invariant and closed in $C([0,\infty);\Xw) $. We know from  Lemma \ref{l:stas-is-wtas} that $P$ is also a weak trajectory attracting set. Hence, $\mathfrak A_{\mathrm{w}}\subset P $. This concludes that $\mathfrak A_{\mathrm{w}}$ is indeed a strong trajectory attractor $\mathfrak A_{\mathrm{s}}$ according to Definition \ref{d:sta}.

Now we demonstrate  the  compactness of $\mathfrak A_{\mathrm{s}}$ in $C([0,\infty);\Xs) $.

First, we have  $\mathfrak A_{\mathrm{s}}\subset C([0,\infty);\Xs)$. In fact, thanks to Theorem \ref{t:wA}, for every $u\in  \mathfrak A_{\mathrm{s}}$, the set
$$\{u(t): t\in [0,\infty )\}\subset \As,$$
is precompact in $\Xs$. Hence, for every $\tilde{t}\in [0 , \infty)$, any weakly convergent sequence $\{u(t_n):t_n\geq 0\}$  with the limit $u(\tilde{t})$ as $t_n\rightarrow \tilde{t}$ does  strongly converge to $u(\tilde{t}) $, which means   $u\in C([0,\infty);\Xs)$.

Note that  $\mathfrak A_{\mathrm{s}}$ is compact in  $C([0,\infty);\Xw) $.  Now take a sequence $\{u_n(t)\}\subset \mathfrak A_{\mathrm{s}}$ and  $u(t)\in \mathfrak A_{\mathrm{s}}$ that  $u_n(t)\rightarrow u(t)$ in $C([0,\infty);\Xw) $ as $n\rightarrow \infty$. We claim that the convergence is indeed in $C([0,\infty);\Xs) $. Otherwise, there exist $\epsilon_0>0$, $T_0>0$, and  sequences $\{n_j\}$, $n_j\rightarrow \infty$ as $j \rightarrow \infty$ and $\{t_{n_j}\}\subset [0,T_0]$, such that
\begin{equation}\label{e:dsbigdist}
\ds\left(u_{n_j}(t_{n_j}), u(t_{n_j})\right)>\epsilon_0, \quad\forall n_j.
\end{equation}
The sequences
$$\left\{u_{n_j}(t_{n_j})\right\}, \left\{u(t_{n_j})\right\}\subset \As,$$
 are relatively strongly compact  due to the strong compactness of $\As$. Passing to a subsequence and dropping a subindex,  we may assume  that  $\left\{u_{n_j}(t_{n_j})\right\}$ and $\left\{u(t_{n_j})\right\}$ are strongly convergent with limits $x$ and $y$, respectively. We have that
\begin{equation}\label{i:sA1}
\dw(x,y)\leq \dw\left(u_{n_j}(t_{n_j}), x\right)+\dw\left(u_{n_j}(t_{n_j}),u(t_{n_j})\right)+\dw\left(u(t_{n_j}),y\right),\  \forall n_j.
\end{equation}
By the assumption,
\begin{equation*}\label{lim:wdist}
\lim_{j\rightarrow \infty} \sup_{t\in [0,T_0]} \dw \left(u_{n_j}(t), u(t)\right)=0.
\end{equation*}
Together with
\begin{equation*}
\lim_{j\rightarrow \infty} \dw\left(u_{n_j}(t_{n_j}), x\right)=0,
\end{equation*}
and
\begin{equation*}
\lim_{j\rightarrow \infty} \dw\left(u(t_{n_j}),y\right)=0,
\end{equation*}
it follows from (\ref{i:sA1}) that $x=y$, which is a contradiction to (\ref{e:dsbigdist}).

 By Lemma \ref{l:sta=sutp}, $\mathfrak A_{\mathrm{s}}$  possesses the strong  uniform tracking property for $\Dc$.
Therefore, for any $\epsilon >0$ and $T>0$, there exists $t_0$, such
that for any $t^*>t_0$, every trajectory $u \in \Dc([0,\infty))$
satisfies
\begin{equation}\label{ie:stupSTA}
\ds\left(u(t), v^*(t-t^*)\right) < \epsilon/2, \quad\forall t\in [t^*,t^*+T],
\end{equation}
for some  $T$-time length piece $v^* \in \mathfrak A_{\mathrm{s}}|_{[0,T]}$. Now since $\mathfrak A_{\mathrm{s}}|_{[0,T]}$ is compact in  $C([0,T];\Xs) $, we can take a finite number set, $$P^f_T:=\left\{u_1(t),u_2(t),\cdots,u_K(t)\right\}\subset \mathfrak A_{\mathrm{s}}|_{[0,T]},$$ such  that, for any $v \in \mathfrak A_{\mathrm{s}}|_{[0,T]}$,
\[
\ds\left(u_i(t), v(t)\right) < \epsilon/2, \quad\forall t\in [0,T],
\]
for some $u_i(t) \in P^f_T$.
Hence, there exists some $u_j(t)\in P_T^f$ satisfying
\[
\ds\left(u_j(t), v^*(t)\right) < \epsilon/2, \quad\forall t\in [0,T],
\]
which is equivalent to
\[
\ds\left(u_j(t-t^*), v^*(t-t^*)\right) < \epsilon/2, \quad\forall t\in [t^*,t^*+T].
\]
Together with (\ref{ie:stupSTA}), it implies that
\[
\ds\left(u(t), u_j(t-t^*)\right) < \epsilon, \quad\forall t\in [t^*,t^*+T],
\]
for some $T$-time length piece $u_j(t) \in P^f_T$. That is, $\mathfrak A_{\mathrm{s}}$ satisfies the finite strong  uniform tracking property for $\Dc$.

Finally, we show that $\mathfrak A_{\s}=\Pi_+\bar{\Dc}((-\infty, \infty))$  is strongly equicontinuous on $[0,\infty)$.
Without loss of generality, we assume that $|t_1-t_2|\leq 1$. Hence, $t_1$ and $t_2$ belong to some interval $[\gamma, \gamma+2]$, $\gamma\ge 0$. Denote by
\[
\Pi_{[\alpha,\beta]}\bar{\Dc}((-\infty, \infty)):=\left\{u(\cdot)|_{[\alpha,\beta]}:u\in \bar{\Dc}((-\infty, \infty))\right\}.
\]
Notice that
\[
\left\{v(\cdot+\gamma):v(\cdot)\in \Pi_{[\gamma,\gamma+2]}\bar{\Dc}((-\infty, \infty))\right\}=\Pi_{[0,2]}\bar{\Dc}((-\infty, \infty)).
\]Thus, we need only to verify that $\Pi_{[0,2]}\bar{\Dc}((-\infty, \infty))$
is strongly equicontinuous on $[0,2]$. Thanks to the Arzel\`{a}-Ascoli compactness criterion, this is a consequence of  the fact that $\Pi_{[0,2]}\bar{\Dc}((-\infty, \infty))$ is compact in $C ([0,2];\Xs)$.

The proof is complete.
\end{proof}

Similarly, according to the above theorem and Theorem \ref{t:wA}, it is also convenient to call that the global attractor $\As$, or $\bar{\Dc}((-\infty, \infty))$, or $\Dc$ possesses a finite strong uniform tracking
property. That is, for any fixed accuracy $\epsilon$ and time length $T$, a finite number of $T$-time length pieces on $[0,T]$ of the complete trajectories on $\As$ are capable of uniformly approximating all trajectories within the accuracy $\epsilon$ in the strong metric after sufficiently large time. By applying Theorem \ref{t:sA} repeatedly, we have the following two corollaries that indicate how the dynamics on the global attractor determine  the long-time dynamics of all trajectories of an evolutionary system.
\begin{corollary}\label{t:fsutpdetermine1} Let $\Dc$ be an asymptotically compact evolutionary system satisfying A1 and let $\bar{\Dc}$ be the closure of $\Dc$. Then, for any $\epsilon >0$ and $T>0$, there exist $j_0\in \mathbb N$ and a finite number of  $T$-time length pieces of the complete trajectories $\bar{\Dc}((-\infty, \infty))$ on the global attractor for  $\Dc$:
$$P_T^f:=\left\{u_1(t),u_2(t),\cdots,u_K(t)\right\}\subset \mathfrak A_{\mathrm{s}}|_{[0,T]},$$
such
that, for every trajectory $u \in \Dc([0,\infty))$, there is a sequence $\{i_j\}_{j\geq j_0}$ with $i_j\in  \{1, 2, \cdots, K\}$ that satisfies
$$\ds\left(u(t), u_{i_j}(t-jT)\right) < \epsilon,\quad\forall t\in [jT,(j+1)T],$$
for $j\geq j_0$ and $u_{i_j}\in P_T^f$.
\end{corollary}
\begin{proof}By Theorem \ref{t:sA},
for any $\epsilon>0$ and $T>0$,  there exist $t_0$ and a finite subset $P^f_T\subset \mathfrak A_{\s}|_{[0,T]}$ consisting of $K$ elements,
$$P_T^f:=\left\{u_1(t),u_2(t),\cdots,u_K(t)\right\},$$
such that for any $t^*>t_0$, every trajectory $u \in \Dc([0,\infty))$
satisfies
$$\ds\left(u(t), u_i(t-t^*)\right) < \epsilon,\quad\forall t\in [t^*,t^*+T],$$
for some  $T$-time length piece  $u_i\in P^f_T$. Then, we obtain the conclusion if we apply this result successively by taking $t^*=j_0T>t_0$ and $t^*=jT$, $j>j_0$.
\end{proof}
It is interesting to note that the number $K$ of the candidate approximating pieces $P^f_T$ depends only on the accuracy $\epsilon$ and the time length $T$, and then,  for any fixed accuracy and time length, every trajectory is assigned a sequence of elements in the finite set $\{1, 2, \cdots, K\}$.
\begin{corollary}\label{t:fsutpdetermine2}Let $\Dc$ be an asymptotically compact evolutionary system satisfying A1 and let $\bar{\Dc}$ be the closure of $\Dc$. Give two sequences $\{\epsilon_n\}$ and $\{T_n\}$ satisfying
$$\epsilon_1>\epsilon_2>\cdots>\epsilon_n>\cdots \rightarrow 0,\quad \mbox{ as } n\rightarrow \infty,$$
and
$$0<T_1<T_2<\cdots<T_n<\cdots \rightarrow \infty, \quad \mbox{ as } n\rightarrow \infty.$$
Then, there exist a time $t_0$, a sequence $\{J_n\}$ and a series of finite sets $P^f_{T_n}$ consisting of $T_n$-time length pieces of the complete trajectories $\bar{\Dc}((-\infty, \infty))$ on the global attractor for $\Dc$:
$$P^f_{T_n}:=\left\{u^{n}_1(t),u^{n}_2(t),\cdots,u^{n}_{K_n}(t)\right\}\subset \mathfrak A_{\mathrm{s}}|_{[0,T_n]},$$
such that, for every trajectory $u \in \Dc([0,\infty))$, there is a sequence
$$i^1_1,i^1_2,\cdots, i^1_{J_1}, i^2_1,i^2_2,\cdots, i^2_{J_2}, \cdots,i^n_1,i^n_2,\cdots, i^n_{J_n}, \cdots,$$
with $i^n_j\in  \{1, 2, \cdots, K_n\}$, $1\leq j\leq J_n$ that satisfies $u^n_{i^n_j}\in P_{T_n}^f$ and
\begin{align*}
&\ds\left(u(t), u^n_{i^n_j}\left(t-\left(t_0+\sum\nolimits_{l=1}^{n-1}J_lT_l+(j-1)T_n\right)\right)\right) < \epsilon_n,\\
&\qquad\qquad\forall t\in \left[t_0+\sum\nolimits_{l=1}^{n-1}J_lT_l+(j-1)T_n,t_0+\sum\nolimits_{l=1}^{n-1}J_lT_l+jT_n\right].
\end{align*}
\end{corollary}
\begin{proof}Due to Theorem \ref{t:sA}, for any $\epsilon_n>0$, and $T_n>0$, there exist $t_n$ and
a finite subset $P^f_{T_n}\subset \mathfrak A_{\s}|_{[0,T_n]}$ consisting of $K_n$ elements,  $$P^f_{T_n}:=\left\{u^{n}_1(t),u^{n}_2(t),\cdots,u^{n}_{K_n}(t)\right\},$$
such that, for any $t^*_n>t_n$,  every trajectory $u \in \Dc([0,\infty))$,  satisfies
$$\ds\left(u(t), v(t-t^*_n)\right) < \epsilon,\quad\forall t\in \left[t^*_n,t^*_n+T\right],$$
for some  $T_n$-time length piece $v\in P_{T_n}^f$.  Without loss of generality, we take $\{t_n\}$ satisfying
$t_2-t_1>1$ and $t_{n+1}-t_n> T_{n-1}$, $n\ge 2$.

Now we are going to determine   $t_0$ and the sequences $\{J_n\}$ and $ \{i^n_j\}$.  We process inductively.

Let $$t_0:=t_1+1,$$ and let
$$ J_1:=\left[\frac{t_2-t_0}{T_1}\right]+1.$$
Here $[\cdot]$ is the greatest integer function. Then, $t_0+J_1T_1\in (t_2,t_3)$. We apply the previous result in the first paragraph $J_1$ times with $\epsilon_1$ and $T_1$, and gain that, for every trajectory $u \in \Dc([0,\infty))$,
\begin{align*}
&\ds\left(u(t), u^1_{i^1_j}\left(t-\left(t_0+(j-1)T_1\right)\right)\right) < \epsilon_1,\\&\qquad\qquad\qquad\qquad\forall t\in \left[t_0+(j-1)T_1,t_0+jT_1\right],
\end{align*}for $i^1_j\in  \{1, 2, \cdots, K_1\}$, $1\leq j\leq J_1$ and $u^1_{i^1_j}\in P_{T_1}^f$. Next, let
 $$ J_2:=\left[\frac{t_3-\left(t_0+J_1T_1\right)}{T_2}\right]+1.$$
 We apply the previous result in the first paragraph $J_2$ times with $\epsilon_2$ and $T_2$, and have that,
\begin{align*}
&\ds\left(u(t), u^2_{i^2_j}\left(t-\left(t_0+J_1T_1+(j-1)T_2\right)\right)\right) < \epsilon_2,\\
&\qquad\qquad\qquad\forall t\in \left[t_0+J_1T_1+(j-1)T_2,t_0+J_1T_1+jT_2\right],
\end{align*}for $i^2_j\in  \{1, 2, \cdots, K_2\}$, $1\leq j\leq J_2$ and $u^2_{i^2_j}\in P_{T_2}^f$. Note that, $t_0+\sum_{l=1}^{2}J_lT_l\in (t_3,t_4)$.

Suppose we have obtained $\{J_l\}$ and $\{i^l_j\} $, $1\leq l\leq n$, $1\leq j\leq J_l$, that
$$J_l:=\left[\frac{t_{l+1}-\left(t_0+\sum_{m=1}^{l-1}J_mT_m\right)}{T_l}\right]+1,$$
and
$$i^l_j\in  \left\{1, 2, \cdots, K_l\right\}, \quad u^l_{i^l_j}\in P^f_{T_l},$$
 satisfying
 \begin{align*}
&\ds\left(u(t), u^l_{i^l_j}\left(t-\left(t_0+\sum\nolimits_{m=1}^{l-1}J_mT_m+(j-1)T_l\right)\right)\right) < \epsilon_l,\\
&\qquad\quad\forall t\in \left[t_0+\sum\nolimits_{m=1}^{l-1}J_mT_m+(j-1)T_l,t_0+\sum\nolimits_{m=1}^{l-1}J_mT_m+jT_l\right].
\end{align*}
We know that $t_0+\sum_{l=1}^{n}J_lT_l \in (t_{n+1}, t_{n+2})$. Now take
$$J_{n+1}:=\left[\frac{t_{n+2}-\left(t_0+\sum_{l=1}^{n}J_lT_l\right)}{T_{n+1}}\right]+1.$$
 We can consecutively apply the previous conclusion in the first paragraph  $J_{n+1}$ times with $\epsilon_{n+1}$ and $T_{n+1}$, and  obtain that,
\begin{align*}
&\ds\left(u(t), u^{n+1}_{i^{n+1}_j}\left(t-\left(t_0+\sum\nolimits_{l=1}^{n}J_lT_l+(j-1)T_{n+1}\right)\right)\right) < \epsilon_{n+1},\\
&\qquad\quad\forall t\in \left[t_0+\sum\nolimits_{l=1}^{n}J_lT_l+(j-1)T_{n+1},t_0+\sum\nolimits_{l=1}^{n}J_lT_l+jT_{n+1}\right],
\end{align*}
for $i^{n+1}_j\in  \{1, 2, \cdots, K_{n+1}\}$, $1\leq j\leq J_{n+1}$  and $u^{n+1}_{i^{n+1}_j}\in P_{T_{n+1}}^f$.

 The proof is completed.
\end{proof}
Now we give a property of the complete trajectories $\bar{\Dc}((-\infty, \infty))$ on $\As$.
\begin{corollary}\label{t:fsutpdetermine3} Let $\Dc$ be an asymptotically compact evolutionary system satisfying A1 and let $\bar{\Dc}$ be
the closure of $\Dc$. Then,  every complete trajectory $u(t)\in \bar{\Dc}((-\infty, \infty))$ is translation compact in $C((-\infty, \infty);\Xs)$, i.e., the set $$\overline{\{u(\cdot+h):h\in \mathbb R\}}^{C((-\infty, \infty);\Xs)}$$ is compact  in $C((-\infty, \infty);\Xs)$.
\end{corollary}
\begin{proof}
It follows from Theorem \ref{t:sA} that,  the set
$\{u(t): t\in (-\infty,\infty )\}\subset \As$
is precompact in $\Xs$, and $u(t)$ is uniformly strongly continuous on $(-\infty,\infty )$. Then the conclusion follows from Proposition V.2.2 in \cite{CV02}.
\end{proof}
The complete trajectories being periodic, quasi-periodic, almost periodic, homoclinic and heteroclinic on $\As$ are translation compact in $C((-\infty, \infty);\Xs)$. See \cite{CV02} for more details.
\begin{remark}\label{r:newpov}
Such forms of Theorems \ref{t:wA} and \ref{t:sA} suggest the following comments.
\begin{enumerate}
\item[1.] Theorem \ref{t:sA} indicates that the notion of a strongly compact strong trajectory attractor is an apt  description of the strongly compact strong global attractor possessing the  finite strong uniform tracking property and the  strong equicontinuity of all the complete trajectories on it.
\item[2.] Comparing with Theorem \ref{t:wA}, Theorem \ref{t:sA} implies that both the strong compactness of  the strong global attractor and the strong trajectory attractor follow simultaneously once we obtain the asymptotical compactness of an evolutionary system.
\item[3.] Theorems \ref{t:wA} and \ref{t:sA} show that the global attractor is a section of  the trajectory attractor and the trajectory attractor consists of the restriction of all the complete trajectories on the global attractor on time semiaxis $[0,\infty)$; the notion of a global attractor stresses the property of attracting trajectories staring from sets in phase space $X$ while the notion of a trajectory attractor emphasizes the uniform tracking  property.
\end{enumerate}
\end{remark}

In fact, the asymptotical compactness of $\Dc$ is also a necessary condition.
\begin{theorem} \label{t:suffience and necessity}
An evolutionary system $\Dc$ is asymptotically compact if and only if its strongly compact strong global attractor $\As$ exists.
\end{theorem}

\begin{proof}The sufficiency is just the conclusion 1 of Theorem \ref{t:sA}. We now prove the necessity. Take any sequences $t_n\rightarrow+\infty $ as $n\rightarrow \infty$ and $x_n\in R(t_n)X$. Then, for any positive integer $k$, by Definition \ref{d:attractor}, there exist $x_{n_k} $ and $y_k$ such that, $y_k\in \As $ and $$\ds(x_{n_k},y_k)<\frac{1}{k}.$$  Since $\As$ is strongly compact, passing to a subsequence and dropping a subindex, we may assume that the sequence $\{y_k\}$ strongly converges with a limit $y$. It follows that
$$\ds(x_{n_k},y)\le\ds(x_{n_k},y_k)+\ds(y_k,y), \quad\forall k,$$which means that  the subsequence $\{x_{n_k}\}$ also strongly converges to $y$. Hence, $\{x_{n}\} $ is relatively compact, that is,  $\Dc$ is asymptotically compact.
\end{proof}

\begin{corollary}\label{t:SAtoSTA}
Let $\Dc$ be an evolutionary system satisfying A1 and let $\bar{\Dc}$ be the closure of $\Dc$. If the strongly compact strong global attractor $\As$ for $\Dc$ exists,  then the  strongly compact strong trajectory attractor $\mathfrak A_{\mathrm{s}}$ for $\Dc$ exists. Hence,
\begin{itemize}
\item[1.] $\mathfrak A_{\mathrm{s}}=\Pi_+\bar{\Dc}((-\infty, \infty))$ satisfies the finite strong uniform tracking property for $\Dc$, i.e., for any $\epsilon >0$ and $T>0$, there exist $t_0$ and a finite subset $P^f_T\subset \mathfrak A_{\s}|_{[0,T]}$, such
that for any $t^*>t_0$, every trajectory $u \in \Dc([0,\infty))$
satisfies
$$\ds(u(t), v(t-t^*)) < \epsilon, \quad\forall t\in [t^*,t^*+T],$$
for some $T$-time length piece  $v\in P^f_T$.
\item[2.]$\mathfrak A_{\s}=\Pi_+\bar{\Dc}((-\infty, \infty))$ is strongly equicontinuous on $[0,\infty)$, i.e.,
$$
\ds\left(v(t_1),v(t_2)\right)\leq \theta\left(|t_1-t_2|\right), \quad\forall\, t_1,t_2\ge 0, \  \forall v\in \mathfrak A_{\s},
$$
where $\theta(l) $ is a positive function tending
to $0$ as $l\rightarrow 0^+$.

\end{itemize}
\end{corollary}

\begin{remark}\label{r:SAtoSTA}
With this corollary in hand, we are able to apply our theory to the systems for which the existence of strongly compact strong global attractors have been proved. After reformulate them by our framework of evolutionary systems,  the existence  of strongly compact strong trajectory attractors follow immediately, once A1 is verified. We will see this idea in the next sections where we apply our theory to the 2D Navier-Stokes equations and a general dissipative reaction-diffusion system.
\end{remark}

\subsection{Kernel of evolutionary system}\label{ss:Ker}In this subsection, we investigate further the kernels of evolutionary systems, especially of closed evolutionary systems and evolutionary systems with uniqueness.

\subsubsection{Closed evolutionary system} We introduce a closed evolutionary system $\Dc_\Sigma$ with symbol space $\Sigma$ and study its properties.
\begin{definition}\label{d:ESclosed}An evolutionary system $\Dc_{\Sigma}$ is (weakly) closed if for any $\tau \in \mathbb R$,   $u_n\in \Dc_{\sigma_n}([\tau,\infty))$,
the convergences $u_n\rightarrow u$ in $C([\tau, \infty);\Xw)$ and $\sigma_n\rightarrow
\sigma$ in some topological space $\Im$ as $n\rightarrow \infty$ imply $u\in \Dc_{\sigma}([\tau,\infty))$.
\end{definition}

\begin{lemma}\label{t:NonempfESsig}
Let $\Dc_{\Sigma}$ be a closed evolutionary system satisfying A1. Then, $\mathcal E_{\sigma} ((-\infty,\infty))$ is nonempty for any $\sigma\in \Sigma$.
\end{lemma}
\begin{proof}
Fix $\sigma\in \Sigma$. Since for any  $\tau \in \mathbb R$,   $\mathcal E_{\sigma} ([\tau,\infty))\neq \emptyset$, we are able to take a sequence $u_{n}(t)\in \mathcal E_{\sigma} ([t_n,\infty))$, where $t_n$ is decreasing and $t_n\rightarrow -\infty$ as $n\rightarrow \infty$. Thanks to A1 and the definition of $\Dc_{\Sigma}$, we have that, for every fixed $t_m$, $\Dc_{\Sigma}([t_m,\infty)) $ is precompact in $C([t_m, \infty);\Xw)$. Thus, together with the condition that $\Dc_{\Sigma}$ is closed, the sequence $$\left\{u_n(t)|_{[t_1,\infty)}\right\}_{n\ge 1}\subset \Dc_{\sigma}([t_1,\infty)) ,$$ passing to a subsequence and dropping a subindex, converges  in $C([t_1, \infty);\Xw)$  with a limit $u^1(t)\in \mathcal E_{\sigma} ([t_1,\infty))$.
Again, passing to a subsequence and dropping a subindex, the sequence $$\left\{u_n(t)|_{[t_2,\infty)}\right\}_{n\ge 2}\subset \Dc_{\sigma}([t_2,\infty)) ,$$converges in $C([t_2, \infty);\Xw)$ with a limit $u^2(t)\in \mathcal E_{\sigma} ([t_2,\infty))$.
It is easy to see that, $u^1(t)=u^{2}(t)$ on $[t_1,\infty)$.  By a standard diagonalization process, we obtain a subsequence of $u_n(t)$, still denoted by $u_n(t)$,  such that, for every $m\in \mathbb N$,
$$u_n(t)|_{[t_m,\infty)}\rightarrow u^m(t)\ \mbox{ in } C([t_m, \infty);\Xw), \mbox{ as } n\rightarrow \infty,$$
with $u^m(t)\in \mathcal E_{\sigma} ([t_m,\infty))$. Note that, for any $m$, $u^m(t)=u^{m+1}(t)$ on $[t_m,\infty)$. Hence, we can define $u(t)$ on $(-\infty, \infty)$, such that,
$$ u(t):=u^m(t), \quad \forall t\in [t_m,\infty),\  \forall m\in \mathbb N.$$
Obviously, it yields
$$u(t)|_{[\tau,\infty)}\in \mathcal E_{\sigma} ([\tau,\infty)), \quad\forall \tau\in \mathbb R,$$
which means that $u(t)\in \mathcal E_{\sigma} ((-\infty,\infty))$ by Definition \ref{d:NDc}.
\end{proof}
\begin{lemma}\label{t:kerofESSig}
Let $\Im$ be some topological space and
$\Sigma\subset \Im$ be sequentially compact in itself. Let $\Dc_{\Sigma}$ be a closed evolutionary system. Then,
\begin{align*}\mathcal E_{\Sigma} ((-\infty,\infty))=\bigcup_{\sigma\in {\Sigma}}\mathcal E_\sigma ((-\infty,\infty)).\end{align*}
\end{lemma}
\begin{proof}
Thanks to Definition \ref{d:NDc} and the definition of an evolutionary system $\Dc_{\Sigma}$, we have
$$\mathcal E_{\Sigma} ((-\infty,\infty))\supset \bigcup_{\sigma\in {\Sigma}}\mathcal E_\sigma ((-\infty,\infty)).$$
Take
$u(t)\in \mathcal E_{\Sigma} ((-\infty,\infty))$. Denote by
$$u_{n}(t):=u(t)|_{[t_n,\infty)}\in
\mathcal E_{\sigma_n}([t_n,\infty)),$$
where $\sigma_n\in \Sigma$, $t_n$ is decreasing and $t_n\rightarrow -\infty$ as $n\rightarrow \infty$. Since $\Sigma$ is sequentially compact in itself, passing to a subsequence and dropping a subindex, there exists $\sigma\in \Sigma$, such that $\sigma_n\rightarrow \sigma$ in $\Im$  as $n\rightarrow \infty$. Note that,  for every fixed $m$, $$u_{n}(t)|_{[t_m,\infty)}=u(t),\quad \forall n\ge m, t\in [t_m, \infty).$$Hence, by the closedness of $\Dc_\Sigma$,
$$u(t)|_{[t_m,\infty)}\in
\mathcal E_{\sigma}([t_m,\infty)), \quad \forall m\in \mathbb N.$$
This implies that $u(t)\in \mathcal E_{\sigma} ((-\infty,\infty))$ due to Definition \ref{d:NDc}. Therefore,
we have
$$\mathcal E_{\Sigma} ((-\infty,\infty))\subset \bigcup_{\sigma\in {\Sigma}}\mathcal E_\sigma ((-\infty,\infty)).$$
The proof is completed.
\end{proof}

\begin{lemma}\label{t:ClosureofESSig}
Let $\Im$ be some topological space and
$\Sigma\subset \Im$ be sequentially compact in itself. Let $\Dc_{\Sigma}$ be a closed evolutionary system.  Then, for any $\tau\in \mathbb R$, the set
\begin{align*}\mathcal E_{\Sigma} ([\tau,\infty))=\bigcup_{\sigma\in {\Sigma}}\mathcal E_\sigma ([\tau,\infty))\end{align*}
is closed in $C([\tau, \infty);\Xw)$.
\end{lemma}
\begin{proof}
Take a sequence $\{u_{n}(t)\}$ that
$u_{n}(t)\in \mathcal E_{\sigma_n} ([\tau,\infty))$, $\sigma_n\in  \Sigma$ and $u_{n}(t)\rightarrow u(t)$ in $C([\tau, \infty);\Xw)$ as $n\rightarrow \infty$. Due to the sequential compactness of  $\Sigma$, passing to a subsequence and dropping a subindex, there exists $\sigma\in \Sigma$ that is the limit of $\{\sigma_n\}$ in $\Im$.
Since $\Dc_\Sigma$ is closed,  we have $u(t)\in \mathcal E_{\sigma} ([\tau,\infty))$. Hence, $\mathcal E_{\Sigma} ([\tau,\infty))$ is a closed set in $C([\tau, \infty);\Xw)$.
\end{proof}

\subsubsection{Evolutionary system with uniqueness}\label{sss:ESuni}
Let $\Dc_{\bar{\Sigma}}$ be an evolutionary system with symbol space $\bar{\Sigma}$, and let $\Sigma\subset
\bar{\Sigma}$ be such that $T(h)\Sigma=\Sigma$ for all $h\geq0$.
It is easy to check that the subfamily of maps $\{ \Dc_\sigma\}$, $\sigma\in \Sigma$ is also an evolutionary
system, with $\Sigma$ as its symbol space. Then, it is  an evolutionary subsystem of $\Dc_{\bar{\Sigma}}$. Denote it by $\Dc_{\Sigma}$, as in Section \ref{s:ES}.
We suppose in this subsection that $\bar{\Sigma}$ is the sequential closure of $\Sigma$ in some topological space $\Im$.

\begin{theorem}\label{t:A0Aw}Let $\Dc_\Sigma$ be an evolutionary system with uniqueness and  with symbol space $\Sigma$ satisfying A1 and let $\bar{\Dc}_\Sigma$ be the closure of $\Dc_\Sigma$.
Let $\bar{\Sigma}$ be the sequential closure of $\Sigma$ in some topological space
$\Im$ and  $\Dc_{\bar{\Sigma}}\supset\Dc_\Sigma$ be  a closed evolutionary system with uniqueness and with symbol space $\bar{\Sigma}$.
Then, $\Dc_{\bar{\Sigma}}\subset\bar{\Dc}_\Sigma$. Hence,
\begin{itemize}
\item[1.] The three
weak uniform global attractors $\Aw^\Sigma$, $\bar{\mathcal A}^\Sigma_{\mathrm w}$ and $\mathcal A^{\bar{\Sigma}}_{\mathrm{w}}$ for
evolutionary systems $\Dc_\Sigma$, $\bar{\Dc}_\Sigma$ and $\Dc_{\bar{\Sigma}}$,
respectively, exist.
\item[2.]$\Aw^\Sigma$, $\bar{\mathcal A}^\Sigma_{\mathrm w}$ and $\mathcal A^{\bar{\Sigma}}_{\mathrm{w}}$  are  the maximal
invariant and maximal quasi-invariant set w.r.t. $\bar{\Dc}_{\Sigma}$ and satisfy the following:
\begin{equation*}\label{i:A0unistru}
\begin{split}
\Aw^\Sigma&=\bar{\mathcal A}^\Sigma_{\mathrm
w}=\mathcal A^{\bar{\Sigma}}_{\mathrm{w}}=\left\{ u_0: \ u_0=u(0) \mbox{ for some } u \in \bar{\Dc}_\Sigma((-\infty,
\infty))\right\}.
\end{split}
\end{equation*}
\item[3.]The three
weak trajectory attractors $\mathfrak A_{\w}^\Sigma$, $\bar{\mathfrak A}_{\w}^\Sigma$ and $\mathfrak A_{\w}^{\bar{\Sigma}}$ for  $\Dc_\Sigma$, $\bar{\Dc}_\Sigma$ and $\Dc_{\bar{\Sigma}}$,
respectively, exist and satisfy the following:
\begin{equation*}\label{i:A0unistru}
\begin{split}
\mathfrak A_{\w}^\Sigma&=\bar{\mathfrak A}^\Sigma_{\mathrm
w}=\mathfrak A^{\bar{\Sigma}}_{\mathrm{w}}=\Pi_+\bar{\Dc}_{\Sigma}((-\infty, \infty)).\\
\end{split}
\end{equation*}
Hence, the three
weak trajectory attractors satisfy the finite weak uniform tracking property for all the three evolutionary systems and are weakly equicontinuous on $[0,\infty)$.
\item[4.]$\Aw^\Sigma$, $\bar{\mathcal A}^\Sigma_{\mathrm w}$ and $\mathcal A^{\bar{\Sigma}}_{\mathrm{w}}$ are sections of $\mathfrak A_{\w}^\Sigma$, $\bar{\mathfrak A}_{\w}^\Sigma$ and $\mathfrak A_{\w}^{\bar{\Sigma}}$:
\[\Aw^\Sigma=\bar{\mathcal A}^\Sigma_{\mathrm
w}=\mathcal A^{\bar{\Sigma}}_{\mathrm{w}}=\mathfrak A_{\w}^\Sigma(t)=\bar{\mathfrak A}^\Sigma_{\mathrm
w}(t)=\mathfrak A^{\bar{\Sigma}}_{\mathrm{w}}(t),\quad \forall\, t\geq 0.\]
\end{itemize}
Furthermore, assume that $\bar{\Sigma}\subset \Im$ is sequentially compact in itself. Then, $\Dc_{\bar{\Sigma}}=\bar{\Dc}_\Sigma$. Hence,
\begin{itemize}
\item[5.]The following relationships on kernels hold: \begin{align*}\bar{\Dc}_\Sigma((-\infty,\infty))=\mathcal E_{\bar{\Sigma}} ((-\infty,\infty))=\bigcup_{\sigma\in {\bar{\Sigma}}}\mathcal E_\sigma ((-\infty,\infty)),\end{align*}and $\mathcal E_\sigma ((-\infty,\infty))$ is nonempty for any $\sigma\in \bar{\Sigma}$.
\end{itemize}
\end{theorem}
\begin{proof}
For any $\tau\in \mathbb R$ and $u_{\sigma}(t)\in\mathcal E_\sigma ([\tau,\infty))$ with $\sigma\in \bar{\Sigma}\backslash \Sigma$, we can take a sequence $\{\sigma_n\}\subset \Sigma$ such that $\sigma_n \rightarrow \sigma$ in $\Im$ as $n\rightarrow \infty$ since $\bar{\Sigma} $ is the sequential closure
of $\Sigma$. Because of the uniqueness of $\mathcal E_\Sigma$, there is a sequence $\{u_{n}(t)\}\subset \mathcal E_{\sigma_n} ([\tau,\infty))$ with $u_{n}(\tau)=u_{\sigma}(\tau)$. Passing to a subsequence and dropping a subindex, $\{u_{n}(t)\}$ converges to a limit $u(t)$ in $C([\tau, \infty);\Xw)$ due to A1. It follows that the  limit $u(t)\in \mathcal E_\sigma ([\tau,\infty))$  for  $\sigma_n \rightarrow \sigma$  in $\Im$ and $\mathcal E_{\bar{\Sigma}}$ is closed.  Note that  $u(\tau)=u_{\sigma}(\tau)$. By the  uniqueness of $\mathcal E_{\bar{\Sigma}}$, we have $u(t)=u_{\sigma}(t)$, $t\ge \tau$. Therefore, $\Dc_{\bar{\Sigma}}\subset\bar{\Dc}_\Sigma$.

The existence of $\Aw^\Sigma$, $\bar{\mathcal A}^\Sigma_{\mathrm
w}$ and  $\mathcal A^{\bar{\Sigma}}_{\mathrm{w}}$ follow from Theorem \ref{t:wA}.

Obviously, we have,   $\Aw^\Sigma\subset\mathcal A^{\bar{\Sigma}}_{\mathrm{w}}\subset\bar{\mathcal A}^\Sigma_{\mathrm
w}$, for  $\Dc_{\Sigma}\subset\Dc_{\bar{\Sigma}}\subset\bar{\Dc}_\Sigma$. On the other hand,  since $\mathcal E_\Sigma$ satisfies A1, $\Aw^\Sigma=\bar{\mathcal A}^\Sigma_{\mathrm
w}$, which deduces that $\mathcal A^{\bar{\Sigma}}_{\mathrm{w}}=\Aw^\Sigma$. The rest results of the conclusion 2 are also obtained by Theorem \ref{t:wA}.

Note that $\mathcal E_\Sigma$ satisfies A1 implies that $\bar{\Dc}_\Sigma$ satisfies \=A1 and then $\Dc_{\bar{\Sigma}}$ satisfies A1.  Thus, the conclusions 3 and 4 follow again from Theorem \ref{t:wA} in a similar way.

Now suppose that $\bar{\Sigma}$ is sequentially compact in $\Im$. Thanks to Lemma \ref{t:ClosureofESSig},  for any $\tau\in \mathbb R$, the set
$\mathcal E_{\bar{\Sigma}} ([\tau,\infty))$ is closed in $C([\tau, \infty);\Xw)$, which implies that $\bar{\Dc}_\Sigma\subset\Dc_{\bar{\Sigma}}$. Hence, $\bar{\Dc}_\Sigma=\Dc_{\bar{\Sigma}}$. The last equality in the conclusion 5 is deduced from Lemma \ref{t:kerofESSig}.  Finally, by Lemma \ref{t:NonempfESsig}, $\mathcal E_\sigma ((-\infty,\infty))$ is nonempty for every $\sigma\in \bar{\Sigma}$.
\end{proof}
\begin{theorem}\label{t:saofESac}
Under the conditions of Theorem \ref{t:A0Aw}, assume that one of the followings is valid:
\begin{itemize}
\item[i).]$\bar{\Dc}_\Sigma$ is asymptotically compact.
\item[ii).] $\Dc_\Sigma$  satisfies A1, A2, and
A3, and $\bar{\Dc}((-\infty,\infty))\subset C((-\infty, \infty);\Xs)$.
\item[iii).]$\bar{\Dc}_\Sigma$ possesses a strongly compact strong global attractor.
\end{itemize}
Then the three
weak uniform global attractors in Theorem \ref{t:A0Aw} are
strongly compact strong uniform global attractors and  the three
weak trajectory attractors are
strongly compact strong trajectory attractors.  Moreover, the three trajectory attractors satisfy the finite
strong uniform tracking property for all the three evolutionary systems and are strongly equicontinuous on $[0,\infty)$.
\end{theorem}
\begin{proof}Assume that i) is valid. Obviously, $\Dc_\Sigma$ and $\Dc_{\bar{\Sigma}}$  are also asymptotically compact.  All the  conclusions follow from Theorem \ref{t:sA}.

Assume that ii) holds. Thanks to Lemmas 3.1 and 3.4 in \cite{CL14}, $\bar{\Dc}_\Sigma$  satisfies \=A1, \=A2, and
\=A3. Therefore, by Theorem \ref{t:A123Comp}, $\bar{\Dc}_\Sigma$ is asymptotically compact. All the conclusions follow from i) just proved.

Now suppose that $\bar{\Dc}_\Sigma$ possesses a strongly compact strong global attractor. Due to Theorem \ref{t:suffience and necessity},  $\bar{\Dc}_\Sigma$ is also asymptotically compact. Then, all the conclusions are obtained.

 We complete the proof.
\end{proof}
\begin{remark}The assumption that  $\bar{\Sigma}\subset\Im$ is the sequential closure of $\Sigma$ was actually used in the proof of Theorem 3.10 in \cite{CL14}. Here is a  remediation for more preciseness. In many applications, the topology of $\Im$ is metrizable on
$\Sigma$. Then, in these cases, the notions of a closure and a sequential closure are equivalent, as well as a compact set and a sequential compact set.
\end{remark}
Theorems \ref{t:A0Aw} and \ref{t:saofESac}  generalize Theorems 3.10-3.13 in \cite{CL14}, and the related results  in \cite{CV02, LWZ05, Lu06, Lu07} where were concerned with the classical cases of a process and a family of processes.

Note that a process and a family of processes always define evolutionary systems with uniqueness, as shown in \cite{CL14}. Let $\{U_\sigma(t,\tau)\}$, $\sigma \in \bar{\Sigma}$ be a family of processes acting on a separable reflexive Banach space $H$: For any $\sigma\in \bar{\Sigma}$, $\tau\in\mathbb{R}$, the following conditions are satisfied:
\begin{equation*}\label{definitionprocesses}\begin{split}
U_{\sigma}(t,\tau)&: H\rightarrow H \text{ is single-valued},\quad \forall\, t\ge \tau,\\
U_{\sigma}(t,s)\circ U_{\sigma}(s,\tau)&=U_{\sigma}(t,\tau),
\quad \forall\, t\ge s\ge\tau,
\\U_{\sigma}(\tau,\tau)&=\text{Identity
operator}.
\end{split}
\end{equation*}
Assume that it satisfies the following translation
identity arising naturally from applications:
\begin{equation*}\label{translation identity}U_{\sigma}(t+h,\tau+h)=U_{T(h)\sigma}(t,\tau),\ \forall\,
\sigma\in\bar{\Sigma},\ t\ge\tau,\ \tau\in\mathbb{R},\ h\ge 0,
\end{equation*}
and $T(s)\bar{\Sigma}=\bar{\Sigma}$, $\forall s\geq 0$. Here $\{T(s)\}_{s\ge 0}$ is the translation semigroup. Assume further that the family of processes is dissipative, i.e., there exists a uniformly (w.r.t. $\sigma\in
\bar{\Sigma}$) absorbing ball $B$:
  For any  $\tau\in\mathbb{R}$ and bounded set $A\subset H$, there exists
$t_0=t_0(A)\geq \tau$, such that,
\begin{equation*}\label{d:absorbingball}\bigcup_{\sigma\in\bar{\Sigma}}\bigcup_{t\geq
t_0}U_{\sigma}(t+\tau,\tau)A\subset B.
\end{equation*}
 Now take $X=B$. Note that since $X$ is a bounded subset of a separable
reflexive Banach space, both the strong and the weak topologies on
$X$ are metrizable. Define a family of maps $\Dc_{\sigma}$, $\sigma\in  \bar{\Sigma}$ in the following way:
$$\Dc_{\sigma}([\tau,\infty)):=\left\{u(\cdot):u(t)=U_{\sigma}(t,\tau)u_{\tau}, u_{\tau}\in X, t\geq\tau\right\}.$$
Conditions 1-4 in Definition \ref{d:NDc} are verified by the definition of the family of processes $\{U_\sigma(t,\tau)\}$, $\sigma\in \bar{\Sigma}$. Hence, the family of processes  $\{U_\sigma(t,\tau)\}$, $\sigma\in \bar{\Sigma}$
 defines an evolutionary system $\Dc_{\bar{\Sigma}}$ with uniqueness and with symbol space $\bar{\Sigma}$.
In applications, $\bar{\Sigma}$ is typically a closure of $\Sigma$
in some appropriate functional space $\Im$, where
$\Sigma:=\{\sigma_0(\cdot+h)|h\in \mathbb R\}$ is the translation
family of a fixed symbol $\sigma_0$.  It can be  checked
\cite{CL14} that the family of processes $\{U_\sigma(t,\tau)\}$,
$\sigma \in \Sigma$, or equivalently, the process $\{U_{\sigma_0}(t,\tau)\}$, also defines  an evolutionary system $\Dc_\Sigma$ with uniqueness and with
symbol space $\Sigma$. It is an evolutionary
subsystem of the former system.

In \cite{LWZ05, Lu06, Lu07}, analogous to the closedness condition, a (weak) continuity condition on  a family of processes was used.  A family of processes
$\{U_\sigma(t,\tau)\}$, $\sigma \in \bar{\Sigma}$  is  $(\Hw\times\bar{\Sigma}, \Hw)$ continuous, if for
any $t\ge\tau$, $\tau\in \mathbb R$, the mapping $(u, \sigma)\rightarrow U_\sigma(t,\tau)u$ is continuous from $ \Hw\times\bar{\Sigma}$ to  $\Hw$. In words of the associated evolutionary system $\Dc_{\bar{\Sigma}}$, we have the following definition.
 \begin{definition}\label{d:XSigmaContinuityofE}An evolutionary system $\mathcal E_{ \bar{\Sigma}}$ with uniqueness and with symbol space $\bar{\Sigma}$ is  $(\Xw\times\bar{\Sigma}, \Xw)$ continuous if for
any $\tau\in \mathbb R$, $\sigma_n, \sigma\in \bar{\Sigma}$, $u_{n}(t)\in \mathcal E_{\sigma_n} ([\tau,\infty))$,  the convergences $u_{n}(\tau)\rightarrow u_\tau$ in $\Xw$ and $\sigma_n \rightarrow\sigma$ in some topological space $\Im$ as $n\rightarrow \infty$ imply $u_{n}(t)\rightarrow u(t)$ in $\Xw$ for all $t\ge\tau$ with a pointwise limit $u(t) \in \mathcal E_{\sigma} ([\tau,\infty))$ and $u(\tau)=u_\tau$.
\end{definition}
Now we have several lemmas about the conditions of  A1, closedness and $(\Xw\times\bar{\Sigma}, \Xw)$ continuity on an evolutionary system with uniqueness.
The first two are obvious.
\begin{lemma}\label{t:closedtocontinuous}
Let $\mathcal E_{ \bar{\Sigma}}$ be an evolutionary system with uniqueness and with symbol space $\bar{\Sigma}$. If $\mathcal E_{ \bar{\Sigma}}$ is $(\Xw\times\bar{\Sigma}, \Xw)$ continuous, then $\mathcal E_{ \bar{\Sigma}}$ is closed.
\end{lemma}
\begin{lemma}\label{t:continuoustoA1}
Let $\Im$ be some topological space and $\bar{\Sigma}\subset\Im$ be sequentially compact in itself. Let $\mathcal E_{ \bar{\Sigma}}$ be an evolutionary system with uniqueness and with symbol space $\bar{\Sigma}$. If $\mathcal E_{ \bar{\Sigma}}$ is $(\Xw\times\bar{\Sigma}, \Xw)$ continuous, then,  for any $\tau\in \mathbb R$, $\sigma_n \in \bar{\Sigma}$, $u_{n}(t)\in \mathcal E_{\sigma_n} ([\tau,\infty))$, there exist  subsequences $\{\sigma_{n_j}\}$, $\{u_{{n_j}}(t)\}$, $\sigma\in  \bar{\Sigma}$ and $ u(t)\in  \mathcal E_{\sigma} ([\tau,\infty))$ such that $u_{{n_j}}(t)\rightarrow u(t)$  in $\Xw$ for all $t\ge\tau$ and $\sigma_{n_j} \rightarrow\sigma$ in $\Im$ as $n_j\rightarrow \infty$.
\end{lemma}
\begin{lemma}\label{t:continuoustoclosed}
Let $\mathcal E_{ \bar{\Sigma}}$ be an evolutionary system with uniqueness and with
symbol space $\bar{\Sigma}$. If $\mathcal E_{ \bar{\Sigma}}$ is closed  and satisfies A1,  then, for
any $\tau\in \mathbb R$, $\sigma_n, \sigma\in \bar{\Sigma}$, $u_{n}(t)\in \mathcal E_{\sigma_n} ([\tau,\infty))$, such that,  $u_{n}(\tau)\rightarrow u_\tau$ in $\Xw$ and $\sigma_n \rightarrow\sigma$ in some topological space $\Im$ as $n\rightarrow \infty$, the convergence $u_{n}(t)\rightarrow u(t)$ in $C([\tau, \infty);\Xw)$  holds with a limit $u(t) \in \mathcal E_{\sigma} ([\tau,\infty))$ and $u(\tau)=u_\tau$.
In particular, $\mathcal E_{ \bar{\Sigma}}$ is  $(\Xw\times\bar{\Sigma}, \Xw)$ continuous.
\end{lemma}
\begin{proof}
 For any $\tau\in \mathbb R$, take $\sigma_n, \sigma\in \bar{\Sigma}$, $u_{n}(t)\in \mathcal E_{\sigma_n} ([\tau,\infty))$, such that, $u_{n}(\tau)\rightarrow u_\tau$ in $\Xw$ and $\sigma_n \rightarrow\sigma$ in $\Im$ as $n\rightarrow \infty$. Thanks to A1, there exists a subsequence $\{u_{{n_j}}(t)\}$ such that
 $$u_{{n_j}}(t)\rightarrow u(t) \mbox{ in } C([\tau, \infty);\Xw), \quad \mbox{as } n_j\rightarrow \infty,$$
 for some $ u(t)\in C([\tau, \infty);\Xw)$ with $u(\tau)=u_\tau$.  Since  $\mathcal E_{ \bar{\Sigma}}$ is closed, we have $u(t)\in \mathcal E_{\sigma} ([\tau,\infty))$. Suppose that
 $$u_{n}(t)\nrightarrow u(t) \mbox{ in } C([\tau, \infty);\Xw), \quad \mbox{as } n\rightarrow \infty.$$
  Thanks to A1 again, there is a subsequence of $\{u_{{n'_j}}(t)\}$,  such that  $$u_{{n'_j}}(t)\rightarrow u^*(t) \mbox{ in } C([\tau, \infty);\Xw), \quad \mbox{as } n'_j\rightarrow \infty,$$ for some $ u^*(t)\in C([\tau, \infty);\Xw)$ with $ u^*(\cdot)\neq u(\cdot)$. Note that  $u^*(t)\in \mathcal E_{\sigma} ([\tau,\infty))$ and $u^*(\tau)=u_{\tau}$. By uniqueness of  $\mathcal E_{ \bar{\Sigma}}$ , $u^*(t)=u(t)$,  for all $t\ge\tau$,  which is a contradiction. Hence,
  $$u_{n}(t)\rightarrow u(t) \mbox{ in } C([\tau, \infty);\Xw), \quad \mbox{as } n\rightarrow \infty.$$
Especially, $u_{n}(t)\rightarrow u(t)$ in $\Xw$ for all $t\ge\tau$, which means that  $\mathcal E_{ \bar{\Sigma}}$ is  $(\Xw\times\bar{\Sigma}, \Xw)$ continuous.
\end{proof}
\begin{remark}The condition A1 generally holds  for dissipative systems  in mathematical physics (cf. e.g. \cite{T88, CV02}). See applications in the next sections.
\end{remark}

\section{Attractors for 2D and 3D Navier-Stokes Equations} \label{s:NSE}

In this section, we apply the new theory established in the previous section to the 2D and the 3D Navier-Stokes equations. Main new results are the existence of strongly compact strong trajectory attractors and sequent properties without further assumptions (cf. \cite{LWZ05, Lu06, CL14}).

More precisely,
we consider the space periodic 2D and 3D incompressible
Navier-Stokes equations (NSE)
\begin{equation} \label{NSE1}
\left\{
\begin{aligned}
&\ddt u - \nu \Delta u + (u \cdot \nabla)u + \nabla p = f(t),\\
&\nabla \cdot u =0,
\end{aligned}
\right.
\end{equation}
where $u$, the velocity, and $p$, the pressure, are unknowns; $f(t)$
is a given driving force, and $\nu>0$ is the kinematic  viscosity
coefficient of the fluid. By a Galilean change of variables, we can
assume that the space average of $u$ is zero, i.e.,
\[
\int_\Omega u(x,t) \, dx =0, \quad \forall t,
\]
where $\Omega=[0,L]^n$, $n=2, 3 $, is a periodic box.\footnote{The no-slip
case can be considered in a similar way, only with some adaption
on the functional  setting.}

First, let us introduce some notations and functional setting.
Denote by $(\cdot,\cdot)$ and $|\cdot|$ the $\left(L^2(\Omega)\right)^n$-inner
product and the corresponding $\left(L^2(\Omega)\right)^n$-norm. Let
$\mathcal{V}$ be the space of all $\mathbb{R}^n$ trigonometric
polynomials of period $L$ in each variable satisfying $\nabla \cdot
u =0$ and $\int_\Omega u(x) \, dx =0$. Let $H$ and $V$ to be the
closures of $\mathcal{V}$ in $\left(L^2(\Omega)\right)^n$ and $\left(H^1(\Omega)\right)^n$,
respectively. Define the strong and weak distances by
\[
\ds(u,v):=|u-v|, \quad
\dw(u,v):= \sum_{\kappa \in \mathbb{Z}^n} \frac{1}{2^{|\kappa|}}
\frac{|u_{\kappa}-v_{\kappa}|}{1 + |u_{\kappa}-v_{\kappa}|},
\quad u,v \in H,
\]
where $u_{\kappa}$ and $v_{\kappa}$ are Fourier coefficients of $u$
and $v$ respectively. Note that the weak metric $\dw$ induces the
weak topology in any ball in $\left(L^2(\Omega)\right)^n$.

Let also  $P_{\sigma} : \left(L^2(\Omega)\right)^n \to H$ be the $L^2$-orthogonal
projection, referred to as the Leray projector. Denote by
$A=-P_{\sigma}\Delta = -\Delta$ the Stokes operator with the domain
$D(A)=(H^2(\Omega))^n \cap V$. The Stokes operator is a self-adjoint
positive operator with a compact inverse.
Let
\[
\|u\| := |A^{1/2} u|,
\]
which is called the enstrophy norm.
Note that $\|u\|$ is equivalent to the $H^1$-norm of $u$ for $u\in D(A^{1/2})$. The corresponding inner product is denoted by
\[
((u,v)):=(A^{1/2} u,A^{1/2} v).
\]

Let $V'$ be the dual of $V$. Now denote $B(u,v):=P_{\sigma}(u \cdot \nabla v)\in V'$ for all
$u, v \in V$. This bilinear form has the following property:
\[
\langle B(u,v),w\rangle=-\langle B(u,w),v\rangle, \qquad u,v,w \in V,
\]
in particular, $\langle B(u,v),v\rangle=0$ for all $u,v \in V$.

Now we can rewrite (\ref{NSE1}) as the following differential equation in $V'$:
\begin{equation*} \label{NSE}
\ddt u + \nu A u +B(u,u) = g,\\
\end{equation*}
where $u$ is a $V$-valued function of time and $g = P_{\sigma} f$.

\begin{definition}
A weak solution  of  \eqref{NSE1} on $[\tau,\infty)$ (or $(-\infty, \infty)$, if
$\tau=-\infty$) is an $H$-valued
function $u(t)$ defined for $t \in [\tau, \infty)$, such that
\[
\ddt u \in L_{\mathrm{loc}}^1([\tau, \infty); V'), \quad
u(t) \in C([\tau, \infty); \Hw) \cap
L_{\mathrm{loc}}^2([\tau, \infty); V),
\]
and
\begin{equation*}\label{IntNSE}
\left(u(t)-u(t_0), v\right) = \int_{t_0}^t \left( -\nu ((u, v)) - \langle B(u,u),v\rangle +\langle g, v\rangle \right) \, ds,
\end{equation*}
for all $v \in V$ and  $\tau\leq t_0 \leq t $.
\end{definition}

\begin{theorem}[Leray, Hopf] \label{t:Leray}
For every $u_0 \in H$ and $g \in L^2_\mathrm{loc}(\mathbb{R};V')$, there exists a
weak solution of (\ref{NSE1}) on $[\tau,\infty)$ with $u(\tau)=u_0$
satisfying the following energy inequality
\begin{equation} \label{EI}
|u(t)|^2 + 2\nu \int_{t_0}^t \|u(s)\|^2 \, ds \leq
|u(t_0)|^2 + 2\int_{t_0}^t \langle g(s), u(s)\rangle \, ds,
\end{equation}
for all $t \geq t_0$, $t_0$ a.e. in $[\tau,\infty)$.
\end{theorem}

\begin{definition} \label{d:ex}
A Leray-Hopf solution of \eqref{NSE1} on the interval $[\tau, \infty)$
is a weak solution on $[\tau,\infty)$ satisfying the
energy inequality (\ref{EI}) for all $\tau \leq t_0 \leq t$,
$t_0$ a.e. in $[\tau,\infty)$. The set $Ex$ of measure $0$ on which the energy
inequality does not hold will be called the exceptional set.
\end{definition}

Now fix an external  force $g_0$  that is   translation bounded in
$L^2_\mathrm{loc}(\mathbb{R};V')$ , i.e.,
\[
\left\| g_0\right\| _{L_{\mathrm{b}}^2\left( \Bbb{R};V'\right) }^2 :=\sup_{t \in \mathbb{R}} \int_t^{t+1}
\|g_0(s)\|_{V'}^2 \, ds < \infty.
\]
Denote by $L^{2,\mathrm{w}}_\mathrm{loc}(\mathbb{R};V')$ the space $L^{2}_\mathrm{loc}(\mathbb{R};V')$ endowed with the local weak convergence topology. Then $g_0$ is translation compact in
$L^{2,\mathrm{w}}_\mathrm{loc}(\mathbb{R};V')$, i.e., the
translation family of $g_0$,
\[
\Sigma:=\{g_0(\cdot+h): h\in \mathbb R\},
\]
is precompact in
$L^{2,\mathrm{w}}_\mathrm{loc}(\mathbb{R};V')$ (see \cite{CV02}).  Note that,
\begin{equation}\label{i:gNSE}
\|g\|^2_{L_{\mathrm{b}}^2\left( \Bbb{R};V'\right) }\leq
\|g_0\|^2_{L_{\mathrm{b}}^2\left( \Bbb{R};V'\right) },\quad \forall\, g\in
\Sigma.
\end{equation}
Due to the energy inequality \eqref{EI}, we have
\begin{equation}\label{i:EIgNSE}
|u(t)|^2 + \nu \int_{t_0}^t \|u(s)\|^2 \, ds \leq |u(t_0)|^2 +
\frac{1}{\nu}\int_{t_0}^t \|g(s)\|^2_{V'} \, ds, \quad \forall g
\in \Sigma,
\end{equation}
for all $t \geq t_0$, $t_0$ a.e. in $[\tau,\infty)$. Here $u(t)$ is a Leray-Hopf solution of  \eqref{NSE1} with the force
$g$ on $[\tau,\infty)$.  By Gr\"{o}nwall's inequality, there
exists a uniformly (w.r.t. $\tau\in \mathbb R$ and $g\in \Sigma$) absorbing ball $B_{\mathrm{s}}(0, R)\subset H$, where the radius
$R$ depends on $L$, $\nu$, and $\|g_0\|^2_{L_{\mathrm{b}}^2\left( \Bbb{R};V'\right) }$.
 Let $X$ be a closed uniformly absorbing ball
\begin{equation}\label{i:absorbongB3}
X= \{u\in H: |u| \leq R\},
\end{equation}
which is also weakly compact in $H$. Then, for any bounded set $A \subset
H$, there exists a time $\bar{t}\ge 0$ independent of the initial time $\tau$, such that
\begin{equation}\label{i:absorbongBtrajectory3}
u(t) \in X, \quad \forall t\geq t_1:=\tau+\bar{t},
\end{equation}
for every Leray-Hopf solution $u(t)$ with the force $g\in \Sigma$
and the initial data $u(\tau) \in A$.  For any sequence of
Leray-Hopf solutions $u_k$, the following result (see e.g. \cite{T88, CF89, Ro01, CL14}) holds.

\begin{lemma} \label{l:precompactofLH23}
Let $u_k(t)$ be a sequence of Leray-Hopf solutions of  \eqref{NSE1} with forces
$g_k \in \Sigma$, such that $u_k(t) \in X$ for all $t\geq t_1$. Then
\begin{equation*}\label{bdofun}
\begin{split}
u_k \ \ &\mbox{is bounded in} \ \ L^2(t_1,t_2;V)\ \ \mbox{and} \ \ L^\infty(t_1,t_2;H),\\
\ddt u_k \ \  &\mbox{is bounded in} \ \ L^{p}(t_1,t_2;V'),
\end{split}
\end{equation*}
for all $t_2>t_1$, with $p=2$ if $n=2$ and $p=4/3$ if $n=3$. Moreover, there exists a subsequence $u_{k_j}$ converges to some solution $u(t)$ in $C([t_1, t_2]; \Hw)$,
i.e.,
\[
(u_{k_j},v) \to (u,v) \quad
\mbox{uniformly on }   [t_1,t_2],
\]
as $k_j\to \infty$, for all $v \in H$.
\end{lemma}

\begin{remark}\cite{CL14} \label{r:limitLH}
In the nonautonomous case, i.e., $f(t)$ is dependent on $t$, we don't know here whether the limit $u(t)$ is a Leray-Hopf solution yet when $n=3$.
\end{remark}
Consider an evolutionary system for which a family of trajectories
consists of all Leray-Hopf solutions of the 2D or the 3D NSE with a fixed force $g_0$ in $X$. More precisely, define
\[
\begin{split}
\Dc([\tau,\infty)) := \{&u(\cdot): u(\cdot)
\mbox{ is a Leray-Hopf}
\mbox{ solution on } [\tau,\infty)\\
&\mbox{with the force } g\in \Sigma \mbox{ and } u(t) \in X, \
\forall t \in [\tau,\infty)\}, \   \tau \in \mathbb{R},
\end{split}
\]
\[
\begin{split}
\Dc((-\infty,\infty)) := \{&u(\cdot): u(\cdot)
\mbox{ is a Leray-Hopf} \ \mbox{ solution on } (-\infty,\infty)\\
&\mbox{with the force } g\in \Sigma \mbox{ and } u(t) \in X, \
\forall t \in (-\infty,\infty)\}.
\end{split}
\]

Clearly, the properties 1-4 of an  evolutionary system hold for $\Dc$, if we utilize the
translation semigroup $\{T(s)\}_{s\ge 0}$. Therefore, thanks to
Theorem~\ref{t:wA}, the weak uniform global attractor $\Aw$ for this
evolutionary system exists.

 Now  we give the definition of a normal function which was introduced
in \cite{LWZ05, Lu06}.
\begin{definition}\label{d:normal} Let $\mathcal B$ be a Banach space. A function $\varphi (s)\in L^2_{\mathrm{loc}}(\mathbb{R};\mathcal B
)$ is said to be normal in $L^2_{\mathrm{loc}}(\mathbb{R};\mathcal B
)$ if for any $\epsilon>0$, there exists $\delta>0$, such that
\[
\sup_{t\in \mathbb{R}} \int_{t}^{t+\delta}
\|\varphi(s)\|^2_{\mathcal B} \, ds \leq  \epsilon.
\]
\end{definition}Note that the class of normal
functions is a proper closed subspace of
the class of translation bounded functions (see \cite{LWZ05, Lu06} for more details). Then, we have the
following.
\begin{lemma} \label{l:compact0}
The evolutionary system $\Dc$ of the 2D or the 3D NSE with the force $g_0$
satisfies A1 and A3. Moreover, if $g_0$ is normal in
$L^2_{\mathrm{loc}}(\mathbb{R}; V')$ then A2 holds.
\end{lemma}

\begin{proof}
For the 3D case, it is just Lemma 5.7 in \cite{CL14}. For the 2D case, it is derived in exactly the same way.
\end{proof}

\subsection{3D Navier-Stokes equations}\label{ss:3NSE}
Now applying the theory in Section \ref{s:AttractES}, we have the followings.
\begin{theorem}\label{t:Aw03DNSE} \cite{CL14} Let $g_0$  be  translation bounded in
$L^2_\mathrm{loc}(\mathbb{R};V')$. Then the weak uniform global attractor $\Aw$  and the weak trajectory attractor $\mathfrak A_{\w}$ for the 3D NSE with the fixed force $g_0$ exist, $\Aw$ is the maximal
invariant and maximal quasi-invariant set w.r.t. the closure $\bar{\Dc}$ of the corresponding evolutionary system $\Dc$ and
\begin{equation*}
\begin{split}
\Aw &= \ww(X)=\ws(X)=\left\{ u(0):  u \in
\bar{\Dc}((-\infty, \infty))\right\}, \\ \mathfrak A_{\w}&=\Pi_+\bar{\Dc}((-\infty, \infty))=\left\{u(\cdot)|_{[0,\infty)}:u\in \bar{\Dc}((-\infty, \infty))\right\}, \\
\Aw&=\mathfrak A_{\w}(t)=\left\{u(t): u\in \mathfrak A_{\w}\right\},\quad \forall\, t\geq 0.
\end{split}
\end{equation*}
Moreover,  $\mathfrak A_{\w}$ satisfies the finite weak uniform tracking property and is weakly equicontinuous on $[0,\infty)$.
\end{theorem}

 \begin{proof} This theorem is just Theorem 5.8 and the first part of Theorem 5.10 in \cite{CL14}. We reformulate these results according to Theorem \ref{t:wA}.
 \end{proof}

\begin{theorem}\label{t:As03DNSE}
If  $g_0$ is normal in $L^2_{\mathrm{loc}}(\mathbb{R}; V')$ and
every complete trajectory of $\bar{\Dc}$ is strongly
continuous, then the weak uniform global attractor $\Aw$ is a strongly compact strong global attractor  $\As$, and  the weak trajectory attractor $\mathfrak A_{\w}$ is a strongly
compact strong trajectory attractor $\mathfrak A_{\s}$. Moreover,
\begin{itemize}
\item[1.] $\mathfrak A_{\s}=\Pi_+\bar{\Dc}((-\infty, \infty))$ satisfies the finite strong
uniform tracking property, i.e.,  for any $\epsilon >0$ and $T>0$, there exist $t_0$ and a finite subset $P_T^f\subset \mathfrak A_{\s}|_{[0,T]}$, such
that for any $t^*>t_0$, every trajectory $u \in \Dc([0,\infty))$
satisfies
$$|u(t)-v(t-t^*)| < \epsilon, \quad \forall t\in [t^*,t^*+T],$$
for some  $T$-time length piece  $v\in P_T^f$.
\item[2.]
$\mathfrak A_{\s}=\Pi_+\bar{\Dc}((-\infty, \infty))$  is strongly equicontinuous on $[0,\infty)$, i.e.,
\[
\left|v(t_1)-v(t_2)\right|\leq \theta\left(|t_1-t_2|\right), \quad\forall\, t_1,t_2\ge 0, \  \forall v\in \mathfrak A_{\s},
\]
where $\theta(l) $ is a positive function tending
to $0$ as $l\rightarrow 0^+$.
\end{itemize}
\end{theorem}

The strong compactness of  $\As$ and the strong trajectory attracting property of  $\mathfrak A_{\w}$ have been proved in \cite{CL14}. Here  novelties are the strong compactness of $\mathfrak A_{\s}$ and sequent properties.

\begin{proof}
The theorem follows by applying  Lemma \ref{l:compact0}, Theorems \ref{t:A123Comp} and \ref{t:sA}.
\end{proof}

\subsubsection{Open problem} We give some supplementaries to Section 6 in \cite{CL14}. In this subsection, we further assume that  $g_0$ is translation compact in
$L^2_\mathrm{loc}(\mathbb{R};V')$, i.e., the closure of the translation family $\Sigma$ of $g_0$ in $L^{2}_\mathrm{loc}(\mathbb{R};V')$,
\[
\bar{\Sigma}:=\overline{\{g_0(\cdot+h): h\in \mathbb R\}}^{L^2_\mathrm{loc}(\mathbb{R};V')},
\]
is compact in
$L^{2}_\mathrm{loc}(\mathbb{R};V')$.  It is known \cite{CV02} that $L^2_\mathrm{loc}(\mathbb{R};V')$ is metrizable and the corresponding metric space is complete. Note that the class of translation compact functions is also a  closed subspace of
the class of translation bounded functions, but it is a
proper subset of the class of normal functions (for more details, see \cite{LWZ05, Lu06}). Note also that the argument of (\ref{i:gNSE})-(\ref{i:absorbongBtrajectory3}) is valid for $\Sigma$ replaced by $\bar{\Sigma}$ and Lemma \ref{l:precompactofLH23} can be improved as follows (see e.g. \cite{T88, CF89, CV02, CL14}).

\begin{lemma} \label{l:compactofLH3}
Let $n=3$ and let $u_k(t)$ be a sequence of Leray-Hopf solutions of  \eqref{NSE1} with forces
$g_k \in \bar{\Sigma}$, such that $u_k(t) \in X$ for all $t\geq t_1$. Then
\begin{equation*}
\begin{split}
u_k \ \ &\mbox{is bounded in} \ \ L^2(t_1,t_2;V)\ \ \mbox{and} \ \ L^\infty(t_1,t_2;H),\\
\ddt u_k \ \  &\mbox{is bounded in} \ \ L^{4/3}(t_1,t_2;V'),
\end{split}
\end{equation*}
for all $t_2>t_1$. Moreover, there exists a subsequence $k_j$, such
that $g_{k_j}$ converges in
$L^{2}_{\mathrm{loc}}(\mathbb{R};V')$  to some  $g \in
\bar{\Sigma}$ and $u_{k_j}$ converges in $C([t_1, t_2]; \Hw)$ to some
Leray-Hopf solution $u(t)$ of  \eqref{NSE1} with the force $g$,
i.e.,
\[
(u_{k_j},v) \to (u,v) \quad
\mbox{uniformly on }  [t_1,t_2],
\]
as $k_j\to \infty$, for all $v \in H$.
\end{lemma}

Now consider another evolutionary system with $\bar{\Sigma}$ as a symbol space.  The family
of trajectories for this evolutionary system  consists of all Leray-Hopf solutions of the family
of 3D NSE with  forces $g\in \bar{\Sigma}$ in $X$:
\[
\begin{split}
\Dc_{\bar{\Sigma}}([\tau,\infty)) := \{&u(\cdot): u(\cdot) \mbox{ is a
Leray-Hopf}
\mbox{ solution on } [\tau,\infty)\\
&\mbox{with the force } g\in \bar{\Sigma} \mbox{ and } u(t) \in X, \
\forall t \in [\tau,\infty)\}, \quad \tau \in \mathbb{R},
\end{split}
\]
\[
\begin{split}
\Dc_{\bar{\Sigma}}((-\infty,\infty)) := \{&u(\cdot): u(\cdot)
\mbox{ is a Leray-Hopf} \ \mbox{ solution on } (-\infty,\infty)\\
&\mbox{with the force } g\in \bar{\Sigma} \mbox{ and } u(t) \in X, \
\forall t \in (-\infty,\infty)\}.
\end{split}
\]
We have following lemmas.
\begin{lemma}\cite{CL14}\label{l:compactNSE3}
The evolutionary system $\Dc_{\bar{\Sigma}}$ of the family of  3D NSE with
forces in $\bar{\Sigma}$ satisfies \={A1}, \={A2} and \={A3}.
\end{lemma}
\begin{lemma}\label{l:closedNSE3}
The evolutionary system $\Dc_{\bar{\Sigma}}$ of the family of  3D NSE with
forces in $\bar{\Sigma}$ is closed.
\end{lemma}
\begin{proof}
For any $\tau\in \mathbb R$, take $g_k\in \bar{\Sigma}$, $u_k\in \Dc_{g_k}([\tau,\infty))$, such that,
$u_k\rightarrow u$ in $C([\tau, \infty);\Xw)$ and $g_k\rightarrow
g$ in $L^{2}_{\mathrm{loc}}(\mathbb{R};V')$.  By Lemma \ref{l:compactofLH3} and a standard diagonalization process, we have $u\in \Dc_{g}([\tau,\infty))$. That is, $\Dc_{\bar{\Sigma}}$ is closed.
\end{proof}
\begin{lemma}\label{l:clousurequalsitselfNSE3}Let $\Dc_{\bar{\Sigma}}$ be the evolutionary system of the family of  3D NSE with
forces in $\bar{\Sigma}$. Then $\bar{\Dc}_{\bar{\Sigma}}=\Dc_{\bar{\Sigma}}$.
\end{lemma}
\begin{proof}By the assumption, $\bar{\Sigma}$ is metrizable and compact in
$L^{2}_\mathrm{loc}(\mathbb{R};V')$. Thanks to Lemma \ref{l:closedNSE3}, the corresponding
evolutionary system $\Dc_{\bar{\Sigma}}$ is closed. It follows from Lemma \ref{t:ClosureofESSig} that, for any $\tau \in \mathbb R$, the set $\mathcal E_{\bar{\Sigma}} ([\tau,\infty))$
is closed in $C([\tau, \infty);\Xw)$. Hence, $\bar{\Dc}_{\bar{\Sigma}}=\Dc_{\bar{\Sigma}}$.
\end{proof}
We have the following (cf. \cite{CV02, CL14}) that now knows more on the structure of the kernel.
\begin{theorem} \label{t:AwNSEtrc}Let $g_0$  be  translation compact in
$L^2_\mathrm{loc}(\mathbb{R};V')$. Then the weak uniform global attractor  $\Aw^{\bar{\Sigma}}$  and the weak trajectory attractor $\mathfrak A^{\bar{\Sigma}}_{\w}$  for the family of 3D
NSE with forces $g\in \bar{\Sigma}$ exist, $\Aw^{\bar{\Sigma}}$ is the maximal invariant
and maximal quasi-invariant set w.r.t.  the corresponding
evolutionary system $\Dc_{\bar{\Sigma}}$, and
\begin{equation*}
\begin{split}
\Aw^{\bar{\Sigma}}  &=\left\{u(0): u \in \Dc_{\bar{\Sigma}}((-\infty, \infty))\right\}= \left\{ u(0): u\in \bigcup_{g\in\bar{\Sigma}}\Dc_g ((-\infty, \infty))\right\}, \\ \mathfrak A^{\bar{\Sigma}}_{\w}&=\Pi_+\bigcup_{g\in\bar{\Sigma}}\Dc_g ((-\infty, \infty)), \\
\Aw^{\bar{\Sigma}}&=\mathfrak A^{\bar{\Sigma}}_{\w}(t)=\left\{u(t): u\in \mathfrak A^{\bar{\Sigma}}_{\w}\right\},\quad \forall\, t\geq 0,
\end{split}
\end{equation*}
where $\mathcal E_{g} ((-\infty,\infty))$ is nonempty for any $g\in \bar{\Sigma}$.
Moreover,  $\mathfrak A^{\bar{\Sigma}}_{\w}$ satisfies the finite weak uniform tracking property and  is weakly equicontinuous on $[0,\infty)$.
\end{theorem}
\begin{proof}
By Lemma \ref{l:clousurequalsitselfNSE3}, $\Dc_{\bar{\Sigma}}$ equals to its closure $\bar{\Dc}_{\bar{\Sigma}}$. Especially, we have $$\Dc_{\bar{\Sigma}}((-\infty, \infty))=\bar{\Dc}_{\bar{\Sigma}}((-\infty, \infty)). $$ It follows from Lemma \ref{t:kerofESSig} that
$$\Dc_{\bar{\Sigma}}((-\infty, \infty))=\bigcup_{g\in\bar{\Sigma}}\Dc_g ((-\infty, \infty)). $$
We know from Lemma \ref{l:compactNSE3} that $\Dc_{\bar{\Sigma}}$ satisfies \=A1. Therefore, $\mathcal E_{g} ((-\infty,\infty))$ is nonempty for any $g\in \bar{\Sigma}$ due to Lemma \ref{t:NonempfESsig}. The rest part of the conclusions follow by applying Theorem \ref{t:wA}, which is in fact
 Theorems 6.3 and 6.6 in \cite{CL14}.
\end{proof}Theorem \ref{t:A123Comp} and Lemma \ref{l:compactNSE3} give a criterion for the strong compactness of the attractors. Thereout,  we would obtain that all the complete Leray-Hopf solutions of the family of 3D NSE with forces $g\in
\bar{\Sigma}$ satisfy the finite strong uniform tracking property and are strongly equicontinuous on $(-\infty,\infty)$.
\begin{theorem}\label{t:AsNSEtrc}
Furthermore, if every complete trajectory of the family of 3D NSE with forces $g\in
\bar{\Sigma}$ is strongly continuous , then the weak uniform global attractor $\Aw^{\bar{\Sigma}}$
is a strongly compact strong global attractor $\As^{\bar{\Sigma}}$ and  the weak trajectory attractor $\mathfrak A^{\bar{\Sigma}}_{\w}$ is a strongly
compact strong trajectory attractor $\mathfrak A^{\bar{\Sigma}}_{\s}$. Moreover, $\mathfrak A^{\bar{\Sigma}}_{\s}$ satisfies the finite
strong uniform tracking property and is strongly equicontinuous on $[0,\infty)$.
\end{theorem}
Let $\bar{\Dc}$ be the closure of the  evolutionary  system $\Dc$.
It follows from Lemma \ref{l:clousurequalsitselfNSE3} that $\Dc\subset\bar{\Dc}\subset\Dc_{\bar{\Sigma}}$. Then, an
interesting problem naturally arises:
\begin{open}\cite{CL14}\label{op:attractor}
Are  the attractors $\Ab$,  $\mathfrak A_{\bullet}$ and
$\Ab^{\bar{\Sigma}}$, $\mathfrak A^{\bar{\Sigma}}_{\bullet}$ in Theorems \ref{t:Aw03DNSE}
and \ref{t:AsNSEtrc} identical?
\end{open}

Due to Theorems \ref{t:A0Aw} and \ref{t:saofESac}, the answer is positive if the solutions of the 3D NSE are unique (cf. footnote \ref{fn:nonuniq}). However, the negative answer would imply that the Leray-Hopf
weak solutions are not unique, and especially that, the attractors $\Ab^{\bar{\Sigma}}$ and $\mathfrak A^{\bar{\Sigma}}_{\bullet}$  for the auxiliary family of 3D NSE with forces in $\bar{\Sigma}$ do not
satisfy the minimality property w.r.t. $\db$-attracting  and $\mathrm d_{ C([0, \infty);X_\bullet)}$-attracting, respectively,  for the evolutionary system $\Dc$ corresponding to the original 3D NSE with the fixed force $g_0$.
For more details, see \cite{CL14}.

\subsection{2D Navier-Stokes equations: Weak solutions}\label{ss:2NSEws} In the 2D case, there are better properties of weak solutions than those in Theorem \ref{t:Leray}. The weak solutions are unique and strongly continuous w.r.t. time $t$,  and the equality in (\ref{EI}) holds for every weak solution  (see e.g. \cite{T88, Ro01, CV02}). We will show in this subsection that these properties provide better results.
\begin{theorem}(Weak solutions)\label{t:Leray2Dws}
Let $n=2$. For every $u_0 \in H$ and $g \in L^2_\mathrm{loc}(\mathbb{R};V')$,  the weak solution of  (\ref{NSE1}) on $[\tau,\infty)$ with $u(\tau)=u_0$ is unique and satisfies
$$u(t)\in C([\tau, \infty);H).  $$
\end{theorem}

Denote again by
$$\bar{\Sigma}:=\overline{\{g_0(\cdot+h): h\in \mathbb R\}}^{L^{2,\mathrm{w}}_\mathrm{loc}(\mathbb{R};V') }.$$
It can be known \cite{CV02} that $\bar{\Sigma}$ endowed with the topology of $L^{2,\mathrm{w}}_\mathrm{loc}(\mathbb{R};V')$ is metrizable and the corresponding metric space is compact.
Due to the better properties of the weak solutions, we have the following better version of Lemma \ref{l:precompactofLH23} (see e.g. \cite{T88, Ro01, CV02}).
\begin{lemma} \label{l:compactofLH2w}
Let $n=2$ and let $u_k(t)$ be a sequence of weak solutions of  \eqref{NSE1} with forces
$g_k \in \bar{\Sigma}$, such that $u_k(t) \in X$ for all $t\geq t_1$. Then
\begin{equation*}
\begin{split}
u_k \ \ &\mbox{is bounded in} \ \ L^2(t_1,t_2;V)\ \ \mbox{and} \ \ L^\infty(t_1,t_2;H),\\
\ddt u_k \ \  &\mbox{is bounded in} \ \ L^{2}(t_1,t_2;V'),
\end{split}
\end{equation*}
for all $t_2>t_1$. Moreover, there exists a subsequence $k_j$, such
that $g_{k_j}$ converges in
$L^{2,\mathrm{w}}_\mathrm{loc}(\mathbb{R};V')$  to some  $g \in
\bar{\Sigma}$ and $u_{k_j}$ converges in $C([t_1, t_2]; \Hw)$ to some
weak solution $u(t)$ of  \eqref{NSE1} with the force $g$,
i.e.,
\[
(u_{k_j},v) \to (u,v) \quad
\mbox{uniformly on }  [t_1,t_2],
\]
as $k_j\to \infty$, for all $v \in H$.
\end{lemma}
Similarly, together with $\Dc$, we can also consider another evolutionary system with $\bar{\Sigma}$ as a symbol space.  The family
of trajectories for this evolutionary system  consists of all weak solutions of the family
of 2D NSE with  forces $g\in \bar{\Sigma}$ in $X$:
\[
\begin{split}
\Dc_{\bar{\Sigma}}([\tau,\infty)) := \{&u(\cdot): u(\cdot) \mbox{ is a
weak }
\mbox{ solution on } [\tau,\infty)\mbox{ with }\\
&\quad\mbox{ the force } g\in \bar{\Sigma} \mbox{ and } u(t) \in X, \
\forall t \in [\tau,\infty)\}, \ \tau \in \mathbb{R},
\end{split}
\]
\[
\begin{split}
\Dc_{\bar{\Sigma}}((-\infty,\infty)) := \{&u(\cdot): u(\cdot)
\mbox{ is a weak} \ \mbox{ solution on } (-\infty,\infty) \mbox{ with }\\
&\ \quad\mbox{ the force } g\in \bar{\Sigma} \mbox{ and } u(t) \in X, \
\forall t \in (-\infty,\infty)\}.
\end{split}
\]
Analogously, we have the following lemma. The proof is exactly the same as that of above Lemma \ref{l:closedNSE3}.
\begin{lemma}\label{l:closedNSE2w}
The evolutionary system $\Dc_{\bar{\Sigma}}$ of the family of  2D NSE with
forces in $\bar{\Sigma}$ is closed.
\end{lemma}

The following theorems recover and generalize the related results in \cite{CV02, Lu06}.
\begin{theorem}\label{t:Aw02DNSEw} Let $g_0$ be translation bounded in $ L^{2}_\mathrm{loc}(\mathbb{R};V')$. The two weak uniform global attractors $\Aw$, $\Aw^{\bar{\Sigma}}\subset \Hw$  and the two weak trajectory attractors $\mathfrak A_{\w}$,  $\mathfrak A^{\bar{\Sigma}}_{\w}\subset C([0,\infty);\Hw)$ for the 2D NSE with the fixed force $g_0$ and for the family of 2D NSE with forces $g\in \bar{\Sigma}$, respectively, exist, $\Aw$ and $\Aw^{\bar{\Sigma}}$ are the  maximal
invariant and maximal quasi-invariant set w.r.t. the closure $\bar{\Dc}=\Dc_{\bar{\Sigma}}$ of the corresponding evolutionary system $\Dc$ and
\begin{equation*}
\begin{split}
\Aw&=\Aw^{\bar{\Sigma}}=\left\{u(0): u \in \Dc_{\bar{\Sigma}}((-\infty, \infty))\right\}= \left\{ u(0): u\in \bigcup_{g\in\bar{\Sigma}}\Dc_g ((-\infty, \infty))\right\}, \\ \mathfrak A_{\w}&=\mathfrak A^{\bar{\Sigma}}_{\w}=\Pi_+\bigcup_{g\in\bar{\Sigma}}\Dc_g ((-\infty, \infty)), \\
\Aw&=\mathfrak A^{\bar{\Sigma}}_{\w}(t)=\left\{u(t): u\in \mathfrak A^{\bar{\Sigma}}_{\w}\right\},\quad \forall\, t\geq 0,
\end{split}
\end{equation*}
where $\mathcal E_{g} ((-\infty,\infty))$ is nonempty for any $g\in \bar{\Sigma}$.
Moreover,  $\mathfrak A_{\w}=\mathfrak A^{\bar{\Sigma}}_{\w}$ satisfies the finite weak uniform tracking property for $\Dc_{\bar{\Sigma}}$ and  is weakly equicontinuous on $[0,\infty)$.
\end{theorem}
\begin{proof}
Thanks to Theorem \ref{t:Leray2Dws}, the evolutionary systems $\Dc$ and $\Dc_{\bar{\Sigma}}$  are unique.  Lemmas \ref{l:compact0} and \ref{l:closedNSE2w} indicate that $\Dc$  satisfies A1 and  that $\Dc_{\bar{\Sigma}}$ is closed, respectively.
Then, the theorem follows by applying Theorem \ref{t:A0Aw}.
\end{proof}
\begin{theorem}\label{t:As02DNSEw}
If  $g_0$ is normal in $L^2_{\mathrm{loc}}(\mathbb{R}; V')$, then the two weak global attractors $\Aw$ and $\Aw^{\bar{\Sigma}}$  are strongly compact strong global attractors  $\As$ and $\As^{\bar{\Sigma}}$ in $H$, and  the two weak trajectory attractors $\mathfrak A_{\w}$ and $\mathfrak A^{\bar{\Sigma}}_{\w}$ are strongly
compact strong trajectory attractors $\mathfrak A_{\s}$ and $\mathfrak A^{\bar{\Sigma}}_{\s}$ in $ C([0,\infty);H)$, respectively. Moreover,
\begin{itemize}
\item[1.] $\mathfrak A_{\s}=\mathfrak A^{\bar{\Sigma}}_{\s}=\Pi_+\bigcup_{g\in\bar{\Sigma}}\Dc_g ((-\infty, \infty))$ satisfies the finite strong
uniform tracking property, i.e.,  for any $\epsilon >0$ and $T>0$, there exist $t_0$ and a finite subset $P_T^f\subset \mathfrak A_{\s}|_{[0,T]}$, such
that for any $t^*>t_0$, every trajectory $u \in \Dc_{\bar{\Sigma}}([0,\infty))$
satisfies
$$|u(t)-v(t-t^*)|< \epsilon,\quad\forall t\in [t^*,t^*+T],$$
for some $T$-time length piece  $v\in P_T^f$.
\item[2.]
 $\mathfrak A_{\s}=\mathfrak A^{\bar{\Sigma}}_{\s}=\Pi_+\bigcup_{g\in\bar{\Sigma}}\Dc_g ((-\infty, \infty))$   is strongly equicontinuous on $[0,\infty)$, i.e.,
\[
\left|v(t_1)-v(t_2)\right|\leq \theta\left(|t_1-t_2|\right), \quad\forall\, t_1,t_2\ge 0, \  \forall v\in \mathfrak A_{\s},
\]
where $\theta(l) $ is a positive function tending
to $0$ as $l\rightarrow 0^+$.
\end{itemize}
\end{theorem}
Part of this theorem was obtained in \cite{Lu06}. Here novelties are the existence of  the strongly compact strong trajectory attractor $\mathfrak A_{\s}$ in $ C([0,\infty);H)$ and its corollaries.
\begin{proof}
We follow the proof of Theorem \ref{t:Aw02DNSEw}. The only thing we need to do is to verify the asymptotical compactness of the evolutionary system $\Dc_{\bar{\Sigma}}$. Note that, according to Lemma \ref{l:compact0}, $\Dc$ also satisfies A2 and A3 when $g_0$ is normal in $L^2_{\mathrm{loc}}(\mathbb{R}; V')$. Moreover, Theorem \ref{t:Leray2Dws} ensures that all the weak solutions of the family of  2D NSE with forces $g$ in $\bar{\Sigma}$ are strongly continuous w.r.t. time $t$. Hence, we obtain the conclusions by applying Theorem \ref{t:saofESac}.
\end{proof}
\begin{remark}\label{r:uniq2DNSE} Let $u_1$ and $u_2$ be the solutions of \eqref{NSE1} with forces $g_1(t)$, $g_2(t)\in \bar{\Sigma}$, respectively. By the standard estimates (see e.g. \cite{T88, Ro01, CV02, Lu06}), we have
\begin{equation}\label{i:uniq2DNSEnonauto}
|u_1(t_2)-u_2(t_2)|^2\leq \left(|u_1(t_1)-u_2(t_1)|^2+\frac{1}{\nu}\int^{t_2}_{t_1}\|g_1-g_2\|^2_{V'}\,ds\right)e^{C_1}.
\end{equation}
Here $C_1$ only depends on $L$, $\nu$, and increasingly on  $t_2-t_1$, $|u_1(t_1)|^2$, and $\|g_1\|^2_{L_{\mathrm{b}}^2\left( \Bbb{R};V'\right) }$. Then,  the uniqueness of the solutions in Theorem \ref{t:Leray2Dws} is deduced from \eqref{i:uniq2DNSEnonauto}, so is their continuous dependence on the initial data and the forces, which informs the continuity of the associated family of processes (see \cite{CV02}). In the  autonomous case, i.e., the force $g(t)$ is independent of time $t$, \eqref{i:uniq2DNSEnonauto} yields the following classical estimates
\begin{equation}\label{i:uniq2DNSEautot}
|u_1(t)-u_2(t)|\leq |u_1(0)-u_2(0)| e^{C_2(t)},
\end{equation}where $C_2(t)$ depends increasingly on $t$. Hence,  we obtain, for any $T>0$,
\begin{equation}\label{i:uniq2DNSEautoT}
|u_1(t)-u_2(t)|\leq |u_1(0)-u_2(0)| e^{C_2(T)},\quad \forall\,t\in [0,T].
\end{equation} Now consider a family of solutions $A_{[0,\infty)}:=\{u(\cdot):u(0)\in A\} $ with initial data in a compact subset $A\subset H$. We can see from \eqref{i:uniq2DNSEautoT} that $A_{[0,T]}:=\left\{u(\cdot)|_{[0,T]}:u\in A_{[0,\infty)}\right\} $  is compact in $C([0,T];H) $. Hence, $A_{[0,\infty)}$ is compact in $C([0,\infty);H) $. Especially, the trajectory attractor $\mathfrak A_{\s}$ in the above theorem is strongly compact. Although simple, to our knowledge, both the deducing process and the compactness result were not noticed before. On the other hand, however, for the nonautonomous case, we can not obtain the similar compactness from \eqref{i:uniq2DNSEnonauto}.
\end{remark}

\subsection{2D Navier-Stokes equations: Strong solutions}\label{ss:2NSEss}Concerning the strong solutions of the 2D NSE, there are similar results obtained by the same method as we do in previous subsection. We will present  the main steps and omit details.

Now in this subsection,  fix a more regular force $g_0$  that is   translation bounded in
$L^2_\mathrm{loc}(\mathbb{R};H)$, i.e.,
\[
\left\| g_0\right\| _{L_{\mathrm{b}}^2\left( \Bbb{R};H\right) }^2 := \sup_{t \in \mathbb{R}} \int_t^{t+1}
|g_0(s)|^2 \, ds < \infty.
\]
Then $g_0$ is translation compact in
$L^{2,\mathrm{w}}_\mathrm{loc}(\mathbb{R};H)$, that is, the translation family of $g_0$,
\[
\Sigma:=\{g_0(\cdot+h): h\in \mathbb R\}
\]
is precompact in
$L^{2,\mathrm{w}}_\mathrm{loc}(\mathbb{R};H)$ (see \cite{CV02}).    Denote again by
$$\bar{\Sigma}:=\overline{\{g_0(\cdot+h): h\in \mathbb R\}}^{L^{2,\mathrm{w}}_\mathrm{loc}(\mathbb{R};H) }.$$
It is known \cite{CV02} that $\bar{\Sigma}$ endowed with the topology of $L^{2,\mathrm{w}}_\mathrm{loc}(\mathbb{R};H)$ is metrizable and the corresponding metric space is compact.

 We have more regular solutions (see e.g. \cite{T88, Ro01, CV02}).
\begin{theorem}(Strong solutions)\label{t:Leray2Dss}
Let $n=2$. For every $u_0 \in V$ and $g \in L^2_\mathrm{loc}(\mathbb{R};H)$,  the solution of  (\ref{NSE1}) on $[\tau,\infty)$ with $u(\tau)=u_0$ is unique and satisfies
$$u(t)\in C([\tau, \infty);V)\cap
L_{\mathrm{loc}}^\infty([\tau, \infty); V)\cap
L_{\mathrm{loc}}^2([\tau, \infty); D(A)).  $$
\end{theorem}
By the classical estimates, there
exists a uniformly (w.r.t. $\tau \in \mathbb R$ and $g\in  \bar{\Sigma}$) absorbing ball $B_{\mathrm{s}}(0, R)\subset V$, where the radius
$R$ depends on $L$, $\nu$, and $\|g_0\|^2_{L_{\mathrm{b}}^2\left( \Bbb{R};H\right) }$.
 Let $X$ be a closed  uniformly  absorbing ball
\begin{equation*}\label{i:absorbongB2s}
X= \{u\in V: \|u\| \leq R\},
\end{equation*}
which is also weakly compact in $V$. Then for any bounded set $A \subset
V$, there exists a time $\bar{t}\ge 0$ independent of $\tau$, such that
\begin{equation*}\label{i:absorbongBtrajectory2s}
u(t) \in X, \quad \forall t\geq t_1:=\tau+\bar{t},
\end{equation*}
for every solution $u(t)$ with the force $g\in \bar{\Sigma}$
and the initial data $u(\tau) \in A$.

Due to the better regularity of the solutions, we have the following better version of Lemma \ref{l:precompactofLH23} (see e.g. \cite{T88, Ro01}).
\begin{lemma} \label{l:compactofLH2s}
Let $n=2$ and let $u_k(t)$ be a sequence of strong solutions of  \eqref{NSE1} with forces
$g_k \in \bar{\Sigma}$, such that $u_k(t) \in X$ for all $t\geq t_1$. Then
\begin{equation*}
\begin{split}
u_k \ \ &\mbox{is bounded in} \ \ L^2(t_1,t_2;D(A))\ \ \mbox{and} \ \ L^\infty(t_1,t_2;V),\\
\ddt u_k \ \  &\mbox{is bounded in} \ \ L^{2}(t_1,t_2;H),
\end{split}
\end{equation*}
for all $t_2>t_1$. Moreover, there exists a subsequence $k_j$, such
that $g_{k_j}$ converges in
$L^{2,\mathrm{w}}_\mathrm{loc}(\mathbb{R};H)$  to some  $g \in
\bar{\Sigma}$ and $u_{k_j}$ converges in $C([t_1, t_2]; V_{\w})$ to some
strong solution $u(t)$ of  \eqref{NSE1} with the force $g$,
i.e.,
\[
((u_{k_j},v)) \to ((u,v)) \quad
\mbox{uniformly on } [t_1,t_2],
\]
as $k_j\to \infty$, for all $v \in D(A)$.
\end{lemma}
Now we consider two evolutionary systems. One for which a family of trajectories
consists of all strong solutions of the 2D  NSE with the fixed force $g_0$ in $X$. More precisely, define
\[
\begin{split}
\Dc([\tau,\infty)) := \{&u(\cdot): u(\cdot)
\mbox{ is a strong}
\mbox{ solution on } [\tau,\infty)\mbox{ with }\\
&\quad\mbox{ the force } g\in \Sigma \mbox{ and } u(t) \in X, \
\forall t \in [\tau,\infty)\}, \   \tau \in \mathbb{R},
\end{split}
\]
\[
\begin{split}
\Dc((-\infty,\infty)) := \{&u(\cdot): u(\cdot)
\mbox{ is a strong} \ \mbox{ solution on } (-\infty,\infty)\mbox{ with }\\
&\quad\mbox{ the force } g\in \Sigma \mbox{ and } u(t) \in X, \
\forall t \in (-\infty,\infty)\}.
\end{split}
\]
Another one we consider is with $\bar{\Sigma}$ as a symbol space.  The family
of trajectories for this evolutionary system  consists of all strong solutions of  the family of 2D NSE with  forces $g\in \bar{\Sigma}$ in $X$:
\[
\begin{split}
\Dc_{\bar{\Sigma}}([\tau,\infty)) := \{&u(\cdot): u(\cdot) \mbox{ is a
strong }
\mbox{ solution on } [\tau,\infty)\mbox{ with }\\
&\quad\mbox{ the force } g\in \bar{\Sigma} \mbox{ and } u(t) \in X, \
\forall t \in [\tau,\infty)\}, \   \tau\in \mathbb{R},
\end{split}
\]
\[
\begin{split}
\Dc_{\bar{\Sigma}}((-\infty,\infty)) := \{&u(\cdot): u(\cdot)
\mbox{ is a strong} \ \mbox{ solution on } (-\infty,\infty)\mbox{ with }\\
&\ \quad\mbox{ the force } g\in \bar{\Sigma} \mbox{ and } u(t) \in X, \
\forall t \in (-\infty,\infty)\}.
\end{split}
\]
By analogous arguments to those of Lemmas \ref{l:compactNSE3} and  \ref{l:closedNSE3}, we have following lemmas.
\begin{lemma}\label{l:A1NSE2s}
The evolutionary system $\Dc$ of the  2D NSE with the fixed
force $g_0$ satisfies A1.
\end{lemma}
\begin{lemma}\label{l:closedNSE2s}
The evolutionary system $\Dc_{\bar{\Sigma}}$ of the family of  2D NSE with
forces in $\bar{\Sigma}$ is closed.
\end{lemma}
Note that the strongly compact strong global attractor $\As^{\bar{\Sigma}}$ for the family of  2D NSE with forces in $\bar{\Sigma}$  had been obtained in \cite{LWZ05}. Now, we have more: The following theorems recover and generalize the related results in \cite{CV02, LWZ05}.

\begin{theorem}\label{t:Aw02DNSEs} Let $g_0$ be translation bounded in $ L^{2}_\mathrm{loc}(\mathbb{R};H)$. The two weak uniform global attractors $\Aw$, $\Aw^{\bar{\Sigma}}\subset V_{\w}$  and the two weak trajectory attractors $\mathfrak A_{\w}$,  $\mathfrak A^{\bar{\Sigma}}_{\w}\subset C([0,\infty);V_{\mathrm{w}})$ for the 2D NSE with the fixed force $g_0$ and  for the family of 2D NSE with forces $g\in \bar{\Sigma}$, respectively, exist, $\Aw$ and $\Aw^{\bar{\Sigma}}$ are  the  maximal
invariant and maximal quasi-invariant set w.r.t. the closure $\bar{\Dc}=\Dc_{\bar{\Sigma}}$ of the corresponding evolutionary system $\Dc$ and
\begin{equation*}
\begin{split}
\Aw&=\Aw^{\bar{\Sigma}}=\left\{u(0): u \in \Dc_{\bar{\Sigma}}((-\infty, \infty))\right\}= \left\{ u(0): u\in \bigcup_{g\in\bar{\Sigma}}\Dc_g ((-\infty, \infty))\right\}, \\ \mathfrak A_{\w}&=\mathfrak A^{\bar{\Sigma}}_{\w}=\Pi_+\bigcup_{g\in\bar{\Sigma}}\Dc_g ((-\infty, \infty)), \\
\Aw&=\mathfrak A^{\bar{\Sigma}}_{\w}(t)=\left\{u(t): u\in \mathfrak A^{\bar{\Sigma}}_{\w}\right\},\quad \forall\, t\geq 0,
\end{split}
\end{equation*}
where $\mathcal E_{g} ((-\infty,\infty))$ is nonempty for any $g\in \bar{\Sigma}$.
Moreover,  $\mathfrak A_{\w}=\mathfrak A^{\bar{\Sigma}}_{\w}$ satisfies the finite weak uniform tracking property for $\Dc_{\bar{\Sigma}}$ and  is weakly equicontinuous on $[0,\infty)$.
\end{theorem}
\begin{proof}The proof is similar to that of Theorem \ref{t:Aw02DNSEw}. Due to Theorem \ref{t:Leray2Dss}, evolutionary systems $\Dc$ and $\Dc_{\bar{\Sigma}}$  are unique.  The facts that $\Dc$  satisfies A1 and  $\Dc_{\bar{\Sigma}}$ is closed are obtained by Lemmas \ref{l:A1NSE2s} and \ref{l:closedNSE2s}, respectively.
Then, the theorem follows by applying Theorem \ref{t:A0Aw} again.
\end{proof}
\begin{theorem}\label{t:As02DNSEs}
If  $g_0$ is normal in $L^2_{\mathrm{loc}}(\mathbb{R};H)$, then the two weak global attractors $\Aw$ and $\Aw^{\bar{\Sigma}}$  are strongly compact strong global attractors  $\As$ and $\As^{\bar{\Sigma}}$ in $V$, and  the two weak trajectory attractors $\mathfrak A_{\w}$ and $\mathfrak A^{\bar{\Sigma}}_{\w}$ are strongly
compact strong trajectory attractors $\mathfrak A_{\s}$ and $\mathfrak A^{\bar{\Sigma}}_{\s}$ in $ C([0,\infty);V)$, respectively. Moreover,
\begin{itemize}
\item[1.] $\mathfrak A_{\s}=\mathfrak A^{\bar{\Sigma}}_{\s}=\Pi_+\bigcup_{g\in\bar{\Sigma}}\Dc_g ((-\infty, \infty))$ satisfies the finite strong
uniform tracking property, i.e.,  for any $\epsilon >0$ and $T>0$, there exist $t_0$ and a finite subset $P_T^f\subset \mathfrak A_{\s}|_{[0,T]}$, such
that for any $t^*>t_0$, every trajectory $u \in \Dc_{\bar{\Sigma}}([0,\infty))$
satisfies
$$\|u(t)-v(t-t^*)\|< \epsilon,\quad\forall t\in [t^*,t^*+T],$$
for some $T$-time length piece  $v\in P^f_T$.
\item[2.]
$\mathfrak A_{\s}=\mathfrak A^{\bar{\Sigma}}_{\s}=\Pi_+\bigcup_{g\in\bar{\Sigma}}\Dc_g ((-\infty, \infty))$  is strongly equicontinuous on $[0,\infty)$, i.e.,
\[
\left\|v(t_1)-v(t_2)\right\|\leq \theta\left(|t_1-t_2|\right), \quad\forall\, t_1,t_2\ge 0, \  \forall v\in \mathfrak A_{\s},
\]
where $\theta(l) $ is a positive function tending
to $0$ as $l\rightarrow 0^+$.
\end{itemize}
\end{theorem}
 Part of this theorem recovers the corresponding results in \cite{LWZ05}. Novelties here are the existence of the strongly compact strong trajectory attractor $\mathfrak A_{\s}$ in $ C([0,\infty);V)$ and sequent properties.
\begin{proof}Analogous to the proof of Theorem \ref{t:As02DNSEw}, we follow the proof of Theorem \ref{t:Aw02DNSEs}. Note that, by Theorem \ref{t:A0Aw},  $\bar{\Dc}_\Sigma=\Dc_{\bar{\Sigma}}$. Theorem 3.2 in \cite{LWZ05} provides the existence of the strongly compact strong global attractor $\As^{\bar{\Sigma}}$ in $V$ for  the family of 2D NSE with forces $g\in \bar{\Sigma}$  when $g_0$ is normal in $L^2_{\mathrm{loc}}(\mathbb{R}; H)$. In other words,  the  evolutionary system $\bar{\Dc}_\Sigma=\Dc_{\bar{\Sigma}}$ possesses a strongly compact strong global attractor $\As^{\bar{\Sigma}}$. Hence, we obtain the conclusions by applying Theorem \ref{t:saofESac} again.
\end{proof}

\section{Attractors for Reaction-Diffusion System} \label{s:RDS}In this section, we study the long-time behavior of solutions of  the following nonautonomous reaction-diffusion system (RDS):
\begin{equation}\label{RDS}
\begin{split}
\partial_t u- a\Delta u+f(u,
t)&=g(x,t),\quad\text{in }\Omega,
 \\
 u&=0,\qquad\quad\ \text{on }\partial\Omega,\\
u|_{t=\tau}&=u_{\tau},\qquad\quad \tau\in{\mathbb{R}}.
\end{split}
\end{equation}
\noindent Here $\Omega$ is a bounded domain in $\mathbb{R}^n$ with
a boundary $\partial\Omega$ of sufficient smoothness;
$a=\left\{a_{ij}\right\}^{j=1,\cdots, N}_{i=1,\cdots, N}$ is an
$N\times N$ real matrix with positive symmetric part
$\frac12(a+a^*)\geq\beta I$, $\beta>0$; $u=u(x,t)=(u^1,\cdots,u^N)$ is the unknown function;
$g=(g^1,\cdots,g^N)$ is the driving force and $f=(f^1,\cdots,f^N)$ is the  interaction function.

Denote the spaces by
$H:=\left(L^2(\Omega)\right)^N$ and $V:=\left(H_0^1(\Omega)\right)^N$, and denote by $(\cdot,\cdot)$ and $|\cdot|$ the $H$-inner product and the corresponding $H$-norm.
 Let  $V'$ be the
dual of $V$. Assume that  $g(s)= g( \cdot ,s)$ is
translation bounded in $L_{\mathrm{loc}}^2\left(
\Bbb{R};V'\right)$, i.e.,
\begin{equation*}\label{c:gRDS}
\left\| g\right\| _{L_{\mathrm{b}}^2\left( \Bbb{R};V'\right) }^2:=
\sup_{t\in \Bbb{R}}\int_t^{t+1}\left\| g(s)\right\|_{V'}
^2\,ds<\infty,
\end{equation*}
and $f(v,s)$ satisfies the following conditions of continuity, dissipativeness and growth:
\begin{align}
\label{c:continuityoffRDS} f(v,s)&\in  C(\mathbb R^N\times\mathbb R;
\mathbb R^N),\\
\label{c:dissipationoffRDS}\sum_{i=1}^{N}\gamma_i|v^i|^{p_i}-C&\leq\sum_{i=1}^{N}f^i(v,s)v^i=f(v,s)\cdot
v,\ \forall v\in \mathbb
R^N,\\
\notag &\quad\qquad\qquad\qquad\gamma_k>0, \ k=1, \cdots, N, \\
\label{c:boundednessoffRDS}\sum_{i=1}^{N}|f^i(v,s)|^\frac{p_i}{p_i-1}&\leq
C\left(\sum_{i=1}^{N}|v^i|^{p_i}+1\right),\ \forall v\in \mathbb
R^N,\\
\notag &\quad\qquad p_1\geq p_2\geq\cdots\geq p_n\geq 2,
\end{align}
where  the letter $C$ denotes a constant which may be different in
each occasion throughout this section.\footnote{RDS (\ref{RDS}) with other
boundary conditions such as Neumann or periodic boundary conditions
can be handled in the same way, and all results hold for these
boundary conditions. For the Dirichlet boundary conditions, instead
of considering $p_k \geq 2$, $k=1,\cdots, N$, for simplicity, we may
assume that $p_k>1$. See Remarks II.4.1 and II.4.2 in \cite{CV02} for
more details.}

Let $q_k:=p_k/(p_k-1)$, $r_k:=\max\{1, n(1/2-1/p_k)\}$, $k=1,
\cdots, N$, and denote by ${\bf p}:=(p_1,\cdots, p_N)$,
${\bf q}:=(q_1,\cdots, q_N)$, ${\bf r}:=(r_1,\cdots, r_N)$ and
\begin{align*}
\begin{split}
&L^{\bf p}(\Omega):=L^{p_1}(\Omega)\times L^{ p_2}(\Omega)\times\cdots\times L^{
p_N}(\Omega),\\
&H^{-{\bf r}}(\Omega):=H^{-r_1}(\Omega)\times H^{-r_2}(\Omega)\times\cdots\times
H^{-r_N}(\Omega),\\
&L^{\bf p}(\tau, t;L^{\bf p}(\Omega)):=L^{p_1}(\tau, t;L^{p_1}(\Omega))\times\cdots\times L^{
p_N}(\tau, t;L^{p_N}(\Omega)),\\
&L^{\bf q}(\tau,t;H^{-{\bf r}}(\Omega)):=L^{q_1}(\tau, t;H^{-r_1}(\Omega))\times\cdots\times L^{
q_N}(\tau, t;H^{-r_N}(\Omega)).
\end{split}
\end{align*}
\begin{definition}
A weak solution  of  \eqref{RDS} on $[\tau,\infty)$ (or $(-\infty, \infty)$, if
$\tau=-\infty$) is a
function $u(x,t)\in L_{\mathrm{loc}}^{\bf p}(\tau, \infty;L^{\bf
p}(\Omega))\cap L_{\mathrm{loc}}^2(\tau, \infty;V)$ that satisfies \eqref{RDS} in the distribution sense of the space $\mathcal D'(\tau,\infty;H^{-{\bf r}}(\Omega))$.
\end{definition}
We recall the results on the existence of weak solutions of (\ref{RDS}) (see e.g.
\cite{CV02}). Note that conditions (\ref{c:dissipationoffRDS})-(\ref{c:boundednessoffRDS}) do not
ensure the uniqueness of the solutions.

\begin{theorem} \label{t:existenceRDS} Let $ f$ satisfy (\ref{c:continuityoffRDS})-(\ref{c:boundednessoffRDS}) and $g \in L^2_{\mathrm {loc}}(\mathbb
R;V')$. For every $u_\tau \in H$, there exists a weak solution $u(t)$ of
(\ref{RDS})  satisfying
\begin{equation*}\label{existence}
 u\in C([\tau,
\infty);H)\cap L_{\mathrm{loc}}^2(\tau, \infty;V)\cap L_{\mathrm{loc}}^{\bf p}(\tau, \infty;L^{\bf
p}(\Omega)).
\end{equation*}Moreover, the function $|u(t)|^2$ is absolutely
continuous on $[\tau,\infty)$ and
\begin{equation}\label{energyequality}
\frac{1}{2}\ddt |u(t)|^2+(a\nabla u(t),\nabla u(t))+(f(u(t),t), u(t))=\langle
g(t),u(t)\rangle,
\end{equation} for a.e. $t\in [\tau,\infty)$.
\end{theorem}

Now, we consider a fixed pair of  an interaction function $f_0$ and a driving
force $g_0$, such that,  $f_0(v,t)$ satisfies
(\ref{c:continuityoffRDS})-(\ref{c:boundednessoffRDS}) and $g_0(t)\in
L^{2}_{\mathrm{b}}(\mathbb{R};V')$.  Let $$\sigma_0:=(f_0, g_0),$$ and $$\Sigma:=\{\sigma_0(\cdot+h):h\in
\mathbb R)\}.$$ Obviously, for every
$\sigma=(f,g) \in \Sigma$, $f$  satisfies
(\ref{c:continuityoffRDS})-(\ref{c:boundednessoffRDS}) with the same
constants, and
\begin{equation}\label{i:gRDS}
\|g\|^2_{L_{\mathrm{b}}^2\left( \Bbb{R};V'\right) }\leq \|g_0\|^2_{L_{\mathrm{b}}^2\left( \Bbb{R};V'\right) }.
\end{equation}
Let $u(t)$, $t\in [\tau,\infty)$, be a weak solution of \eqref{RDS} with $\sigma=(f,g) \in
\Sigma$ guaranteed by
Theorem~\ref{t:existenceRDS}. Thanks to  \eqref{c:dissipationoffRDS} and \eqref{i:gRDS},
we obtain from (\ref{energyequality}) that
\begin{equation}\label{energyinequality}
\ddt |u(t)|^2 + \lambda_1 \beta |u(t)|^2 \leq C + \beta^{-1}
\|g_0\|^2_{V'},
\end{equation}
for a.e. $t\in[\tau,\infty)$. Here $\lambda_1$ is the first
eigenvalue of the Laplacian with Dirichlet boundary conditions.
Due to the absolute continuity of $|u(t)|$ and Gr\"{o}nwall's inequality,  (\ref{energyinequality}) implies that
\begin{equation}\label{i:absorbingBRDS}
|u(t)|^2  \leq |u(\tau)|^2 e^{-\lambda_1 \beta(t-\tau)}+ C, \quad
\forall\, t \in [\tau,\infty).
\end{equation}
Therefore, there exists a  uniformly (w.r.t. $\tau\in \mathbb R$ and $\sigma\in \Sigma$) absorbing ball $B_{\mathrm{s}}(0, R)\subset H$, where the radius $R$ depends on $\lambda_1$, $ \beta$, the constant
in (\ref{c:dissipationoffRDS}) and $\|g_0\|^2_{L_{\mathrm{b}}^2\left( \Bbb{R};V'\right) }$. We denote by
$X$  a closed   absorbing ball
\begin{equation}\label{absorbingballRDS}
X= \{u\in H: |u| \leq R\}.
\end{equation}That is, for any bounded set $A
\subset H$, there exists a time $\bar{t} \geq 0$ independent of the initial time $\tau$, such that
\begin{equation}\label{absorbingtrajectoryRDS}
u(t) \in X, \quad \forall\, t\geq t_1:=\tau+\bar{t},
\end{equation}
for every weak solution $u(t)$ with  $\sigma\in \Sigma$ and
the initial data $u (\tau) \in A$.
It is known that $X$ is  weakly compact in $H$ and  metrizable with a metric $\dw$
deducing the weak topology on $X$.

Consider an evolutionary system for which a family of trajectories consists of all weak solutions of
(\ref{RDS}) with the fixed $\sigma_0$ in $X$. More precisely, define
\[
\begin{split}
\Dc([\tau,\infty)) := \{&u(\cdot): u(\cdot) \mbox{ is a weak solution on } [\tau,\infty)\\
&\ \mbox{ with } \sigma\in \Sigma \mbox{ and } u(t) \in
X, \forall\, t \in [\tau,\infty)\}, \  \tau \in \mathbb{R},
\end{split}
\]
\[
\begin{split}
\Dc((-\infty,\infty)) := \{&u(\cdot): u(\cdot)
\mbox{ is a weak solution on } (-\infty,\infty)\\
&\ \mbox{ with } \sigma\in \Sigma \mbox{ and } u(t) \in X, \forall\, t \in (-\infty,\infty)\}.
\end{split}
\]

Clearly, the properties 1-4 in Definition \ref{d:NDc} hold for $\Dc$ if we utilize the fact that is formulated as the translation identity: A weak solution of  (\ref{RDS}) with  $\sigma\in \Sigma$ initiating at time $\tau+h$ is also a weak solution of  (\ref{RDS}) with  $\sigma(\cdot+h)\in \Sigma$ initiating at time $\tau$.
Thanks to Theorem \ref{t:wA}, the weak global
attractor $\Aw$ for this evolutionary system exists.

\begin{lemma}\label{l:convergenceofsolutionRDS}
Let $u_k(t)$ be a sequence of weak solutions of (\ref{RDS}) with
 $\sigma_k \in \Sigma$, such that $u_k(t) \in X$ for all
$t\geq t_1$. Then
\[
\begin{aligned}
u_k \ \ &\mbox{is bounded in} \ \ L^2(t_1,t_2;V),\\
\partial_t u_k \ \  &\mbox{is bounded in} \ \
L^{{\bf q}}(t_1,t_2;H^{-{\bf r}}(\Omega)),
\end{aligned}
\]
for all $t_2>t_1$. Moreover, there exists a subsequence  $u_{k_j}$ converges in $ C([t_1,t_2]; \Hw)$ to some
 $\phi(t)\in  C([t_1,t_2]; H)$, i.e.,
\[
(u_{k_j},v) \to (\phi,v) \  \ \mbox{uniformly on } [t_1,t_2],
\]
as $k_j\to \infty$, for all $v \in H$.
\end{lemma}

\begin{proof}The proof is analogous with that of Lemma 3.2 in \cite{Lu07} and  that of Lemma 2.1 in \cite{R98}.
Standard estimates (see e.g. \cite{CV02}) show that, for all $t_2>t_1$,
\begin{equation}\label{boundedness_E} \{u_k\} \ \ \mbox{is bounded in} \ \ L^2(t_1,t_2;V)\cap
L^{\infty}(t_1,t_2;H)\cap L^{{\bf p}}\left(t_1,t_2;L^{{\bf p}}(\Omega)\right),
\end{equation} and
\begin{align}\label{boundedness_dE}
\{\partial_t u_k\} \ \  &\mbox{ is bounded in} \ \ L^{{\bf q}}(t_1,t_2;H^{-{\bf r}}(\Omega)),\\
\label{bounednessofu_nlp}\{f_k(u_k(x,t),t)\} \ &\mbox{  is
bounded in  }
L^{\bf q}\left(t_1,t_2;L^{\bf q}(\Omega)\right).
\end{align} By the embedding theorem (cf. Theorem II.1.4 in \cite{CV02}, Theorem 8.1 in
\cite{Ro01}), we obtain that
\begin{equation}\label{precompactness_E}
\{u_k\} \ \ \mbox{is precompact in} \ \ L^2(t_1,t_2;H).
\end{equation}
Passing to  a subsequence and dropping a subindex, we know
from (\ref{boundedness_E})-(\ref{precompactness_E})
that,
\begin{align}\label{convergenceofu_n}
u_{k}(t)\rightarrow\phi(t) \quad &\text{ weak-star in
}L^\infty(t_1,t_2;H),\notag\\
&\text{ weakly in }L^2(t_1,t_2;V)\cap L^{{\bf p}}\left(t_1,t_2;L^{{\bf p}}(\Omega)\right),\\
&\text{ strongly in }L^2(t_1,t_2;H),\notag
\end{align}
 and
\begin{align}\label{convergenceofu_nd} \Delta u_{k}(t)\rightarrow \Delta\phi(t)\quad &\text{ weakly in }L^2(t_1,t_2;V'),\notag\\
\partial_t u_{k}(t)\rightarrow \partial_t\phi(t)\quad &\text{ weakly in
}L^{{\bf q}}(t_1,t_2;H^{-{\bf r}}(\Omega)),\\
f_{k}(u_{k}(x,t),t)\rightarrow \psi(t)\quad &\text{ weakly in
}L^{{\bf q}}\left(t_1,t_2;L^{{\bf q}}(\Omega)\right),\notag
\end{align} for some
\begin{align}\label{phi}
\phi(t)\in L^\infty(t_1,t_2;H)\cap L^2(t_1,t_2;V)\cap
L^{{\bf p}}\left(t_1,t_2;L^{{\bf p}}(\Omega)\right),
\end{align}
and some
\begin{align}\label{phi} \psi(t)\in L^{{\bf q}}\left(t_1,t_2;L^{{\bf q}}(\Omega)\right).
\end{align}
Note that $g_0$ is translation
compact in $ L^{2,\mathrm{w}}_\mathrm{loc}(\mathbb{R};V')$ (see \cite{CV02}). Thus, passing to  a subsequence and dropping a subindex again,  we also have,
\begin{equation}\label{convergenceofg_n}
g_k(t) \rightarrow g(t)\quad \text{ weakly in }L^{2}(t_1,t_2;V'),
\end{equation}
with some
$g(t)\in L^{2}(t_1,t_2;V').$
Passing
the limits  yields the following equality
\begin{align}\label{quasiequ} \partial_t \phi-a\Delta \phi+ \psi=g
\end{align}
in the distribution sense of the space $\mathcal D'(t_1,t_2;H^{-\bf{r}}(\Omega))$. Thanks to a vector version of Theorem
II.1.8 in \cite{CV02}, (\ref{convergenceofu_nd})-(\ref{quasiequ})
indicate that $\phi(t)\in  C([t_1,t_2];H)$.

Now we prove that
$u_{k}(t)\rightarrow \phi(t)$ in $C([t_1,t_2]; H_{\mathrm w})$.

Thanks  to  the strong convergence in (\ref{convergenceofu_n}), we know that (passing to  a subsequence and dropping a subindex in above procedures if it is necessary)
\[
u_k(t)\rightarrow \phi(t) \quad \text{ strongly in } H, \quad a.e. \ t\in [t_1,t_2].
\]
Thus, for any test function $v\in (C_0^\infty(\Omega))^N$,
\[
(u_k(t),v)\rightarrow(\phi(t),v),\quad a.e. \ t\geq t_1.
\]
It follows from (\ref{boundedness_E}) that
$\{(u_k(t),v)\}$ is uniformly bounded on $[t_1,t_2]$.
On the other hand, by (\ref{boundedness_dE}), for every $v\in (C_0^\infty(\Omega))^N$, $0<\delta <1$  and
$t_1\leq t \leq t+\delta \leq t_2$,
\begin{align*}
|(u_k(t+\delta)-u_k(t),v)|&=\left|\int^{t+\delta}_t\langle \partial_t u_k(s),v\rangle\,ds\right|\\
&\leq C\delta^{\frac{1}{p_1}}\|v\|_{H^{{\bf r}}}\|\partial_t u_k\|_{L^{{\bf q}}(t,t+\delta;H^{-\bf{r}})}\\
&\leq C\delta^{\frac{1}{p_1}}\|v\|_{H^{{\bf r}}}.
\end{align*}
That is,  the sequence $\{(u_k(t),v)\}$ is equicontinuous on $[t_1,t_2]$. Hence,
\begin{align*}
(u_k(t),v)\rightarrow(\phi(t),v) \quad \mbox{uniformly on } [t_1,t_2],\quad \forall v\in
(C_0^\infty(\Omega))^N. \end{align*}
Note that  $(C_0^\infty(\Omega))^N$ is dense in $H$. We have
\begin{align*}
(u_k(t),v)\rightarrow(\phi(t),v) \quad \mbox{uniformly on }  [t_1,t_2],\quad \forall v\in H.
\end{align*}

We complete the proof.
\end{proof}

Then, we have the
following.
\begin{lemma} \label{l:compactRDS}
The evolutionary system $\Dc$ of (\ref{RDS}) with the fixed $\sigma_0$ satisfies A1 and
A3. Moreover, if $g_0$ is normal in
$L^2_{\mathrm{loc}}(\mathbb{R};V')$ then A2 holds.
\end{lemma}
\begin{proof}The proof is analogous to that of Lemma 3.4 in \cite{CL09}.
First,  by Theorem \ref{t:existenceRDS}, $\Dc([0,\infty)) \subset C([0,\infty);\Xs)$. Now take a sequence $\{u_k \}\subset \Dc([0,\infty))$. Owing to
Lemma~\ref{l:convergenceofsolutionRDS}, there exists a subsequence,
still denoted by  $\{u_k\}$, which converges in $C([0, 1];\Xw)$ to some
$\phi^{1} \in C([0, 1];\Xs)$  as $k \to \infty$. Passing to a
subsequence and dropping a subindex once more, we have that $u_k
\to \phi^2$ in $C([0, 2];\Xw)$ as $k \to \infty$ for some $\phi^{2}
\in C([0, 2];\Xs)$. Note that $\phi^1(t)=\phi^2(t)$ on $[0, 1]$.
Continuing this diagonalization process, we obtain a subsequence
$\{u_{k_j}\}$ of $\{u_k\}$ that converges in $ C([0, \infty);\Xw)$ to some
$\phi \in C([0, \infty);\Xs)$  as $k_j \to \infty$. Therefore, A1
holds.

 Take a sequence $\{u_k\} \subset \Dc([0,\infty))$ be such
that it is a $\mathrm d_{ C([0, T];\Xw)}$-Cauchy sequence  in $ C([0, T];\Xw)$ for
some $T>0$.   Thanks to Lemma~\ref{l:convergenceofsolutionRDS} again, the
sequence $\{u_k\}$ is bounded in $L^2(0,T;V)$. Hence, there
exists some $\phi(t)\in  C([0, T]; \Xw)$, such that
\[
\int_{0}^{T} |u_k(s)-\phi(s)|^2 \, ds \to 0, \quad \mbox{ as }
 k \to \infty.
\]
In particular, $|u_k(t)| \to |\phi(t)|$ as $k \to \infty$ a.e.
on $[0,T]$, which means that  $u_k(t)$ is a $\ds$-Cauchy
sequence a.e. on $[0,T]$. Thus, A3 is valid.

For  any $u \in \Dc([0,\infty))$ and $t>0$, it follows from
(\ref{energyinequality}) and the absolute continuity of $|u(\cdot)|^2$ that
\begin{equation}\label{ie:abvestimate}
|u(t)|^2\leq   |u(t_0)|^2+C(t-t_0)+ \frac{1}{\beta}\int_{t_0}^t
\|g_0\|^2_{V'}\,ds,
\end{equation}
for all $0\leq t_0<t$. Here $C$ is independent of $u$.
Suppose now that $g_0$ is normal in $L^2_{\mathrm{loc}}(\mathbb{R};V')$. Then given $\epsilon>0$, there exists $ 0<\delta<\epsilon/2C$, such that
\begin{equation*} \label{normality}\sup_{t\in \mathbb{R}} \int_{t-\delta}^t
\|g_0(s)\|^2_{V'} \, ds \leq \frac{\beta\epsilon}{2}.
\end{equation*}
It follows from (\ref{ie:abvestimate}) that
\begin{equation*}\label{A2}
|u(t)|^2 \leq |u(t_0)|^2  + \epsilon, \quad\,\forall t_0\in (t-\delta,t),
\end{equation*}which concludes that A2 holds.
\end{proof}

We have the followings.
\begin{theorem} \label{t:structureofattractorwRDS} Let $ f_0$ satisfy (\ref{c:continuityoffRDS})-(\ref{c:boundednessoffRDS}) and $g_0$ be translation bounded in $ L^{2}_\mathrm{loc}(\mathbb{R};V')$.  Then the weak uniform global attractor $\Aw$  and the weak trajectory attractor $\mathfrak A_{\w}$ for (\ref{RDS}) with the fixed $\sigma_0=(f_0, g_0)$ exist, $\Aw$ is the maximal invariant and maximal quasi-invariant set w.r.t. the closure $\bar{\Dc}$ of the
corresponding evolutionary system $\Dc$, and
\begin{equation*}
\begin{split}
\Aw &= \ww(X)=\ws(X)=\left\{ u(0):  u \in
\bar{\Dc}((-\infty, \infty))\right\}, \\ \mathfrak A_{\w}&=\Pi_+\bar{\Dc}((-\infty, \infty))=\left\{u(\cdot)|_{[0,\infty)}:u\in \bar{\Dc}((-\infty, \infty))\right\}, \\
\Aw&=\mathfrak A_{\w}(t)=\left\{u(t): u\in \mathfrak A_{\w}\right\},\quad \forall\, t\geq 0.
\end{split}
\end{equation*}
Moreover,  $\mathfrak A_{\w}$ satisfies the finite weak uniform tracking property and is weakly equicontinuous on $[0,\infty)$.
\end{theorem}
\begin{proof}It is known from Lemma \ref{l:compactRDS} that the associated evolutionary system $\Dc$ satisfies A1. Then the conclusions follow from Theorems \ref{t:wA}.
\end{proof}
\begin{theorem}\label{t:sturctureofattractorsRDS}
Furthermore, if  $g_0$ is  normal in $L^2_{\mathrm{loc}}(\mathbb{R}; V')$, then
the weak global attractor $\Aw$ is a strongly compact
strong global attractor $\As$, and the weak trajectory attractor $\mathfrak A_{\w}$ is a strongly compact strong trajectory attractor $\mathfrak A_{\s}$. Moreover,
\begin{itemize}
\item[1.] $\mathfrak A_{\s}=\Pi_+\bar{\Dc}((-\infty, \infty))$ satisfies the finite strong
uniform tracking property, i.e.,  for any $\epsilon >0$ and $T>0$, there exist $t_0$ and a finite subset $P_T^f\subset \mathfrak A_{\s}|_{[0,T]}$, such
that for any $t^*>t_0$, every trajectory $u \in \Dc([0,\infty))$
satisfies
$$|u(t)-v(t-t^*)| < \epsilon, \quad \forall t\in [t^*,t^*+T],$$
for some $T$-time length piece  $v\in P_T^f$.
\item[2.]
$\mathfrak A_{\s}=\Pi_+\bar{\Dc}((-\infty, \infty))$  is strongly equicontinuous on $[0,\infty)$, i.e.,
\[
\left|v(t_1)-v(t_2)\right|\leq \theta\left(|t_1-t_2|\right), \quad\forall\, t_1,t_2\ge 0, \  \forall v\in \mathfrak A_{\s},
\]
where $\theta(l) $ is a positive function tending
to $0$ as $l\rightarrow 0^+$.
\end{itemize}
\end{theorem}
\begin{proof}According to Theorem \ref{t:A123Comp}, Theorem \ref{t:existenceRDS} and  Lemma \ref{l:compactRDS} mean that the associated evolutionary system $\Dc$ is asymptotically compact.  Hence, we obtain the theorem by applying Theorem \ref{t:sA}.
\end{proof}
\begin{remark}
The existence of $\As$ is obtained in \cite{CL09}. Thus, instead of Theorem \ref{t:A123Comp}, we can also utilize Theorem \ref{t:suffience and necessity} to ensure the asymptotical compactness of the corresponding evolutionary system $\Dc$.
\end{remark}
\begin{remark}\label{r:ansOP}
The equality in the conclusion 1 of  Theorem \ref{t:sturctureofattractorsRDS} answers an open problem in \cite{Lu07, CL09}, which concerns how to describe the structure of $\As$ for \eqref{RDS} with general interaction terms, satisfying no additional  assumption other than conditions of  (\ref{c:continuityoffRDS})-(\ref{c:boundednessoffRDS}). Note that, only with these three conditions on nonlinearities, it is not known how to construct a suitable symbol space that is necessary for applying previous framework \cite{CV02}. See Subsetion \ref{ss:nonln} for more.
\end{remark}
\subsection{RDS with more regular interaction terms} \label{ss:MoreRegN}In this subsection, we study \eqref{RDS} with more regular interaction functions.

Denote by $\mathcal M$ the space $C(\mathbb R^N; \mathbb R^N)$  endowed
with the local uniform convergence topology.
Denote by  $ C^{\mathrm{p.p.}}(\mathbb R;\mathcal
M)$  the space $C(\mathbb R;\mathcal M)$ endowed with the
topology of the following convergence: $\varphi_k(s)\rightarrow
\varphi(s)$ in $ C^{\mathrm{p.p.}}(\mathbb
R;\mathcal M)$ as $k\rightarrow \infty$, if $\varphi_k(v,s)$ is uniformly bounded on any
ball in $\mathbb R^N\times \mathbb R$ and for every $(v,s)\in \mathbb
R^N\times \mathbb R$,
\[\|\varphi_k(v,s)-\varphi(v,s)\|_{\mathbb R^N} \rightarrow 0,\quad
\text{as }k\rightarrow \infty.\]
Note that $ C^{\mathrm{p.p.}}(\mathbb R;\mathcal
M)$ is in fact  the space $C (\mathbb R^N\times \mathbb R; \mathbb R^N)$ endowed with the usual weak topology.
Let  $C^{\mathrm{p.u.}}(\mathbb R;\mathcal M)$ denote  the space $ C(\mathbb R;\mathcal M)$ endowed with another topology of the following convergence:  $\varphi_k(s)\rightarrow \varphi(s)$ in $C^{\mathrm{p.u.}}(\mathbb R;\mathcal
M)$ as
$k\rightarrow \infty$, if $\varphi_k(v,s)$ is uniformly bounded on any ball in $\mathbb R^N\times \mathbb R$ and for every $s\in \mathbb
R$, $R>0$,
\[\max_{\|v\|_{\mathbb R^N}\leq
R}\|\varphi_k(v,s)-\varphi(v,s)\|_{\mathbb R^N} \rightarrow
0,\quad \text{as }k\rightarrow \infty.\]

Now we  assume in addition  that $f_0(v,t)$ is translation
compact in $C^\mathrm{p.u.}(\mathbb R;\mathcal M)$, i.e., the following closure
$$\bar{\Sigma_1}:=\overline{\{f_0(\cdot+h):h\in
\mathbb R)\}}^{C^\mathrm{p.u.}(\mathbb R;\mathcal M)}$$
is compact in $ C^\mathrm{p.u.}(\mathbb R;\mathcal M)$. Note that $\bar{\Sigma_1}$ is metrizable  in the space\footnote{\label{fn:metrizable}If
$K$ is a (relatively) weakly compact set in a Banach space $\mathcal B$ and the dual $\mathcal B'$ of $\mathcal B$
contains a countable total set, then the $
\overline{K}^{\text{weak}}$ is metrizable. Recall that a set
$A\subset \mathcal B'$ is called total if $a(x)=0$ for every $a\in A$
implies $x=0$ (see \cite{Di84}, p18). Note that $\bar{\Sigma_1}$
is also a compact set in $C^{\mathrm{p.p.}}(\mathbb R;\mathcal M)$,
i.e., a weakly compact set in $C(\mathbb R^N\times\mathbb
R;\mathbb R^N)$ (cf. Theorem VII.1 in \cite{Di84}). For any  closed bounded ball $B\subset \mathbb
R^N\times\mathbb R$, $\bar{\Sigma_1}|_{B}:=\{f|_{B}=\text{restriction of } f \text{ on } B:f\in \bar{\Sigma_1}\}$ is weakly compact in $C(B;\mathbb R^N)$.
$C(B;\mathbb R^N)'$ contains a total set of Dirac
$\delta$-measures on rational points of $B$. Hence $\bar{\Sigma_1}|_{B}$ endowed with the weak topology of $C(B;\mathbb R^N)$ is metrizable. It follows that $\bar{\Sigma_1}$ is metrizable in $C^{\mathrm{p.p.}}(\mathbb R;\mathcal M)$ by means of the so-called Fr\'{e}chet metric.} $ C^\mathrm{p.p.}(\mathbb R;\mathcal M)$ and is compact w.r.t. such a  metric. For convenience, we gather the properties of this kind of translation compact functions in Subsection \ref{ss:nonln} below. Denote by
$$\bar{\Sigma_2}:=\overline{\{g_0(\cdot+h): h\in \mathbb R\}}^{L^{2,\mathrm{w}}_\mathrm{loc}(\mathbb{R};V') }.$$
Then, $\bar{\Sigma_2}$ endowed with the topology of $L^{2,\mathrm{w}}_\mathrm{loc}(\mathbb{R};V')$ is metrizable and the corresponding metric space is compact (see \cite{CV02}).

Let $$\bar{\Sigma}:=\overline{\{\sigma_0(\cdot+h):h\in
\mathbb R)\}}^{C^\mathrm{p.u.}(\mathbb R;\mathcal M)\times
L^{2,\mathrm{w}}_{\mathrm{loc}}(\mathbb{R};V')}=\bar{\Sigma_1}\times\bar{\Sigma_2}.$$
Then $\bar{\Sigma}$ is compact in the product space $C^\mathrm{p.u.}(\mathbb R;\mathcal M)\times
L^{2,\mathrm{w}}_{\mathrm{loc}}(\mathbb{R};V')$ and is metrizable in the weaker space $C^\mathrm{p.p.}(\mathbb R;\mathcal M)\times
L^{2,\mathrm{w}}_{\mathrm{loc}}(\mathbb{R};V')$.  For every $\sigma=(f,g)  \in \bar{\Sigma}$, $g$
satisfies (\ref{i:gRDS}) by Proposition V.4.2 in \cite{CV02} and there exists  a sequence $\{f_0(\cdot, \cdot+h_n)\}$ that converges to $f$ in $C^\mathrm{p.u.}(\mathbb R;\mathcal M)$ as $n\rightarrow \infty$. Then, $\{f_0(\cdot, \cdot+h_n)\}$  also converges to $f$ in the weak topology of  $C (\mathbb R^N\times \mathbb R; \mathbb R^N)$.
Note that the weak convergence of $C (\mathbb R^N\times \mathbb R; \mathbb R^N)$ is equivalent  to the local uniform boundedness and the pointwise
convergence of a sequence of functions of $C (\mathbb R^N\times \mathbb R; \mathbb R^N)$ (cf. \cite{Di84, Me98}). Thus, $f$ satisfies
(\ref{c:continuityoffRDS})-(\ref{c:boundednessoffRDS}) with the same
constants.

It can be seen that the argument of (\ref{energyinequality})-(\ref{absorbingtrajectoryRDS}) is  still valid for  $\Sigma$ replaced by $\bar{\Sigma}$ and Lemma
\ref{l:convergenceofsolutionRDS} can be  improved as follows.

\begin{lemma}\cite{CL09}\label{l:convergenceofsolutionRDSclosed}
Let $u_k(t)$ be a sequence of weak solutions of (\ref{RDS}) with
 $\sigma_k \in \bar{\Sigma}$, such that $u_k(t) \in X$ for all
$t\geq t_1$. Then
\[
\begin{aligned}
u_k \ \ &\mbox{is bounded in} \ \ L^2(t_1,t_2;V),\\
\partial_t u_k \ \  &\mbox{is bounded in} \ \
L^{{\bf q}}(t_1,t_2;H^{-{\bf r}}(\Omega)),
\end{aligned}
\]
for all $t_2>t_1$. Moreover, there exists a subsequence  $k_j$, such that $\sigma_{k_j}$ converges in $C^\mathrm{p.u.}(\mathbb R;\mathcal M)\times
L^{2,\mathrm{w}}_{\mathrm{loc}}(\mathbb{R};V')$ to some $\sigma\in \bar{\Sigma}$, and  $u_{k_j}$ converges in $ C([t_1,t_2]; \Hw)$ to some
weak solution $u(t)$ of (\ref{RDS}) with $\sigma$, i.e.,
\[
(u_{k_j},v) \to (u,v) \  \ \mbox{uniformly on }  [t_1,t_2],
\]
as $k_j\to \infty$, for all $v \in H$.
\end{lemma}
\begin{proof}Note that Lemma \ref{l:convergenceofsolutionRDS} still holds for  $\Sigma$ replaced by $\bar{\Sigma}$ and the proof needs no modification at all. Now we continue the proof.
Since $\bar{\Sigma_1}$
is also a compact set in $C^{\mathrm{p.p.}}(\mathbb R;\mathcal M)$, $\{f_k\}$ in (\ref{convergenceofu_nd}) can be taken as a convergent sequence in $C^{\mathrm{p.p.}}(\mathbb R;\mathcal M)$ with a limit $f\in \bar{\Sigma_1}$.

We claim that $f(\phi(x,t),t)=\psi(t)$, which will imply
that $\phi (t)$ is a weak solution of (\ref{RDS}) with the interaction function $f$ and the driving force $g$ in (\ref{convergenceofg_n}). Obviously, $g\in \bar{\Sigma_2}$.  Thanks to  the strong convergence in
(\ref{convergenceofu_n}), we know that, passing to a subsequence
if necessary,
\[u_k(x,t)\rightarrow \phi(x,t), \quad a.e. \ (x,t)\in \Omega\times [t_1,t_2].
\]
Note that
\begin{align*} \|f_k(u_k(x,t),t)-f(\phi(x,t),t)\|_{\mathbb
R^N}\leq & \|f_k(u_k(x,t),t)-f_k(\phi(x,t),t)\|_{\mathbb
R^N}\\&+\|f_k(\phi(x,t),t)-f(\phi(x,t),t)\|_{\mathbb
R^N}.
\end{align*}
Due to  Theorem \ref{proofinteractionfunction} in Subsection \ref{ss:nonln} below,
all $f_k$ satisfy (\ref{ucofphiv}) with the same function
$\theta$. Hence we obtain that
 \begin{align}\label{pointwiseoffk}
\|f_k(u_k(x,t),t)-f(\phi(x,t),t)\|_{\mathbb R^N}\rightarrow 0,
 \quad a.e. \ (x,t)\in \Omega\times [t_1,t_2].
 \end{align} On the other hand, according to Lemma II.1.2 in \cite{CV02}, the uniform boundedness of
(\ref{bounednessofu_nlp}) and the pointwise convergence of \eqref{pointwiseoffk} yield that
\begin{align*}
f_k(u_k(x,t),t)\rightarrow f(\phi(x,t),t) \quad\mbox{ weakly in }
L^{\bf q}\left(t_1,t_2;L^{\bf q}(\Omega)\right).
\end{align*}
Therefore, $f(\phi(x,t),t)=\psi(t)$ in $L^{\bf q}\left(t_1,t_2;L^{\bf q}(\Omega)\right)$ for all $t_2>t_1$.

Finally, thanks to Theorem \ref{t:squenconvergence} below, indeed, $f_k\rightarrow f$ in $ C^{\mathrm{p.u.}}(\mathbb R;\mathcal M)$. Together with (\ref{convergenceofg_n}), it deduces that
\begin{align}\label{convergenceofsigma_n}\sigma_{k}=(f_k, g_k)\rightarrow (f,g)\quad\mbox{ in } C^\mathrm{p.u.}(\mathbb R;\mathcal M)\times
L^{2,\mathrm{w}}_{\mathrm{loc}}(\mathbb{R};V'),
\end{align} with $\sigma=(f,g)\in \bar{\Sigma}$.
\end{proof}

Similarly, we can now also consider another evolutionary system with $\bar{\Sigma}$ as a symbol space.  The family
of trajectories for this evolutionary system  consists of all weak solutions of the family
of (\ref{RDS}) with  $\sigma\in \bar{\Sigma}$ in $X$:
\[
\begin{split}
\Dc_{\bar{\Sigma}}([\tau,\infty)) := \{&u(\cdot): u(\cdot) \mbox{ is a
weak solution on } [\tau,\infty)\\
&\ \mbox{ with  } \sigma\in \bar{\Sigma} \mbox{ and } u(t) \in X, \
\forall t \in [\tau,\infty)\}, \  \tau \in \mathbb{R},
\end{split}
\]
\[
\begin{split}
\Dc_{\bar{\Sigma}}((-\infty,\infty)) := \{&u(\cdot): u(\cdot)
\mbox{ is a weak solution on } (-\infty,\infty)\\
&\ \mbox{with } \sigma\in \bar{\Sigma} \mbox{ and } u(t) \in X, \
\forall t \in (-\infty,\infty)\}.
\end{split}
\]
We have the following lemmas.
\begin{lemma}\cite{CL09}\label{l:compactRDSclosed}
The evolutionary system $\Dc_{\bar{\Sigma}}$ of the family of  (\ref{RDS}) with
$\sigma\in\bar{\Sigma}$ satisfies \={A1}.
\end{lemma}
\begin{lemma}\label{l:closedRDSclosed}
The evolutionary system $\Dc_{\bar{\Sigma}}$ of the family of  (\ref{RDS}) with
$\sigma\in\bar{\Sigma}$  is closed.
\end{lemma}
\begin{proof}
With Lemma \ref{l:convergenceofsolutionRDSclosed} in hand, the proof is just the same as that of Lemma \ref{l:closedNSE3}.
\end{proof}
\begin{lemma}\label{l:clousurequalsitselfRDS}Let $\Dc_{\bar{\Sigma}}$ be the evolutionary system of the family of  (\ref{RDS}) with
$\sigma\in\bar{\Sigma}$. Then $\bar{\Dc}_{\bar{\Sigma}}=\Dc_{\bar{\Sigma}}$.
\end{lemma}
\begin{proof}Using Lemma \ref{l:closedRDSclosed}, the argument is the same as that of Lemma \ref{l:clousurequalsitselfNSE3}.
\end{proof}
\begin{remark} In the  proofs of Lemmas \ref{l:compactRDSclosed}, \ref{l:closedRDSclosed} and \ref{l:clousurequalsitselfRDS},  whenever utilizing Lemma \ref{l:convergenceofsolutionRDSclosed}, instead of (\ref{convergenceofsigma_n}), the convergence of $\sigma_k$ in $C^\mathrm{p.p.}(\mathbb R;\mathcal M) \times
L^{2,\mathrm{w}}_{\mathrm{loc}}(\mathbb{R};V')$ is enough. Note again that $\bar{\Sigma}$ is compact in $C^\mathrm{p.u.}(\mathbb R;\mathcal M)\times
L^{2,\mathrm{w}}_{\mathrm{loc}}(\mathbb{R};V')$ and is metrizable in $C^\mathrm{p.p.}(\mathbb R;\mathcal M)\times
L^{2,\mathrm{w}}_{\mathrm{loc}}(\mathbb{R};V')$. Hence, $\bar{\Sigma}$ is sequentially  compact in $C^\mathrm{p.p.}(\mathbb R;\mathcal M)\times
L^{2,\mathrm{w}}_{\mathrm{loc}}(\mathbb{R};V')$.  Though it is also sequentially  compact in $C^\mathrm{p.u.}(\mathbb R;\mathcal M)\times
L^{2,\mathrm{w}}_{\mathrm{loc}}(\mathbb{R};V')$ due to Corollary \ref{t:sequencomp} below, it is not necessary for our current procedures.
\end{remark}
\begin{theorem} \label{t:AwRDStrcclosed} Let $ f_0$ satisfy (\ref{c:continuityoffRDS})-(\ref{c:boundednessoffRDS}) and be translation
compact in $C^\mathrm{p.u.}(\mathbb R;\mathcal M)$,  and $g_0$ be translation bounded in $ L^{2}_\mathrm{loc}(\mathbb{R};V')$.  Then  the weak uniform global attractor  $\Aw^{\bar{\Sigma}}$  and the weak trajectory attractor $\mathfrak A^{\bar{\Sigma}}_{\w}$  for the family of (\ref{RDS})
 with $\sigma=(f, g)\in \bar{\Sigma}$ exist, $\Aw^{\bar{\Sigma}}$ is the maximal invariant
and maximal quasi-invariant set w.r.t.  the corresponding
evolutionary system $\Dc_{\bar{\Sigma}}$, and
\begin{equation*}
\begin{split}
\Aw^{\bar{\Sigma}}  &=\left\{u(0): u \in \Dc_{\bar{\Sigma}}((-\infty, \infty))\right\}= \left\{ u(0): u\in \bigcup_{\sigma\in\bar{\Sigma}}\Dc_\sigma ((-\infty, \infty))\right\}, \\ \mathfrak A^{\bar{\Sigma}}_{\w}&=\Pi_+\bigcup_{\sigma\in\bar{\Sigma}}\Dc_\sigma ((-\infty, \infty)), \\
\Aw^{\bar{\Sigma}}&=\mathfrak A^{\bar{\Sigma}}_{\w}(t)=\left\{u(t): u\in \mathfrak A^{\bar{\Sigma}}_{\w}\right\},\quad \forall\, t\geq 0,
\end{split}
\end{equation*}
where $\mathcal E_{\sigma} ((-\infty,\infty))$ is nonempty for any $\sigma\in \bar{\Sigma}$.
Moreover,  $\mathfrak A^{\bar{\Sigma}}_{\w}$ satisfies the finite weak uniform tracking property and  is weakly equicontinuous on $[0,\infty)$.
\end{theorem}
\begin{proof}
Utilizing  Lemmas \ref{l:closedRDSclosed}, \ref{l:clousurequalsitselfRDS} and \ref{l:compactRDSclosed}, the proof is analogous to that of Theorem \ref{t:AwNSEtrc}.
\end{proof}

The existence of  $\As^{\bar{\Sigma}}$  as well as $\Aw^{\bar{\Sigma}}$ was proved in \cite{CL09}. Hence, we are able to obtain strong compactness of  $\mathfrak A^{\bar{\Sigma}}_{\s}$  from that of $\As^{\bar{\Sigma}}$ by applying Corollary \ref{t:SAtoSTA}.
\begin{theorem}\label{t:AsRDSclosed}
Furthermore, if  $g_0$ is  normal in $L^2_{\mathrm{loc}}(\mathbb{R}; V')$, then the weak uniform global attractor $\Aw^{\bar{\Sigma}}$
is a strongly compact strong uniform global attractor $\As^{\bar{\Sigma}}$ and  the weak trajectory attractor $\mathfrak A^{\bar{\Sigma}}_{\w}$ is a strongly
compact strong trajectory attractor $\mathfrak A^{\bar{\Sigma}}_{\s}$. Moreover, $\mathfrak A^{\bar{\Sigma}}_{\s}$ satisfies the finite
strong uniform tracking property and is strongly equicontinuous on $[0,\infty)$.
\end{theorem}

Let $\bar{\Dc}$ be the closure of the  evolutionary  system $\Dc$.
It follows from Lemma \ref{l:clousurequalsitselfRDS} that $\Dc\subset\bar{\Dc}\subset\Dc_{\bar{\Sigma}}$. Then, an
interesting problem arises:
\begin{open}\cite{CL09}\label{op:attractorRDS}
Are the attractors $\Ab$,  $\mathfrak A_{\bullet}$ and
$\Ab^{\bar{\Sigma}}$, $\mathfrak A^{\bar{\Sigma}}_{\bullet}$ in Theorems \ref{t:structureofattractorwRDS}
and \ref{t:AsRDSclosed} identical?
\end{open}

As indicated in Theorems \ref{t:A0Aw} and \ref{t:saofESac}, when the solutions of \eqref{RDS} with $\sigma\in \bar{\Sigma}$ are unique, the answer is positive. Otherwise, the negative answer would imply that the solutions are not unique and, moreover, that the attractors $\Ab^{\bar{\Sigma}}$ and $\mathfrak A^{\bar{\Sigma}}_{\bullet}$  for the auxiliary family of \eqref{RDS} with $\sigma\in \bar{\Sigma}$ do not
satisfy the minimality property w.r.t. $\db$-attracting and $\mathrm d_{ C([0, \infty);X_\bullet)}$-attracting, respectively,  for the evolutionary system $\Dc$ corresponding to the original \eqref{RDS} with the fixed $\sigma_0$.
For more details, see \cite{CL09, CL14}.

\subsubsection{RDS with unqiueness}

Now we suppose further the following condition on the nonlinearity $f_0(v,s)$,
\begin{equation}
\label{c:uniquenessoffRDS}\left(f_0(v_1,s)-f_0(v_2,s),v_1-v_2\right)\geq
-C\|v_1-v_2\|_{\mathbb R^N}^2, \forall v_1,v_2\in \mathbb R^N,\forall s\in \mathbb
R.
\end{equation}
It is known  that the weak solutions provided by Theorem \ref{t:existenceRDS} are now unique (see e.g. \cite{CV02}) .
\begin{theorem} \label{t:AwRDSStrcuniq}Let $ f_0$ satisfy (\ref{c:continuityoffRDS})-(\ref{c:boundednessoffRDS}), (\ref{c:uniquenessoffRDS}) and be translation
compact in $C^\mathrm{p.u.}(\mathbb R;\mathcal M)$,  and $g_0$ be translation bounded in $ L^{2}_\mathrm{loc}(\mathbb{R};V')$. Then  the two weak uniform global attractors $\Aw$,  $\Aw^{\bar{\Sigma}}$  and the two weak trajectory attractors $\mathfrak A_{\w}$, $\mathfrak A^{\bar{\Sigma}}_{\w}$  for (\ref{RDS}) with the fixed $\sigma_0=(f_0,g_0)$ and for the family of (\ref{RDS})
 with $\sigma=(f, g)\in \bar{\Sigma}$, respectively,  exist, $\Aw$ and $\Aw^{\bar{\Sigma}}$ are the maximal
invariant and maximal quasi-invariant set w.r.t. the closure $\bar{\Dc}=\Dc_{\bar{\Sigma}}$ of the corresponding evolutionary system $\Dc$ and
\begin{equation*}
\begin{split}
\Aw&=\Aw^{\bar{\Sigma}}  =\left\{u(0): u \in \Dc_{\bar{\Sigma}}((-\infty, \infty))\right\}= \left\{ u(0): u\in \bigcup_{\sigma\in\bar{\Sigma}}\Dc_\sigma ((-\infty, \infty))\right\}, \\ \mathfrak A_{\w}&=\mathfrak A^{\bar{\Sigma}}_{\w}=\Pi_+\bigcup_{\sigma\in\bar{\Sigma}}\Dc_\sigma ((-\infty, \infty)), \\
\Aw&=\Aw^{\bar{\Sigma}}=\mathfrak A^{\bar{\Sigma}}_{\w}(t)=\left\{u(t): u\in \mathfrak A^{\bar{\Sigma}}_{\w}\right\},\quad \forall\, t\geq 0,
\end{split}
\end{equation*}
where $\mathcal E_{\sigma} ((-\infty,\infty))$ is nonempty for any $\sigma\in \bar{\Sigma}$.
Moreover,  $\mathfrak A_{\w}=\mathfrak A^{\bar{\Sigma}}_{\w}$ satisfies the finite weak uniform tracking property  for $\Dc_{\bar{\Sigma}}$ and  is weakly equicontinuous on $[0,\infty)$.
\end{theorem}
\begin{proof}
The proof is analogous to that of Theorem \ref{t:Aw02DNSEw}. By the assumptions, the associated evolutionary systems $\Dc$ and $\Dc_{\bar{\Sigma}}$ are unique. It follows from Lemmas  \ref{l:compactRDS} and \ref{l:closedRDSclosed} that $\Dc$ satisfies A1 and that  $\Dc_{\bar{\Sigma}}$  is closed, respectively. Then, we obtain the theorem by Theorem \ref{t:A0Aw}.
\end{proof}

\begin{theorem}\label{t:AsRDSStrcuniq}
Furthermore, if  $g_0$ is  normal in $L^2_{\mathrm{loc}}(\mathbb{R}; V')$, then the two weak uniform global attractors $\Aw$ and $\Aw^{\bar{\Sigma}}$  are strongly compact strong uniform global attractors  $\As$ and $\As^{\bar{\Sigma}}$, and  the two weak trajectory attractors $\mathfrak A_{\w}$ and $\mathfrak A^{\bar{\Sigma}}_{\w}$ are strongly
compact strong trajectory attractors $\mathfrak A_{\s}$ and $\mathfrak A^{\bar{\Sigma}}_{\s}$, respectively. Moreover,
\begin{itemize}
\item[1.] $\mathfrak A_{\s}=\mathfrak A^{\bar{\Sigma}}_{\s}=\Pi_+\bigcup_{\sigma\in\bar{\Sigma}}\Dc_\sigma  ((-\infty, \infty))$ satisfies the finite strong
uniform tracking property, i.e.,  for any $\epsilon >0$ and $T>0$, there exist $t_0$ and a finite subset $P_T^f\subset \mathfrak A_{\s}|_{[0,T]}$, such
that for any $t^*>t_0$, every trajectory $u \in \Dc_{\bar{\Sigma}}([0,\infty))$
satisfies
$$|u(t)-v(t-t^*)|< \epsilon,\quad\forall t\in [t^*,t^*+T],$$
for some  $T$-time length piece  $v\in P_T^f$.
\item[2.]
 $\mathfrak A_{\s}=\mathfrak A^{\bar{\Sigma}}_{\s}=\Pi_+\bigcup_{\sigma\in\bar{\Sigma}}\Dc_\sigma ((-\infty, \infty))$   is strongly equicontinuous on $[0,\infty)$, i.e.,
\[
\left|v(t_1)-v(t_2)\right|\leq \theta\left(|t_1-t_2|\right), \quad\forall\, t_1,t_2\ge 0, \  \forall v\in \mathfrak A_{\s},
\]
where $\theta(l) $ is a positive function tending
to $0$ as $l\rightarrow 0^+$.
\end{itemize}
\end{theorem}

Part of this theorem recovers the corresponding results in \cite{Lu07}. Here we obtain in addition the existence of  the strongly compact strong trajectory attractor and its corollaries.

\begin{proof}We continue the proof of Theorem \ref{t:AwRDSStrcuniq}. We only need to obtain the asymptotical compactness of the evolutionary system $\Dc_{\bar{\Sigma}}$. By Theorem \ref{t:AsRDSclosed}, the strongly compact strong uniform global attractor $\As^{\bar{\Sigma}}$ for $\Dc_{\bar{\Sigma}}$ exists when  $g_0$ is  normal in $L^2_{\mathrm{loc}}(\mathbb{R}; V')$. Hence, all conclusions follow from Theorem \ref{t:saofESac}.
\end{proof}

\subsection{On nonlinearity}
\label{ss:nonln}
In this subsection, we first collect  properties of some kinds of interaction functions, with which \eqref{RDS} are studied in some previous literature (see \cite{CV02, Lu07, CL09}). Then, we construct several counter examples that do not satisfy part of restrictions on the nonlinearity in these literature, especially do not belong to these classes of interaction functions. However, our Theorems \ref{t:structureofattractorwRDS} and \ref{t:sturctureofattractorsRDS} are still applicable for  \eqref{RDS} with interaction terms being such examples. As indicated in Open Problem \ref{op:attractorRDS}, it is not yet known how to obtain the same results by previous frameworks.

\begin{definition}\label{d:trc}\cite{CV95, CV02} Let $\Xi$ be a topological space of functions defined on $\mathbb R$. A function $\phi(s)\in \Xi$  is said to be  translation compact in $\Xi$ if the closure $$\overline{\{\phi(\cdot+h):h\in \mathbb R\}}^\Xi,$$ is compact  in $\Xi$.
\end{definition}
Denote by $ C^{\mathrm{p.u.}}_{\mathrm{tr.c.}}(\mathbb R;\mathcal M)$,
$C^{\mathrm{p.p.}}_{\mathrm{tr.c.}}(\mathbb R;\mathcal M)$ and $
C_{\mathrm{tr. c.}} (\mathbb R;\mathcal M)$ the classes of
translation  compact functions in $ C^{\mathrm{p.u.}}(\mathbb
R;\mathcal M)$, $ C^{\mathrm{p.p.}}(\mathbb R;\mathcal M)$ and $
 C (\mathbb R;\mathcal M)$, respectively.  We have the following characterizations and relationships of these spaces.
 \begin{theorem} \cite{CV02}\label{t:Charaoftrc} $\varphi(s)\in  C_{\mathrm{tr.c.}}(\mathbb R;\mathcal
M)$ if and only if for any $R>0$, $\varphi(v,s)$ is bounded in
$Q(R)=\{(v,s):\|v\|_{\mathbb R^N}\leq R, s\in \mathbb R\}$, and
\begin{equation*}\label{ucofphisv}
\begin{split}
\|\varphi(v_1,s_1)-\varphi(v_2,s_2)\|_{\mathbb R^N}\leq
\theta\left(\|v_1-v_2\|_{\mathbb R^N}+|s_1-s_2|,R\right), \\  \forall\,
(v_1,s_1), (v_2, s_2) \in Q(R),
\end{split}
\end{equation*} where $\theta(l,R) $ is a positive function tending
to $0$ as $l\rightarrow 0^+$.
\end{theorem}
\begin{theorem} \cite{Lu07}\label{t:Charaoftrwc} $\varphi(s)\in  C^{\mathrm{p.u.}}_{\mathrm{tr.c.}}(\mathbb R;\mathcal
M)$ if and only if $\varphi(s)\in C^{\mathrm{p.p.}}_{\mathrm{tr.c.}}(\mathbb R;\mathcal M)$ and one of the following
holds.
\begin{itemize}
\item[1.] $\{\varphi(s): s\in \mathbb R\}$ is precompact in
$\mathcal M$.

\item[2.] For any $R>0$, $\varphi(v,s)$ is bounded in
$Q(R)=\{(v,s):\|v\|_{\mathbb R^N}\leq R, s\in \mathbb R\}$, and
\begin{equation}\label{ucofphiv}
\|\varphi(v_1,s)-\varphi(v_2,s)\|_{\mathbb R^N}\leq
\theta\left(\|v_1-v_2\|_{\mathbb R^N},R\right), \ \forall\,
(v_1,s), (v_2, s) \in Q(R),
\end{equation} where $\theta(l,R) $ is a positive function tending
to $0$ as $l\rightarrow 0^+$.
\end{itemize}
\end{theorem}
By  Arzel\`{a}-Ascoli
compactness criterion, the property 1 or 2 implies that the family $\{\varphi(\cdot,s):s\in \mathbb R\}$ is equicontinuous  on any ball $\left\{v\in \mathbb R^N: \|v\|_{\mathbb R^N}\leq R\right\}$ with radius $R>0$.

\begin{theorem}\cite{Lu07}\label{proofinteractionfunction} Let  $\varphi\in C^{\mathrm{p.u.}}_{\mathrm{tr.c.}}(\mathbb R;\mathcal M)$ and  denote by $\mathcal
H^{\mathrm{p.u.}}(\varphi)$ the closure of its translation family $\{\varphi(\cdot+h): h\in \mathbb
R\}$ in  $ C^{\mathrm{p.u.}}(\mathbb R;\mathcal M)$.
Then
\begin{itemize}
\item[1.] $\mathcal H^{\mathrm{p.u.}}(\varphi)\subset C^{\mathrm{p.u.}}_{\mathrm{tr.c.}}
(\mathbb R;\mathcal M)$, moreover, $\mathcal H^{\mathrm{p.u.}}(\varphi_1)\subset
\mathcal H^{\mathrm{p.u.}}(\varphi)$,  $\forall \varphi_1\in \mathcal
H^{\mathrm{p.u.}}(\varphi)$.

\item[2.]   Any $\varphi_1\in \mathcal H^{\mathrm{p.u.}}(\varphi)$ satisfies
(\ref{ucofphiv}) for the same $\theta(l,R)$.

\item[3.]  The translation group $\{T(t)\}$ is invariant and
continuous on $\mathcal H^{\mathrm{p.u.}}(\varphi)$ in the topology of $
C^{\mathrm{p.u.}} (\mathbb R;\mathcal M)$.
\end{itemize}
\end{theorem}

\begin{theorem}\label{t:squenconvergence}
Let  $\varphi\in C^{\mathrm{p.u.}}_{\mathrm{tr.c.}}(\mathbb R;\mathcal M)$ and  denote by $\mathcal H^{\mathrm{p.u.}}(\varphi)$  the closure of its translation family $\{\varphi(\cdot+h): h\in \mathbb R\}$ in  $ C^{\mathrm{p.u.}}(\mathbb R;\mathcal M)$. Let $\{\varphi_n\}\subset \mathcal
H^{\mathrm{p.u.}}(\varphi)$ be such that $\varphi_n\rightarrow \varphi_0$ in $ C^{\mathrm{p.p.}}(\mathbb R;\mathcal M)$. Then
\begin{itemize}
\item[1.] $\varphi_0\in \mathcal H^{\mathrm{p.u.}}(\varphi)$.
\item[2.] $\varphi_n\rightarrow \varphi_0$ in $ C^{\mathrm{p.u.}}(\mathbb R;\mathcal M)$.
\end{itemize}
\end{theorem}
\begin{proof}Note that, the compactness of $\mathcal H^{\mathrm{p.u.}}(\varphi)$  in $ C^{\mathrm{p.u.}}(\mathbb R;\mathcal M)$ implies its compactness in $ C^{\mathrm{p.p.}}(\mathbb R;\mathcal M)$. As indicated in the footnote \ref{fn:metrizable}, $\mathcal H^{\mathrm{p.u.}}(\varphi)$ endowed with the topology of $C^{\mathrm{p.p.}}(\mathbb R;\mathcal M)$ is metrizable, and then the
corresponding metric space is complete. Hence, $\varphi_0\in \mathcal H^{\mathrm{p.u.}}(\varphi)$.

 Due to Theorem \ref{t:Charaoftrwc}, $\varphi$ satisfies (\ref{ucofphiv}). Now fix $s\in \mathbb R$, $R>0$, and denote by
$B_R:=\{v\in \mathbb R^N: \| v\|_{\mathbb R^N}\leq R\}$. For any
$\epsilon
>0$, there exists $\delta>0$ such that
\begin{align}\label{ucofphi0}\| \varphi_0 (v_1,s)-
\varphi_0 (v_2,s)\|_{\mathbb R^N}\leq \epsilon/3 ,\quad
\forall\,v_1,v_2\in B_R, \|v_1-v_2\|_{\mathbb R^N}\leq \delta,
\end{align} and  the function $\theta(l,R) $ in (\ref{ucofphiv})
is smaller or equal to  $\varepsilon/3 $ for $l\leq\delta$. For such
$\delta $, there exists a finite number of points $\{v_1,\ldots,v_m\}\subset B_R$ being a $\delta$-net of $B_R$, that is, for any $v\in
B_R$, there exists some $v_j$ satisfying $\| v-v_j\|_{\mathbb
R^N}\leq \delta$. We have that, for any $v\in B_R$,
\begin{align*} \|\varphi_n(v,s)- \varphi_0
(v,s)\|_{\mathbb R^N} \leq& \|\varphi_n(v,s)-
\varphi_n(v_j,s)\|_{\mathbb R^N}\\&+\|\varphi_n(v_j,s)-
\varphi_0 (v_j,s)\|_{\mathbb R^N}\\&+\| \varphi_0 (v_j,s)-
\varphi_0 (v,s)\|_{\mathbb R^N}
\end{align*}
Thanks to the conclusion 2 of Theorem \ref{proofinteractionfunction}, it follows that
\begin{align}\label{ucofphin}
\|\varphi_n(v,s)-
\varphi_n(v_j,s)\|_{\mathbb R^N}
\leq \epsilon/3,\quad \forall n\in \mathbb N.
\end{align}
Combining (\ref{ucofphi0}),  (\ref{ucofphin}) and the fact that $\varphi_n\rightarrow \varphi_0$ in $ C^{\mathrm{p.p.}}(\mathbb R;\mathcal M)$, we obtain
\begin{align*} \|\varphi_n(v,s)- \varphi_0
(v,s)\|_{\mathbb R^N} \leq
\epsilon/3+\epsilon/3+\epsilon/3=\epsilon, \quad \forall v\in B_R,
\end{align*}
for sufficiently large $n$. This means that, for every $s\in \mathbb R$, $R>0$,
\[\max_{\|v\|_{\mathbb R^N}\leq
R}\|\varphi_n(v,s)-\varphi_0(v,s)\|_{\mathbb R^N} \rightarrow
0, \quad \mbox { as } n\rightarrow \infty,\] which is
$\varphi_n(v,s)\rightarrow \varphi_0(v,s)$  in $C^{\mathrm{p.u.}}(\mathbb
R;\mathcal M)$.

We complete the proof.
\end{proof}
\begin{corollary}\label{t:sequencomp}
Let  $\varphi\in C^{\mathrm{p.u.}}_{\mathrm{tr.c.}}(\mathbb R;\mathcal M)$ and  denote by  $\mathcal H^{\mathrm{p.u.}}(\varphi)$  the closure of  its translation family  $\{\varphi(\cdot+h): h\in \mathbb R\}$ in  $ C^{\mathrm{p.u.}}(\mathbb R;\mathcal M)$. Then, $\mathcal H^{\mathrm{p.u.}}(\varphi)$ is sequentially compact in  $ C^{\mathrm{p.u.}}(\mathbb R;\mathcal M)$.
\end{corollary}
\begin{proof}
Take a sequence $\{\varphi_n\}\subset \mathcal
H^{\mathrm{p.u.}}(\varphi)$.  $\{\varphi_n\}$ is compact in the metrizable topological  space $(\mathcal H^{\mathrm{p.u.}}(\varphi),  C^{\mathrm{p.p.}}(\mathbb R;\mathcal M))$. Hence, there is a subsequence $\{\varphi_{n_j}\}$,  such that $\varphi_{n_j}\rightarrow \varphi_0$ in $ C^{\mathrm{p.p.}}(\mathbb R;\mathcal M)$ with $\varphi_0\in C^{\mathrm{p.p.}}(\mathbb R;\mathcal M)$ as $n_j\rightarrow \infty$. By Theorem \ref{t:squenconvergence}, $\varphi_0\in \mathcal
H^{\mathrm{p.u.}}(\varphi)$ and the convergence  is actually valid in $ C^{\mathrm{p.u.}}(\mathbb R;\mathcal M)$.
\end{proof}
Let  $C_{\mathrm{b}} (\mathbb R ;\mathcal M)$ be the space of
bounded continuous functions with values in $\mathcal M $ and endowed with
the uniform convergence topology on $\mathbb R$. We have the
following relationships.
\begin{theorem}\cite{Lu07}\label{t:raltionofspaces} $C_{\mathrm{tr.c.}} (\mathbb
R;\mathcal M)\subset  C^{\mathrm{p.u.}}_{\mathrm{tr.c.}}(\mathbb
R;\mathcal M)\subset  C^{\mathrm{p.p.}}_{\mathrm{tr. c.}} (\mathbb
R;\mathcal M)\subset C_{\mathrm{b}} (\mathbb R;\mathcal M)$ with all
inclusions being proper and the former three sets being closed in
$ C_{\mathrm{b}} (\mathbb R;\mathcal M)$.
\end{theorem}

Now, we construct several examples in $ C(\mathbb R\times\mathbb R;
\mathbb R)$  that satisfy conditions (\ref{c:dissipationoffRDS})-(\ref{c:boundednessoffRDS}), but are not translation compact in $C^{\mathrm{p.u.}}(\mathbb R; C(\mathbb R;
\mathbb R))$ and do not satisfy (\ref{c:uniquenessoffRDS}) as well. Note again that, Theorems \ref{t:structureofattractorwRDS} and \ref{t:sturctureofattractorsRDS} are applicable for  (\ref{RDS}) with such interaction functions.

In the following examples, let $$T:=\max\{0,t\}, \quad t\in \mathbb R,$$
and let $p\ge 2$ and  $\mathcal M =C(\mathbb R, \mathbb R)$.

\smallskip\smallskip
 \noindent {\bf Example I.}
 \[f\left(v,t\right)=
 \begin{cases}
 \vspace{1.2\jot}|v|^{p-2}v, &\text{if } v\leq 0, \\
 \vspace{1.2\jot}- (1+T)v, &\text{if } 0\leq v\leq\frac{1}{1+T},\\
 \left|v-\frac{1}{1+T}\right|^{p-1}-1, & \text{if } v>\frac{1}{1+T}.
 \end{cases}
\]
Note that, the family $\{f(\cdot,t):t\in \mathbb R\}$  is not equicontinuous on $[0,1]$, which means that $f(v,t)$ does not satisfy (\ref{ucofphiv}).  The fact that $$\partial_v f\left(\frac{1}{2(1+n)}, n\right)=-(1+n)\rightarrow -\infty, \quad \mbox{ as } n\rightarrow \infty,$$implies that condition  (\ref{c:uniquenessoffRDS}) does not hold for $f(v,t)$. Moreover,  the pointwise limit function of $f(\cdot,t) $,  as $t\rightarrow +\infty$, is a discontinuous function,
\[f_{\infty}\left(v\right)=
 \begin{cases}
 \vspace{1.2\jot}|v|^{p-2}v, &\text{if } v\leq 0, \\
 \left|v\right|^{p-1}-1, & \text{if } v>0.
 \end{cases}
\]
Hence, $f(v,t)$ does not even belong to $C^{\mathrm{p.p.}}_{\mathrm{tr.c.}}(\mathbb R;\mathcal M)$. In fact, for any sequence $\{ f(\cdot,\cdot+t_n)\}$ with $t_n\rightarrow +\infty$, the pointwise  limit is $f_\infty$.

\smallskip\smallskip\smallskip
 \noindent {\bf Example II.}
 \[f\left(v,t\right)=
 \begin{cases}
 \vspace{1.2\jot}|v+2\pi|^{p-2}(v+2\pi), &\text{if } v\leq -2\pi, \\
 \vspace{1.2\jot} \rho(v)\sin(1+T)v, &\text{if } -2\pi<v<2\pi,\\
 \left|v-2\pi\right|^{p-1}, & \text{if } v\geq 2\pi,
 \end{cases}
\]
where $\rho(\cdot)$ is a continuous function supported on $[-2\pi,2\pi] $. For instance, $\rho(\cdot)$ is an infinitely differentiable function supported on $(-2\pi,2\pi) $ and equals to 1 on $[-\pi,\pi] $.  Note again that, the family $\{f(\cdot,t): t\in \mathbb R\}$  is not equicontinuous in $[-2\pi,2\pi]$. Hence, $f(v,t)$ does not satisfy (\ref{ucofphiv}).  Moreover,   there is no a pointwise limit function of any sequence $\{f(\cdot,\cdot+t_n)\} $,  as $t_n\rightarrow +\infty$. Therefore, $f(v,t)\notin  C^{\mathrm{p.p.}}_{\mathrm{tr.c.}}(\mathbb R;\mathcal M)$. The fact that $\partial_v f\left(\cdot, \cdot\right)$ has no a uniform lower bound in $[-\pi,\pi] \times \mathbb R$  implies that $f(v,t)$ does not satisfy (\ref{c:uniquenessoffRDS}) either.

\smallskip\smallskip\smallskip
 \noindent {\bf Example III.}
 \[f\left(v,t\right)=
 \begin{cases}
 \vspace{1.2\jot}|v+2|^{p-2}(v+2), &\text{if } v\leq -2, \\
 \vspace{1.2\jot} \rho(v, T)\sin T^2, &\text{if } -2<v<2,\\
 \left|v-2\right|^{p-1}, & \text{if } v\geq 2,
 \end{cases}
\]
where $\rho(\cdot,\cdot)$ is an infinitely differentiable function that $\rho(\cdot,t)$  is supported on
$$\left(-2+\frac{1}{2(1+T)},2-\frac{1}{2(1+T)}\right) $$ for any $t\in \mathbb R$, and equals to 1 on $$\left(-2+\frac{1}{1+T},2-\frac{1}{1+T}\right) .$$ Similarly, there exists no a uniform lower bound to $\partial_v f\left(\cdot, \cdot\right)$ in $[-2,2] \times \mathbb R$ and  the family $\{f(\cdot,t): t\in \mathbb R\}$  is not equicontinuous on $[-2,2]$. Hence, $f(v,t)$ does not satisfy (\ref{c:uniquenessoffRDS}) and (\ref{ucofphiv}).
Moveover,  there is also no a pointwise limit function of any sequence $\{f(\cdot,\cdot+t_n)\} $,  as $t_n\rightarrow +\infty$. Therefore, $f(v,t)\notin  C^{\mathrm{p.p.}}_{\mathrm{tr.c.}}(\mathbb R;\mathcal M)$.

\end{document}